\documentclass[draft]{article}

\usepackage{amssymb}
\usepackage{amsmath}
\usepackage{amsthm}
\usepackage{ascmac}
\usepackage{graphics}
\usepackage{color}

\newcommand{\Ad}{\mathop{\mathrm {Ad}}\nolimits}

\newcommand{\sgn}{\mathop{\mathrm {sgn}}\nolimits}

\newcommand{\supp}{\mathop{\mathrm {supp}}\nolimits}

\newcommand{\sI}{\sqrt{-1}}

\newcommand{\cA}{\mathcal{A}}

\newcommand{\cC}{\mathcal{C}}
\newcommand{\cD}{\mathcal{D}}

\newcommand{\cF}{\mathcal{F}}

\newcommand{\cH}{\mathcal{H}}

\newcommand{\cJ}{\mathcal{J}}

\newcommand{\cM}{\mathcal{M}}

\newcommand{\cP}{\mathcal{P}}
\newcommand{\cQ}{\mathcal{Q}}

\newcommand{\cS}{\mathcal{S}}

\newcommand{\cU}{\mathcal{U}}

\newcommand{\ra}{\mathrm{a}}

\newcommand{\rk}{\mathrm{k}}

\newcommand{\ru}{\mathrm{u}}

\newcommand{\bC}{\mathbb{C}}

\newcommand{\bP}{\mathbb{P}}
\newcommand{\bQ}{\mathbb{Q}}
\newcommand{\bR}{\mathbb{R}}
\newcommand{\bZ}{\mathbb{Z}}

\newcommand{\me}{\mathbf{e}}

\newcommand{\ms}{\mathbf{s}}

\newcommand{\gD}{\mathfrak{D}}
\newcommand{\gH}{\mathfrak{H}}

\newcommand{\g }{\mathfrak{g}}

\newcommand{\gl}{\mathfrak{l}}

\newcommand{\gs}{\mathfrak{s}}

\newtheorem{thm}{Theorem}[section]
\newtheorem{lem}[thm]{Lemma}
\newtheorem{prop}[thm]{Proposition}
\newtheorem{cor}[thm]{Corollary}

\newtheorem{rem}[thm]{Remark}

\numberwithin{equation}{section}
\theoremstyle{remark}

\renewcommand{\contentsname}{Contents\\{\footnotesize\normalfont(A table
of contents should normally not be included)}}

\begin{document}

\title{Automorphic pairs of distributions on $\bR$, 
and Maass forms of real weights}

\author{Tadashi Miyazaki\footnote{Department of Mathematics, College of Liberal Arts and Sciences, Kitasato University, 1-15-1 Kitasato, Minamiku, Sagamihara, Kanagawa, 252-0373, Japan. 
E-mail: \texttt{miyaza@kitasato-u.ac.jp}}}

\maketitle

\allowdisplaybreaks 

\begin{abstract}
We give a correspondence between automorphic pairs 
of distributions on $\bR$ and Dirichlet series 
satisfying functional equations and 
some additional analytic conditions. 
Moreover, we show that the notion of automorphic pairs 
of distributions on $\bR$ can be regarded as 
a generalization of automorphic distributions on 
smooth principal series representations of 
the universal covering group of $SL(2,\bR)$. 
As an application, we prove Weil type 
converse theorems for automorphic distributions 
and Maass forms of real weights. 
\end{abstract}

\section{Introduction}
\label{sec:intro}

Our goal is to give a detailed exposition 
of the theory of automorphic pairs of distributions on $\bR$. 
Automorphic pairs of distributions on $Sym(n,\bR )$ 
were first introduced by Suzuki \cite{Suzuki_001} 
(in the name of ``distribution with automorphy''), 
and extended by Tamura \cite{Tamura_001} to 
more general prehomogeneous vector spaces. 
They proved that the Dirichlet series 
associated with automorphic pairs of distributions 
have the meromorphic continuations to the whole complex plane, 
and satisfy the functional equations. 
Their method is based on the idea in the theory of 
prehomogeneous vector spaces (namely, the combination of the summation formula  and the local functional equation), and it is powerful 
for the study of the analytic properties of Dirichlet series. 
For an exploration of the possibility of 
automorphic pairs of distributions, 
we study the simplest case precisely in this paper.  
As an application, we give 
Weil type converse theorems for automorphic distributions 
and Maass forms of real weights, which are 
generalizations of the results of the previous paper \cite{preMSSU}.

Let us recall Suzuki's result in \cite{Suzuki_001} 
for the simplest case $Sym(1,\bR )=\bR$ with slight modification. 
Let $L_i=u_{i1}+u_{i2}\bZ$ with 
$u_{i1}\in \bR$ and $u_{i2}\in \bR^\times$ for $i=1,2$.  
We call such $L_i$ a shifted lattice in $\bR$ in this paper. 
Let $\mu$ and $\nu $ be complex numbers. 
An automorphic pair 
of distributions on $\bR$ for $({L}_1,{L}_2)$ 
with the automorphic factor 
$J_{\mu,\nu }(x)=e^{-\sgn (x)\pi \sI \mu /2} |x|^{-2\nu -1}$ 
($x\in \bR^\times$) 
is a pair $(T_{\alpha_1},T_{\alpha_2})$ of distributions on $\bR$ 
defined by Fourier expansions  
\begin{align}
\label{eqn:intro_fourier}
&T_{\alpha_i}(f)=\sum_{l\in L_i}\alpha_i (l)\cF (f)(l),&
&(f\in C^\infty_0(\bR),\ i=1,2)
\end{align}
such that the coefficient $\alpha_i$ is a polynomial growth function on $L_i$ 
and the following equality holds: 
\begin{align}
\label{eqn:intro_001}
&T_{\alpha_1}(f)=T_{\alpha_2}(f_{\mu ,\nu,\infty})&
&(f\in C^\infty_0(\bR )\text{ with support in $\bR^\times$}).
\end{align}
Here $C^\infty_0(\bR )$ is the space of smooth functions on $\bR$ 
with compact support, 
$\cF$ denotes the Fourier transformation, 
and 
$f_{\mu ,\nu,\infty}$ is a function in $C^\infty_0(\bR )$ 
characterized by 
\begin{align*}
&f_{\mu ,\nu,\infty} (x)
=J_{\mu,\nu }(x)f(-1/x)&
&(x\in \bR^\times ).
\end{align*}
Let 
$\mathcal{A} ({L}_1,{L}_2;J_{\mu ,\nu})$ be 
the space of automorphic pairs of distributions on $\bR$ 
for $({L}_1,{L}_2)$ with the automorphic factor $J_{\mu ,\nu}$. 
Suzuki proved that, 
if $(T_{\alpha_1},T_{\alpha_2})\in \mathcal{A} ({L}_1,{L}_2;J_{\mu ,\nu})$, 
then Dirichlet series 
\begin{align*}
&\xi_+(\alpha_i ;s)=\sum_{0<l\in L_i}
\frac{\alpha_i (l)}{l^s},&
&\xi_-(\alpha_i ;s)=\sum_{0>l\in L_i}
\frac{\alpha_i (l)}{|l|^s}&
&(\mathrm{Re}(s)\gg 0,\ i=1,2)
\end{align*}
have the meromorphic continuations 
to the whole $s$-plane and satisfy 
the functional equation 
\begin{align}
\label{eqn:intro_002}
{\mathrm{E}}(s)\left(\begin{array}{c}
\Xi_+(\alpha_{1};s)\\
\Xi_-(\alpha_{1};s)
\end{array}
\right)
=\Sigma_\mu {\mathrm{E}}(-s-2\nu +1)
\left(\begin{array}{c}
\Xi_+(\alpha_{2};-s-2\nu +1)\\
\Xi_-(\alpha_{2};-s-2\nu +1)
\end{array}
\right),
\end{align}
where $\Xi_\pm (\alpha_i;s)=(2\pi )^{-s}\Gamma (s)\xi_\pm (\alpha_i;s)$ and  
\begin{align*}
&{\mathrm{E}} (s)=\left(\!\begin{array}{cc}
e^{\pi \sI s/2}&e^{-\pi \sI s/2}\\
e^{-\pi \sI s/2}&e^{\pi \sI s/2}
\end{array}\!\right)\!,&
&\Sigma_\mu =\left(\!\begin{array}{cc}
0&e^{\pi \sI \mu /2}\\
e^{-\pi \sI \mu /2}&0
\end{array}\!\right). 
\end{align*}

In this paper, we have three purposes related to automorphic pairs of 
distributions on $\bR$. 
The first purpose is to make Suzuki's result more precise. 
In Theorem \ref{thm:DS}, 
we give a correspondence between automorphic pairs 
of distributions on $\bR$ and Dirichlet series 
with nice properties. 
More precisely, 
for polynomial growth functions $\alpha_1$ and $\alpha_2$ 
respectively on $L_1$ and $L_2$, 
a pair $(T_{\alpha_1},T_{\alpha_2})$ 
defined by (\ref{eqn:intro_fourier}) 
is in $\mathcal{A} ({L}_1,{L}_2;J_{\mu ,\nu})$ 
if and only if the Dirichlet series 
$\xi_\pm (\alpha_1 ;s)$ and $\xi_\pm (\alpha_2 ;s)$ 
satisfy the following conditions (1) and (2):
\begin{enumerate}
\item[{\normalfont (1)}]
The functions $\Xi_\pm (\alpha_{1};s)$ have 
the meromorphic continuations to the whole $s$-plane, 
and the functional equation (\ref{eqn:intro_002}) holds. 
Moreover, 
\begin{align}
\label{eqn:intro_003}
& {\mathrm{E}}(s)\left(\begin{array}{c}
\Xi_+(\alpha_{1};s)\\
\Xi_-(\alpha_{1};s)
\end{array}
\right)+
\left(\frac{\alpha_{1}(0)}{s}1_2
-\frac{\alpha_{2}(0)}{s+2\nu -1}\Sigma_\mu \right)
\left(\begin{array}{c}
1\\
1
\end{array}
\right)
\end{align}
is entire. 
Here, for $i=1,2$, we understand $\alpha_i(0)=0$ if $0\not\in L_i$. 

\item[{\normalfont (2)}]
For any $\sigma_1,\sigma_2\in \bR$ such that $\sigma_1<\sigma_2$, 
there is $c_0\in \bR_{>0}$ such that 
\begin{align*}
&\Xi_\pm (\alpha_{1};s)=O(e^{|s|^{c_0}})&&(|s|\to \infty)&
&\text{uniformly on \ $\sigma_1\leq \mathrm{Re}(s)\leq \sigma_2$}. 
\end{align*}
\end{enumerate}
Our proof is based on Suzuki's idea in \cite{Suzuki_001}, 
that is, the combination of the summation formula (\ref{eqn:intro_001}) 
and the local functional equation. 
By a careful argument using certain test functions, 
we derive from (\ref{eqn:intro_001}) 
not only the functional equation (\ref{eqn:intro_002}) 
but also the entireness of (\ref{eqn:intro_003}) and 
the estimate in (2). 
Conversely, 
we derive (\ref{eqn:intro_001}) assuming (1) and 
(2) by using the Mellin inversion formula, 
similar to the proof of the converse theorem for 
holomorphic modular forms.  

The second purpose of this paper is to clarify 
the relation between automorphic pairs 
of distributions on $\bR$ and automorphic distributions 
on smooth principal series representations of 
the universal covering group $\widetilde{G}$ of $G=SL(2,\bR)$. 
Automorphic distributions are distributional realizations of 
automorphic forms, and those theory is developed by 
Miller and Schmid (\cite{MR2079889}, \cite{MR2036579}, \cite{MR2369497}). 
Let us explain briefly our results for the second purpose, 
which are given in \S \ref{subsec:principal series}, \S \ref{subsec:Fexp_ILL}, 
\S \ref{subsec:quasi_auto} and \S \ref{subsec:quasi_auto2}. 
Let $(\rho ,I_{\mu,\nu}^\infty)$ be 
a smooth principal series representation of $\widetilde{G}$, 
and we call a continuous functional on $I_{\mu,\nu}^\infty$ 
a distribution on $I_{\mu,\nu}^\infty$ in this paper. 
For the shifted lattice $L_i=u_{i1}+u_{i2}\bZ$ ($i=1,2$), 
we define the dual lattice $L_i^\vee$ of $L_i$ by 
$L_i^\vee =u_{i2}^{-1}\bZ$, and 
define a character $\omega_{L_i}\colon L_i^\vee \to \bC^\times $ 
by $\omega_{L_i}(t)=e^{2\pi \sI u_{i1}t}\ \ (t\in L_i^\vee)$. 
A distribution $\lambda $ on $I_{\mu,\nu}^\infty$ is said to be 
automorphic for $(L_1,L_2)$ if $\lambda$ satisfies 
\begin{align*}
&\lambda (\rho (\tilde{\ru}(t_1))F)=\omega_{L_1}(t_1)\lambda (F)&
&\lambda (\rho (\tilde{w}\tilde{\ru}(t_2)\tilde{w}^{-1})F)
=\omega_{L_2}(t_2)\lambda (F)
\end{align*}
for $t_1\in L_1^\vee $, $t_2\in L_2^\vee$ and $F\in I^\infty_{\mu,\nu}$. 
Here $\tilde{\ru}(x)$ and $\tilde{w}$ are the lifts of 
$\ru (x)=\left( \begin{smallmatrix} 1 & x \\ 0 & 1 \end{smallmatrix}\right)$ 
and $w=\left( \begin{smallmatrix} 0&-1 \\ 1&0 \end{smallmatrix}\right)$ 
to $\widetilde{G}$, respectively ({\it cf.} \S \ref{subsec:univ_cover}). 
For a distribution $\lambda$ on $I_{\mu,\nu}^\infty$, 
we define maps $T^{\lambda}$ and $T^{\lambda_\infty}$ from 
$C_0^\infty (\bR)$ to $\bC$ by 
\begin{align*}
&T^{\lambda}(f)=
\lambda (\iota_{\mu,\nu}(f)),&
&T^{\lambda_\infty}(f)=e^{\pi \sI \mu /2}
\lambda (\rho (\tilde{w})\iota_{\mu,\nu}(f))&
&(f\in C_0^\infty (\bR)). 
\end{align*}
Here $\iota_{\mu,\nu}(f)$ is the function in $I_{\mu,\nu}^\infty$ 
characterized by $\iota_{\mu,\nu}(f)(\tilde{w}\tilde{u}(-x))=f(x)$ 
($x\in \bR$). Then $\lambda \mapsto 
(T^\lambda ,T^{\lambda_\infty})$ defines an 
injective map from the space of automorphic distributions 
on $I_{\mu,\nu}^\infty$ for $(L_1,L_2)$ 
to $\mathcal{A} ({L}_1,{L}_2;J_{\mu ,\nu})$. 
Moreover, for $(T_{\alpha_1},T_{\alpha_2})\in 
\mathcal{A} ({L}_1,{L}_2;J_{\mu ,\nu})$, 
there is an automorphic distribution $\lambda$ on 
$I_{\mu,\nu}^\infty$ for $(L_1,L_2)$ such that 
$(T^\lambda ,T^{\lambda_\infty})=(T_{\alpha_1},T_{\alpha_2})$ 
if and only if the associated Dirichlet series 
$\xi_\pm (\alpha_1 ;s)$ and $\xi_\pm (\alpha_2 ;s)$ 
satisfy the following condition (3):
\begin{enumerate}
\item[{\normalfont (3)}]For $i=1,2$, 
the functions $(s-1)(s+2\nu -1)\xi_\pm (\alpha_{i};s)$ are entire 
if $0\in L_{3-i}$, and 
the functions $\xi_\pm (\alpha_{i};s)$ are entire 
if $0\not\in L_{3-i}$. 
\end{enumerate}
As an remarkable fact, 
``Mittag--Leffler'' theorem of Knopp 
\cite[Theorem 2]{Knopp_001} gives Dirichlet series 
which satisfy the conditions (1), (2) and do not satisfy the condition (3). 
Hence, we know that there are automorphic pairs of distributions on $\bR$ 
not coming from automorphic distributions on $I_{\mu,\nu}^\infty$. 
This fact is indicated by Professor Fumihiro Sato.

The third purpose of this paper is 
to give Weil type converse theorems 
for automorphic distributions and Maass forms of real weights. 
As an analogue of the converse theorem of Weil \cite{MR0207658} 
for holomorphic modular forms of level $N$, 
some mathematicians studied 
converse theorems for Maass forms of level $N$. 
Diamantis--Goldfeld \cite{MR2823866} gave a converse theorem 
characterizing linear combinations of 
metaplectic Eisenstein series, 
by considering the twists of 
Dirichlet series by Dirichlet characters 
including imprimitive characters. 
Neururer--Oliver \cite{preNO} 
gave a converse theorem 
characterizing Maass forms of weight $0$, 
and succeeded to avoid the twists by imprimitive characters. 
In the previous paper \cite{preMSSU}, 
following the approach of Diamantis-Goldfeld, 
we gave converse theorems characterizing 
automorphic distributions and Maass forms of integral 
and half-integral weights. 
In this paper, we generalize the result of \cite{preMSSU} 
to real weights. 

Let us explain our converse theorems, more precisely. 
Until the end of the introduction, we assume $\mu \in \bR$. 
Let $N$ be a positive integer, and 
let $\widetilde{\Gamma}_0(N)=\varpi^{-1}(\Gamma_0(N))$ with 
the covering map $\varpi \colon \widetilde{G}\to G$. 
Let $v$ be a multiplier system on $\Gamma_0(N)$ of weight $\mu$, 
and we denote by $\tilde{\chi}_v$ the character of 
$\widetilde{\Gamma}_0(N)$ corresponding to $v$ 
({\it cf.} \S \ref{sec:AD_discrete}). 
A distribution $\lambda$ on $I^{\infty}_{\mu,\nu}$ is said to be 
automorphic for $\widetilde{\Gamma}_0(N)$ with 
character $\tilde{\chi}_v$ if $\lambda$ satisfies 
\begin{align*}
&\lambda (\rho (\tilde{\gamma})F)
=\tilde{\chi}_v(\tilde{\gamma})\lambda (F)&
&(\tilde{\gamma}\in \widetilde{\Gamma}_0(N),\ 
F\in I^{\infty}_{\mu,\nu}). 
\end{align*}
We define shifted lattices $L$ and $\hat{L}$ by 
$L=u+\bZ$ and $\hat{L}=N^{-1}(\hat{u}+\bZ )$ with 
$0\leq u,\hat{u}<1$ determined by 
$v(\ru(1))=e^{2\pi \sI u}$ and 
$v(w\,\ru (N)w^{-1})=e^{2\pi \sI \hat{u}}$. 
Then an automorphic distribution on $I^{\infty}_{\mu,\nu}$ 
for $\widetilde{\Gamma}_0(N)$ with $\tilde{\chi}_v$ 
is automorphic for $(L,\hat{L})$. 
Our converse theorem (Theorem \ref{thm:Weil_converse} (ii)) 
gives a condition for 
automorphic distributions on $I^{\infty}_{\mu,\nu}$ for $(L,\hat{L})$ 
to be automorphic for $\widetilde{\Gamma}_0(N)$ with $\tilde{\chi}_v$ 
in terms of twists of the 
associated Dirichlet series by Dirichlet characters. 
Our strategy of the proof is slightly different from 
that in the previous paper \cite{preMSSU}. 
In \cite{preMSSU}, we investigate the relation between automorphic 
distributions and Dirichlet series directly, generalizing 
the local functional equation to the line model of 
smooth principal series representations. 
In this paper, we investigate the relation through 
the intermediary of automorphic pairs of distributions on $\bR$, 
and do not use the generalized local functional equation. 
Since the Poisson transform of an autormorphic distribution is a Maass form, 
we also obtain a converse theorem for Maass forms of level $N$.

\subsection*{Acknowledgments}
The author would like to thank Professor Fumihiro Sato 
for valuable advices and introducing 
``Mittag--Leffler'' theorem of Knopp. 
The author would also like to thank Professor Kazunari Sugiyama 
for lectures on the converse theorems for autmorphic distributions 
and Maass forms. 
This work was supported by JSPS KAKENHI Grant Number JP18K03252.

\subsection*{Notation}

We denote by $\bZ$, $\bQ$, $\bR$, and $\bC$ 
the ring of rational integers, 
the rational number field, the real number field, and 
the complex number field, respectively. 
For a subset $S$ of $\bR$ and $t\in \bR$, 
we denote by $S_{\geq t}$, $S_{>t}$, $S_{\leq t}$, and $S_{<t}$ 
the subsets of $S$ consisting of all numbers $r$ satisfying 
$r\geq t$, $r> t$, $r\leq t$, and $r< t$, respectively. 
For $x\in \bR^\times$, we set $\sgn (x)=x/|x|$. 

The real part, the imaginary part, 
the complex conjugate, and the absolute value of a complex number $z$ are 
denoted by $\mathrm{Re}(z)$, $\mathrm{Im}(z)$, $\bar{z}$, and $|z|$, 
respectively. 
For $z\in \bC^\times$, 
we denote by $\arg z$ the principal argument of $z$, namely, 
the argument of $z$ satisfying $-\pi <\arg z\leq \pi$. 
Moreover, we set 
$z^s=|z|^se^{\sI s\arg z}$ for $z\in \bC^\times$ and $s\in \bC$.

For an open subset $S$ of $\bR$, 
we denote by $C(S)$ and $C^\infty (S)$ 
the spaces of continuous functions and smooth functions on $S$, 
respectively. 
We denote by $C_{0}(S)$ the subspace of $C(S)$ 
consisting of all functions $f$ with compact support, 
and let $C^\infty_{0}(S)=C^\infty (S)\cap C_{0}(S)$.  
Moreover, $C^n(S)$ denotes the subspace of $C(S)$ consisting of 
all $n$ times continuously-differentiable functions on $S$ 
for $n\in \bZ_{\geq 0}$. 
We denote by $\cS (\bR )$ the space of 
rapidly decreasing functions on $\bR$. 
We denote by $L^1(\bR )$ 
the space of integrable functions on $\bR$ 
with respect to the usual Lebesgue measure. 
For $f\in C(S)$, we denote by $\supp (f)$ the support of $f$. 
For $f\in C^n(S)$, 
we denote by $f^{(n)}$ the $n$-th derivative of $f$. 
Here the $0$-th derivative $f^{(0)}$ is 
just the original function $f$. 
We denote simply by $f'$ and $f''$ the first derivative $f^{(1)}$ 
and the second derivative $f^{(2)}$ of $f$, respectively. 
In this paper, 
we regard $f\in C(\bR^\times )$ as an element of $C(\bR)$ 
by setting $\displaystyle f(0)=\lim_{x\to 0}f(x)$ 
if the limit $\displaystyle \lim_{x\to 0}f(x)$ exists.

For a set $S$, we denote by $S^2$ and $M_2(S)$ 
the sets of $2\times 1$- and $2\times 2$-matrices 
whose entries in $S$, respectively. 
The unit matrix of size $2$ is denoted by $1_2$.

\section{Main results}
\label{sec:auto_dist}

For the sake of readability, 
we summarize the main theorems and 
important propositions in this section. 
Most of the proofs are given in later sections.

\subsection{Periodic distributions on $\bR$}
\label{subsec:distribution}

A $\bC$-linear map $T\colon C_0^\infty (\bR )\to \bC$ is called 
a {\it distribution} on $\bR$ if and only if, for any $u>0$, 
there are $m\in \bZ_{\geq 0}$ and $c>0$ such that 
\begin{align*}
\begin{split}
&|T(f)|\leq c\sum_{i=0}^{m}
\sup_{x\in \bR}\left|f^{(i)}(x)\right|\\
&(f\in C^\infty_0(\bR)\ \text{with}\ 
\supp (f)\subset \{x\in \bR \mid |x|\leq u\}).
\end{split}
\end{align*}
We denote by $\cD'(\bR)$ the space of distributions on $\bR$. 

For $f\in L^1(\bR )$, we define 
the Fourier transform $\cF (f)$ of $f$ by 
\begin{align*}
\cF ({f})(y)=\int_{-\infty}^{\infty}f(x)
{e}^{2\pi \sI xy}\, dx&
&(y\in \bR ).
\end{align*}

Let ${L}$ be a subset of $\bR$ of the form 
${L}=u_1+u_2\bZ $ with $u_1\in \bR$ and $u_2\in \bR^\times$. 
We call such $L$ a {\it shifted lattice} in $\bR$. 
We define the dual lattice ${L}^\vee $ of $L$ by 
\[
{L}^\vee =\{t\in \bR \mid (l_1-l_2)t\in \bZ \quad (l_1,l_2\in L)\}
=u_2^{-1}\bZ, 
\]
and define a character $\omega_L\colon L^\vee \to \bC^\times $ 
by $\omega_L(t)=e^{2\pi \sI lt}\ \ (t\in L^\vee)$ with $l\in L$. 
Here the definition of $\omega_{L}$ does not depend on the choice of $l$. 
We note that $L$ is a lattice in $\bR$ if and only if $0\in L$, 
and the condition $0\in L$ implies that $\omega_L$ is the trivial character. 
We define a subspace 
$\cD'({L}^\vee \backslash \bR ;\omega_L)$ of $\cD'(\bR)$ by 
\begin{align*}
&\cD'({L}^\vee \backslash \bR ;\omega_L)
=\{T\in \cD'(\bR )\mid 
T(\ms_t(f))=\omega_L(-t)T(f)\ \ 
(t\in {L}^\vee ,\ f\in C^\infty_0(\bR))\}, 
\end{align*}
where $\ms_t(f)(x)=f(x+t)\ (x\in \bR)$. 
Here we note that $\cD'({L}^\vee \backslash \bR ;\omega_L)$ is the space of 
periodic distributions on $\bR$ if $0\in L$. 

For $r\in \bR$, 
let ${\mathfrak{M}}_r({L})$ be the space of 
functions $\alpha \colon {L} \to \bC$ satisfying 
$\alpha (l)=O(|l|^{r})\ (|l| \to \infty)$. 
Let ${\mathfrak{M}} ({L})
=\bigcup_{r\in \bR}{\mathfrak{M}}_r({L})$. 
For $\alpha \in {\mathfrak{M}}({L})$, it is convenient 
to set $\alpha (0)=0$ if $0\not\in {L}$. 
For $\alpha \in {\mathfrak{M}} (L)$, 
we define $T_\alpha \in \cD'({L}^\vee \backslash \bR ;\omega_L)$ by 
\begin{align*}
&T_\alpha (f)=
\sum_{l\in {L}}\alpha (l)\cF (f)(l)&&(f\in C^\infty_0(\bR)).
\end{align*}
\begin{prop}
\label{prop:periodic_dist}
Let $L$ be a shifted lattice in $\bR$. 
Then the $\bC$-linear map $\mathfrak{M}(L)\ni \alpha \mapsto 
T_\alpha \in \cD'({L}^\vee \backslash \bR ;\omega_L)$ is bijective. 
\end{prop}
\begin{proof}
Let $L=u_1+u_2\bZ $ with $u_1\in \bR$ and $u_2\in \bR^\times$. 
We define a bijective map 
\begin{align}
\label{eqn:pf_q_p_dist001}
\mathfrak{M}(L)\ni \alpha \mapsto 
\alpha_{(u_1,u_2)}\in \mathfrak{M}(\bZ)
\end{align}
by $\displaystyle 
\alpha_{(u_1,u_2)}(n)
=\alpha (u_1+u_2n)/u_2$\ \ $(n\in \bZ )$. 
Since $\cD'(\bZ \backslash \bR ;\omega_{\bZ})$ is 
the space of periodic distributions, 
it is known that 
\begin{align}
\label{eqn:pf_q_p_dist002}
\mathfrak{M}(\bZ)\ni \alpha \mapsto 
T_\alpha \in \cD'(\bZ \backslash \bR ;\omega_{\bZ})
\end{align}
is bijective 
(see, for example, Friedlander and Joshi \cite[\S 8.5]{Friedlander_001}). 
We define a bijective map 
\begin{align}
\label{eqn:pf_q_p_dist003}
\cD'(\bZ \backslash \bR ;\omega_{\bZ})\ni T\mapsto 
T\circ i_{(u_1,u_2)}\in \cD'(L^\vee \backslash \bR ;\omega_{L})
\end{align}
by $i_{(u_1,u_2)}(f)(x)=e^{2\pi \sI {u_1}x/{u_2}}f(x/{u}_2)$\ \  
$(x\in \bR,\ f\in C^\infty_0(\bR))$.

Since 
$\mathfrak{M}(L)\ni \alpha \mapsto 
T_\alpha \in \cD'({L}^\vee \backslash \bR ;\omega_L)$ is 
the composite of the bijective maps (\ref{eqn:pf_q_p_dist001}), 
(\ref{eqn:pf_q_p_dist002}) and (\ref{eqn:pf_q_p_dist003}), 
we obtain the assertion. 
\end{proof}

For $T\in \cD'({L}^\vee \backslash \bR ;\omega_L)$, 
there is a unique $\alpha \in {\mathfrak{M}} (L)$ such that
\begin{align}
\label{eqn:def_Fexp}
T(f)
&=\sum_{l\in {L}}\alpha (l)\cF (f)(l)
&&(f\in C^\infty_0(\bR))
\end{align}
by Proposition \ref{prop:periodic_dist}. 
We call the expression (\ref{eqn:def_Fexp}) 
the {\it Fourier expansion} of $T$.

\subsection{Automorphic pairs of distributions on $\bR$}
\label{subsec:def_auto_pair}

Let $\mu ,\nu \in \bC$. 
We define a function $J_{\mu ,\nu }$ on $\bR^\times$ by 
\begin{align*}
&J_{\mu ,\nu}(x)=e^{-\sgn (x)\pi \sI \mu /2} 
|x|^{-2\nu -1}&&(x\in \bR^\times )
\end{align*}
with $\sgn (x)=x/|x|$. 
We define a map 
$C(\bR^\times )\ni f\mapsto 
f_{\mu ,\nu,\infty}\in C(\bR^\times )$ by 
\begin{align*}
&f_{\mu ,\nu,\infty} (x)=
J_{\mu ,\nu}(x)f(-1/x)&&(x\in \bR^\times ).
\end{align*}
Here we note that the following equalities hold:
\begin{align}
\label{eqn:basic_f_infty}
&(f_{\mu ,\nu,\infty})_{\mu ,\nu,\infty} =f,&
&f_{\mu +2,\nu,\infty}=-f_{\mu ,\nu,\infty}&
&(f\in C(\bR^\times )). 
\end{align}

In this paper, if $\displaystyle \lim_{x\to 0}f(x)$ exists, 
we regard $f\in C(\bR^\times )$ as an element of $C(\bR)$ 
by setting $\displaystyle f(0)=\lim_{x\to 0}f(x)$. 
In particular, we regard $f\in C^\infty_0(\bR^\times )$ as 
an element of $C^\infty_0(\bR)$ by setting $f(0)=0$. 
Let $\mathcal{A} (J_{\mu ,\nu})$ be 
the space of pairs $(T_1,T_2)$ of distributions on $\bR$ 
such that 
\begin{align}
\label{eqn:autom_pair_def}
&T_1(f)=T_2(f_{\mu ,\nu,\infty})&&(f\in C^\infty_0(\bR^\times )). 
\end{align}
For two shifted lattices ${L}_1$ and ${L}_2$ in $\bR$, 
we define a space $\mathcal{A} ({L}_1,{L}_2;J_{\mu ,\nu})$ by 
\begin{align*}
\mathcal{A} ({L}_1,{L}_2;J_{\mu ,\nu})
&=\mathcal{A}(J_{\mu ,\nu})\cap 
\bigl(\cD'({L}_1^\vee \backslash \bR;\omega_{L_1})\times 
\cD'({L}_2^\vee \backslash \bR;\omega_{L_2})\bigr)\\
&=\{(T_{\alpha_1},T_{\alpha_2})\in \mathcal{A}(J_{\mu ,\nu})\mid 
\alpha_1\in \mathfrak{M}(L_1),\ \alpha_2\in \mathfrak{M}(L_2)\}.
\end{align*}
Here the second expression follows from 
Proposition \ref{prop:periodic_dist}. 
We call an element of $\mathcal{A} ({L}_1,{L}_2;J_{\mu ,\nu})$ 
an {\it automorphic pair of distributions} on $\bR$ 
for $({L}_1,{L}_2)$ with the automorphic factor $J_{\mu ,\nu}$. 

Because of (\ref{eqn:basic_f_infty}), 
we note that 
$(T_2,T_1)\in \mathcal{A} (J_{\mu ,\nu})$ and 
$(T_1,-T_2)\in \mathcal{A} (J_{\mu \pm 2,\nu})$ hold for 
$(T_1,T_2)\in \mathcal{A} (J_{\mu ,\nu})$. 
Moreover, we have 
$(T_2,T_1)\in \mathcal{A} ({L}_2,{L}_1;J_{\mu ,\nu})$ and 
$(T_1,-T_2)\in \mathcal{A} ({L}_1,{L}_2;J_{\mu \pm 2,\nu})$ 
for $(T_1,T_2)\in \mathcal{A} ({L}_1,{L}_2;J_{\mu ,\nu})$.

\subsection{Dirichlet series associated with automorphic pairs}
\label{subsec:DS_AP}

Let $r\in \bR$. 
Let ${L}$ be a shifted lattice in $\bR$. 
For $\alpha \in {\mathfrak{M}}_r({L})$, 
we define Dirichlet series $\xi_+(\alpha ;s)$ 
and $\xi_-(\alpha ;s)$ by 
\begin{align}
\label{eqn:def_DS}
&\xi_+(\alpha ;s)=\sum_{0<l\in {L}}
\frac{\alpha (l)}{l^s},&
&\xi_-(\alpha ;s)=\sum_{0>l\in {L}}
\frac{\alpha (l)}{|l|^s}.
\end{align}
It is easy to see that these series converge absolutely 
and are holomorphic functions on $\mathrm{Re}(s)>r+1$. 
We set 
\begin{align}
\label{eqn:def_cDS}
&\Xi_+(\alpha ;s)=(2\pi )^{-s}\Gamma (s)\xi_+(\alpha ;s),&
&\Xi_-(\alpha ;s)=(2\pi )^{-s}\Gamma (s)\xi_-(\alpha ;s).
\end{align}
For $s\in \bC$ and $\mu \in \bC$, 
we define matrices ${\mathrm{E}} (s)$ and $\Sigma_\mu$ by 
\begin{align}
\label{eqn:def_matrix_gamma}
&{\mathrm{E}} (s)=\left(\begin{array}{cc}
e^{\pi \sI s/2}&e^{-\pi \sI s/2}\\
e^{-\pi \sI s/2}&e^{\pi \sI s/2}
\end{array}\right),\\
\label{eqn:def_matrix_sigma}
&\Sigma_\mu =\left(\begin{array}{cc}
0&e^{\pi \sI \mu /2}\\
e^{-\pi \sI \mu /2}&0
\end{array}\right).
\end{align}

Let ${L}_1$ and ${L}_2$ be two shifted lattices in $\bR$. 
Let $\alpha_1\in \mathfrak{M}(L_1)$, 
$\alpha_2\in \mathfrak{M}(L_2)$ and $\mu ,\nu \in \bC$. 
We consider the following condition 
{\normalfont [D1]} on $\Xi_\pm (\alpha_1;s)$ and $\Xi_\pm (\alpha_2;s)$: 
\begin{description}
\item[{\normalfont [D1]}]
The functions $\Xi_\pm (\alpha_{1};s)$ have 
the meromorphic continuations to the whole $s$-plane, and 
the $\bC^2$-valued function  
\begin{align}
\label{eqn:DS_entire}
& {\mathrm{E}}(s)\left(\begin{array}{c}
\Xi_+(\alpha_{1};s)\\
\Xi_-(\alpha_{1};s)
\end{array}
\right)+
\left(\frac{\alpha_{1}(0)}{s}1_2
-\frac{\alpha_{2}(0)}{s+2\nu -1}\Sigma_\mu \right)
\left(\begin{array}{c}
1\\
1
\end{array}
\right)
\end{align}
is entire. 
Here, for $j\in \{1,2\}$, we understand $\alpha_j(0)=0$ if $0\not\in L_j$. 
Moreover, $\Xi_\pm (\alpha_1;s)$ and $\Xi_\pm (\alpha_2;s)$ 
satisfy the functional equation 
\begin{align*}
{\mathrm{E}}(s)\left(\begin{array}{c}
\Xi_+(\alpha_{1};s)\\
\Xi_-(\alpha_{1};s)
\end{array}
\right)
=\Sigma_\mu {\mathrm{E}}(-s-2\nu +1)
\left(\begin{array}{c}
\Xi_+(\alpha_{2};-s-2\nu +1)\\
\Xi_-(\alpha_{2};-s-2\nu +1)
\end{array}
\right).
\end{align*}
\end{description}
When $\Xi_\pm (\alpha_1;s)$ have 
the meromorphic continuations to the whole $s$-plane, 
we consider the following conditions 
{\normalfont [D2-1]} and {\normalfont [D2-2]} on $\Xi_\pm (\alpha_1;s)$: 
\begin{description}
\item[{\normalfont [D2-1]}]
For any $\sigma_1,\sigma_2\in \bR$ such that $\sigma_1<\sigma_2$, 
the functions $\Xi_\pm (\alpha_{1};s)$ satisfy 
\begin{align*}
&\Xi_{\pm}(\alpha_{1};s)=
O\left({e}^{-\pi |\mathrm{Im}(s)|/2
+\sqrt{|\mathrm{Im}(s)|}}\right)&&(|s|\to \infty )
\end{align*}
uniformly on \ $\sigma_1\leq \mathrm{Re}(s)\leq \sigma_2$.

\item[{\normalfont [D2-2]}]
For any $\sigma_1,\sigma_2\in \bR$ such that $\sigma_1<\sigma_2$, 
there is $c_0\in \bR_{>0}$ such that 
\begin{align*}
&\Xi_\pm (\alpha_{1};s)=O(e^{|s|^{c_0}})&&(|s|\to \infty)&
&\text{uniformly on \ $\sigma_1\leq \mathrm{Re}(s)\leq \sigma_2$}. 
\end{align*}

\end{description}
Here we note that the condition {\normalfont [D2-1]} is stronger 
than {\normalfont [D2-2]}.

\begin{thm}
\label{thm:DS}
Let ${L}_1$ and ${L}_2$ be two shifted lattices in $\bR$. 
Let $\mu ,\nu \in \bC$. \\[1mm]
(i) Let $(T_{\alpha_1},T_{\alpha_2})\in 
\cA (L_1,L_2;J_{\mu ,\nu})$ with 
$\alpha_1\in \mathfrak{M}(L_1)$ and $\alpha_2\in \mathfrak{M}(L_2)$. 
Then $\Xi_\pm (\alpha_1;s)$ and $\Xi_\pm (\alpha_2;s)$ satisfy 
the conditions {\normalfont [D1]} and {\normalfont [D2-1]}. \\[1mm]
(ii) Let $\alpha_1\in \mathfrak{M}(L_1)$ and $\alpha_2\in \mathfrak{M}(L_2)$ 
such that $\Xi_\pm (\alpha_1;s)$ and $\Xi_\pm (\alpha_2;s)$ satisfy 
the conditions {\normalfont [D1]} and {\normalfont [D2-2]}. 
Then $(T_{\alpha_1},T_{\alpha_2})\in \cA (L_1,L_2;J_{\mu ,\nu})$. 
\end{thm}

A proof of Theorem \ref{thm:DS} is given in \S \ref{sec:DS}. 
The statement (i) is proved in \S \ref{subsec:DS_FE} and 
\S \ref{subsec:est_DS}. 
The statement (ii) is proved in \S \ref{subsec:converse}.

\subsection{The universal covering group of $SL(2,\bR)$}
\label{subsec:univ_cover}

Let $G=SL(2,\bR)$, and we fix 
an Iwasawa decomposition $G=UAK$, where 
$U=\{\ru (x)\mid x\in \bR\}$, $A=\{\ra (y)\mid y\in \bR_{>0}\}$ 
and $K=\{\rk (\theta )\mid \theta \in \bR\}$ with 
\begin{align*}
&\ru (x)=\left(\begin{array}{cc}
1&x\\0&1\end{array}\right),&
&\ra (y)
=\left(\begin{array}{cc}
\sqrt{y}&0\\0&1/\sqrt{y}\end{array}\right),&
&\rk (\theta )
=\left(\begin{array}{cc}
\cos \theta &\sin \theta \\
-\sin \theta &\cos \theta \end{array}\right).&
\end{align*}

Let $\gH$ be the upper half plane 
$\{z=x+\sI y \mid x\in \bR,\ y\in \bR_{>0}\}$. 
We set 
\begin{align*}
&gz=\frac{az+b}{cz+d},&
&J(g,z)=cz+d&
&\left(\ g=\left(\begin{array}{cc}a&b\\ c&d
\end{array}\right)\in G,\ z\in \gH \ \right).
\end{align*}
Since $K=\{g\in G\mid g\sI =\sI\}$ and 
$\ru (x)\ra (y)\sI =z$ \ $(z=x+\sI y\in \gH )$, 
we have $G/K\simeq \cH $ via the map induced from 
$G\ni g\mapsto g\sI \in \gH$. It is easy to see that 
the following cocycle condition holds:
\begin{align*}
&J(g_1g_2,z)=J(g_1,g_2z)J(g_2,z)&
&(g_1,g_2\in G,\ z\in \gH). 
\end{align*}

The universal covering group of $G$ is constructed in 
\cite[\S 2.2]{Bruggeman_001}. 
For convenience, we introduce a slightly modified 
construction here. 
Let \[
\widetilde{G}=
\bigl\{(g,\theta )\in G\times \bR
\mid J(g,\sI )e^{\sI \theta }=|J(g,\sI )|\ \bigr\}, 
\]
and regard $\widetilde{G}$ as a group by the operation 
\begin{align}
\label{eqn:op_tilde_G}
(g_1,\theta_1)\, (g_2,\theta_2)=
\left(g_1g_2,\ \theta_1+\theta_2-
\arg \frac{J(g_1,g_2\sI )}{J(g_1,\sI )} \right)
\end{align}
for $(g_1,\theta_1),(g_2,\theta_2)\in \widetilde{G}$. 
Then $\widetilde{G}$ is the universal covering group of 
$G$ with the covering map 
$\varpi \colon \widetilde{G}\ni (g,\theta )\mapsto g\in G$. 
We fix a section $G\ni g\mapsto {}^s\!g\in \widetilde{G}$ by 
${}^{s}\!g=(g,-\arg J(g,\sI ))$. 

We define subgroups $\widetilde{U}$, $\widetilde{A}$ and 
$\widetilde{K}$ of $\widetilde{G}$ by 
\begin{align*}
&\widetilde{U}=\{\tilde{\ru}(x)\mid x\in \bR\},&
&\widetilde{A}=\{\tilde{\ra}(y)\mid y\in \bR_{>0}\},&
&\widetilde{K}=\{\tilde{\rk}(\theta )\mid \theta \in \bR\}
\end{align*}
with 
$\tilde{\ru}(x)=(\ru (x),\, 0)$, 
$\tilde{\ra}(y)=(\ra (y),\, 0)$ and 
$\tilde{\rk }(\theta )=(\rk (\theta ),\, \theta )$. 
Then we have an Iwasawa decomposition 
$\widetilde{G}=\widetilde{U}\widetilde{A}\widetilde{K}$. 
More precisely, for $\tilde{g}=(g,\theta )\in \widetilde{G}$, 
we have 
\begin{align}
\label{eqn:Iwasawa_g}
\tilde{g}&=
\tilde{\ru}(\mathrm{Re}(g\sI ))
\tilde{\ra}(\mathrm{Im}(g\sI ))
\tilde{\rk}(\theta).
\end{align}
The center of $\widetilde{G}$ is given by 
$\widetilde{M}=\{\tilde{\rk}(m\pi )\mid m\in \bZ \}$. 
Let $\widetilde{P}=\widetilde{U}\widetilde{A}\widetilde{M}$ 
and 
\begin{align*}
&\tilde{w}=\tilde{\rk}(-\pi /2)
=(w,\, -\pi /2),&
&w=\left(\begin{array}{cc}
0&-1\\1&0\end{array}\right).
\end{align*}
Then we have a Bruhat decomposition 
$\widetilde{G}=\widetilde{P}\tilde{w}\widetilde{U}\sqcup \widetilde{P}$. 
More precisely, 
for $\tilde{g}=(g,\theta )\in \widetilde{G}$ with 
$g=\left(\begin{array}{cc}
a&b\\
c&d
\end{array}\right)\in G$, 
we have 
\begin{align}
\label{eqn:Bruhat_g}
\tilde{g}&=\left\{\begin{array}{ll}
\displaystyle 
\tilde{\ru}(a/c)\tilde{\ra}(1/c^2)
\tilde{\rk}\!\bigl(\theta +\arg ({J(g,\sI)}/{c})\bigr)
\tilde{w}\,\tilde{\ru}(d/c)
&\text{ if }\ c\neq 0,\\[1mm]
\displaystyle 
\tilde{\ru}(b/d)\tilde{\ra}(a/d)
\tilde{\rk}(\theta)
&\text{ if }\ c=0.
\end{array}
\right.
\end{align}
Here we note that 
$\theta +\arg ({J(g,\sI)}/{c})\in \pi \bZ$ if $c\neq 0$, 
and $\theta \in \pi \bZ$ if $c=0$. 
For $x\in \bR$, we set 
\begin{align*}
&\tilde{\overline{\ru}}(x)
=\tilde{w}\tilde{\ru}(-x)\tilde{w}^{-1}
=(\overline{\ru}(x),\, -\arg (1+\sI x)),&
&\overline{\ru}(x)=\left(\begin{array}{cc}
1&0\\x&1\end{array}\right).
\end{align*}

We denote by $\g$ the Lie algebra $\gs \gl (2,\bR )$ of $G=SL(2,\bR)$ and 
by $[\cdot ,\cdot ]$ the bracket product on $\g$. 
Let $\{H,E_+,E_-\}$ be the standard basis of $\g$ defined by 
\begin{align*}
&H=\left(\begin{array}{cc}1&0\\0&-1\end{array}\right),&
&E_+=\left(\begin{array}{cc}0&1\\0&0\end{array}\right),&
&E_-=\left(\begin{array}{cc}0&0\\1&0\end{array}\right).
\end{align*}
We denote by $\g_\bC$ the complexification 
$\gs \gl (2,\bC)$ of $\g$, 
and by $\cU(\g_\bC)$ the universal enveloping algebra of $\g_\bC$. 
The Lie algebra $\g$ can be regarded as the Lie algebra of $\widetilde{G}$. 
We denote by $\widetilde{\exp}$ the exponential map 
from $\g$ to $\widetilde{G}$.

\subsection{Distributions on principal series representations}
\label{subsec:principal series}

We denote by $C (\widetilde{G})$ and $C^\infty (\widetilde{G})$ 
the spaces of continuous functions and smooth functions 
on $\widetilde{G}$, respectively. 
We define the action $\rho$ of $\widetilde{G}$ on $C(\widetilde{G})$ by 
\begin{align*}
&(\rho (\tilde{g})F)(\tilde{h})=F(\tilde{h}\tilde{g})& 
&(\tilde{g},\tilde{h}\in \widetilde{G},\ F\in C(\widetilde{G})). 
\end{align*}
We define a seminorm $|\cdot |_K$ on $C(\widetilde{G})$ by 
\begin{align*}
&|F|_K=\sup_{k\in K}|F({\mathstrut}^s\!k)|
=\sup_{-\pi \leq \theta <\pi }|F(\tilde{\rk}(\theta))|&
&(F\in C(\widetilde{G})).
\end{align*}
We define the action $\rho$ of $\g$ on $C^\infty (\widetilde{G})$ by 
\begin{align*}
&(\rho(X)F)(\tilde{g})
=\frac{d}{dt}\bigg|_{t=0}F(\tilde{g}\,\widetilde{\exp}(tX))&
&(X\in \g ,\ \tilde{g}\in \widetilde{G},\ F\in C^\infty (\widetilde{G})), 
\end{align*}
and extend this action to $\cU(\g_\bC)$ in the usual way. 
For $X \in \cU (\g_\bC )$, 
we define a seminorm $\cQ_{X}$ on $C^\infty (\widetilde{G})$ by 
$\cQ_{X}(F)=|\rho (X )F|_K$ $(F\in C^\infty (\widetilde{G}))$.

Let $\mu ,\nu \in \bC$. 
Let $I_{\mu,\nu }$ be 
the subspace of $C(\widetilde{G})$ consisting of 
all functions $F$ such that 
\begin{align*}
&F(\tilde{\ru}(x)\tilde{\ra}(y)\tilde{\rk}(m\pi )\tilde{g})
=e^{\pi m \sI \mu }y^{\nu +1/2}F(\tilde{g})&
&(x\in \bR ,\ y\in \bR_{>0},\ m\in \bZ ,\ \tilde{g}\in \widetilde{G}).
\end{align*}

By the decomposition (\ref{eqn:Iwasawa_g}), 
for $\tilde{g}=(g,\theta )\in \widetilde{G}$ and $F\in I_{\mu,\nu }$, 
we have 
\begin{align}
\label{eqn:ps_red_section}
F(\tilde{g})
=e^{\sI \mu (\theta +\arg J(g,\sI))}
\mathrm{Im}(g\sI )^{\nu +1/2}
F(\tilde{\rk}(-\arg J(g,\sI))).
\end{align}
Since $\tilde{\rk}(-\arg J(g,\sI))\in 
{\mathstrut}^s\!K$, 
we know that $F=0$ 
if and only if $|F|_K=0$ for $F\in I_{\mu,\nu }$. 
Hence $|\cdot |_K$ is a norm on $I_{\mu,\nu }$. 
We give $I_{\mu,\nu }$ the topology induced by the norm $|\cdot |_K$. 
Then it is easy to see that $I_{\mu,\nu }$ is complete.

We set $I^\infty_{\mu,\nu }=
I_{\mu,\nu }\cap C^\infty (\widetilde{G})$, and 
give $I_{\mu,\nu }^\infty $ the topology induced by 
the seminorms $\cQ_{X}$ $(X\in U(\g_\bC ))$. 
Then we know that $I_{\mu,\nu }^\infty $ is a Fr\'echet space 
by the standard argument (quite similar to 
the proof of \cite[Lemma 1.6.4 (1)]{Wallach_001}). 
We call $(\rho ,I^\infty_{\mu,\nu })$ 
a {\it smooth principal series representation} of $\widetilde{G}$. 

We define an involution $I_{\mu,\nu }^\infty \ni F
\mapsto F_\infty \in I_{\mu,\nu }^\infty$ by 
$F_\infty =e^{\pi \sI \mu /2}
\rho (\tilde{w})F$. 
Here we note that, for $F\in I_{\mu,\nu }^\infty$, 
the substitution $\mu\to \mu+2$ causes the substitution 
$F_\infty \to -F_\infty$ although 
$I_{\mu,\nu }^\infty =I_{\mu +2,\nu }^\infty$ as a $\widetilde{G}$-module.

According to the decomposition (\ref{eqn:Bruhat_g}), 
we define an injective $\bC$-linear map 
$\iota_{\mu ,\nu}\colon C^\infty_0(\bR )
\to I_{\mu,\nu }^\infty$ by 
\begin{align}
\label{eqn:embedding1}
&\iota_{\mu,\nu }(f)(\tilde{g})=
\left\{\begin{array}{ll}
\displaystyle 
e^{\sI \mu \left(\theta +\arg \left(J(g,\sI)/c\right)\right)}
|c|^{-2\nu -1}f\!\left(-d/c\right)&\text{ if }\ c\neq 0,\\[2mm]
\displaystyle 
0&\text{ if }\ c=0
\end{array}\right.
\end{align}
for $f\in C^\infty_0(\bR )$ and 
$\tilde{g}=(g,\theta )\in \widetilde{G}$ 
with $g=\left(\begin{array}{cc}
a&b\\
c&d
\end{array}\right)\in G$. 
By the definition of $I^\infty_{\mu,\nu}$ and 
the density of $\widetilde{P}\tilde{w}\widetilde{U}$ in $\widetilde{G}$, 
for $f\in C^\infty_0(\bR)$, 
we know that $\iota_{\mu,\nu}(f)$ is the function 
in $I^\infty_{\mu,\nu}$ characterized by the equality 
\begin{align}
\label{eqn:iota_characterization}
&\iota_{\mu,\nu }(f)(\tilde{w}\tilde{\ru}(-x))=f(x)&
&(x\in \bR ).
\end{align}
That is, for $F\in I^\infty_{\mu,\nu}$ and $f\in C^\infty_0(\bR)$, 
we have 
\begin{align}
\label{eqn:iota_iff}
&\text{$F=\iota_{\mu,\nu}(f)$\quad  if and only if \quad 
$F(\tilde{w}\tilde{\ru}(-x))=f(x)$ ($x\in \bR$)}.
\end{align} 
For $f\in C^\infty_0(\bR)$ and $x\in \bR$, we have 
\begin{align}
\label{eqn:pre_infty_change}
(\iota_{\mu,\nu }(f))_\infty (\tilde{w}\tilde{\ru}(-x))
=e^{\pi \sI \mu /2}
\iota_{\mu,\nu }(f)(\tilde{\overline{\ru}}(x)\tilde{\rk}(-\pi))
&=f_{\mu ,\nu ,\infty}(x). 
\end{align}
Hence, by the characterization (\ref{eqn:iota_iff}), we have  
\begin{align}
\label{eqn:infty_change}
&\iota_{\mu,\nu }(f_{\mu ,\nu ,\infty})
=(\iota_{\mu,\nu }(f))_\infty &
&(f\in C^\infty_0(\bR^\times )). 
\end{align}

\begin{lem}
\label{lem:NA_act_iota_infty}
Let $\mu,\nu \in \bC$, $t\in \bR$ and $y\in \bR_{>0}$. \\
(i) For $f\in C_0^\infty (\bR)$, 
the following equalities hold: 
\begin{align*}
&\rho (\tilde{\ru} (t))\iota_{\mu,\nu}(f)=\iota_{\mu,\nu}(\ms_{-t}(f)),&
&\rho (\tilde{\ra} (y))\iota_{\mu,\nu}(f)
=y^{-\nu -1/2}\iota_{\mu,\nu}(f_{[y^{-1}]}),
\end{align*}
where $f_{[y^{-1}]}(x)=f(y^{-1}x)$ $(x\in \bR)$. \\[1mm]
(ii) For $F\in I_{\mu,\nu}^\infty$, 
the following equalities hold: 
\begin{align*}
&(\rho (\tilde{\overline{\ru}} (t))F)_\infty 
=\rho (\tilde{\ru} (-t))F_\infty ,&
&(\rho (\tilde{\ru} (t))F)_\infty 
=\rho (\tilde{\overline{\ru}} (-t))F_\infty ,\\
&(\rho (\tilde{\ra} (y))F)_\infty 
=\rho (\tilde{\ra} (1/y))F_\infty .
\end{align*}
\end{lem}
\begin{proof}
The statement (i) follows from the characterization (\ref{eqn:iota_iff}) and 
\begin{align*}
&(\rho (\tilde{\ru} (t))\iota_{\mu,\nu}(f))(\tilde{w}\tilde{\ru}(-x))
=\iota_{\mu,\nu}(f)(\tilde{w}\tilde{\ru}(-x+t))=f(x-t)=\ms_{-t}(f)(x),\\
&(\rho (\tilde{\ra} (y))\iota_{\mu,\nu}(f))(\tilde{w}\tilde{\ru}(-x))
=\iota_{\mu,\nu}(f)(\tilde{w}\tilde{\ru}(-x)\tilde{\ra}(y))\\
&\phantom{====}
=\iota_{\mu,\nu}(f)(\tilde{\ra}(1/y)\tilde{w}\tilde{\ru}(-x/y))
=y^{-\nu -1/2}f(x/y)=y^{-\nu -1/2}f_{[y^{-1}]}(x)
\end{align*}
for $f\in C^\infty_0(\bR)$ and $x\in \bR$. 
The statement (ii) follows from the equalities 
$\tilde{w}\tilde{\overline{\ru}} (t)=\tilde{\ru}(-t)\tilde{w}$, 
$\tilde{w}\tilde{\ru}(t)=\tilde{\overline{\ru}}(-t)\tilde{w}$ and 
$\tilde{w}\tilde{\ra}(y)=\tilde{\ra} (1/y)\tilde{w}$. 
\end{proof}

In this paper, we call a continuous $\bC$-linear map 
$\lambda \colon I^\infty_{\mu,\nu}\to \bC$ a {\it distribution} 
on $I^\infty_{\mu,\nu}$, and denote by $I^{-\infty}_{\mu,\nu}$ 
the space of distributions on $I^\infty_{\mu,\nu}$. 
Since the topology of $I^\infty_{\mu,\nu }$ is induced by 
the seminorms $\cQ_{X}$ $(X\in U(\g_\bC ))$, 
we note that a $\bC$-linear map $\lambda \colon I^\infty_{\mu,\nu}\to \bC$ 
is continuous if and only if 
there exist $c>0$, $m\in \bZ_{>0}$ and 
$X_1,X_2, \cdots ,X_m\in \cU (\g_{\bC})$ such that 
\begin{align}
\label{eqn:cdn_conti_ps}
&|\lambda (F)|
\leq c\sum_{i=1}^m
\cQ_{X_i}(F)&&(F\in I^\infty_{\mu,\nu}).
\end{align}
For $\lambda \in I^{-\infty}_{\mu,\nu}$, 
we define $\bC$-linear maps 
$T^\lambda \colon C^\infty_0(\bR)\to \bC$ and 
$\lambda_\infty \colon I^\infty_{\mu,\nu}\to \bC$ by 
\begin{align*}
&T^\lambda (f)=\lambda (\iota_{\mu,\nu}(f))\quad 
(f\in C^\infty_0(\bR)),&
&\lambda_\infty (F)=\lambda (F_\infty)\quad 
(F\in I^\infty_{\mu,\nu}).
\end{align*}

For $F\in I^\infty_{\mu,\nu}$, 
we call a pair $(f_1,f_2)$ a {\it partition} of $F$ if and only if 
$f_1$ and $f_2$ are functions in $C_0^\infty (\bR )$ such that 
$\iota_{\mu,\nu }(f_1)+(\iota_{\mu,\nu }(f_2))_\infty =F$. 
It is easy to see that 
any function $F$ in $I_{\mu,\nu}^\infty$ 
has a partition (\textit{cf}. Lemma \ref{lem:partition_exist}), 
although a partition is not unique. 
For $(T_1,T_2)\in \cA (J_{\mu,\nu})$ and 
$F\in I^\infty_{\mu,\nu}$, we set 
\begin{align}
\label{eqn:pair_to_fctl}
\Lambda (T_1,T_2)(F)=T_1(f_1)+T_2(f_2),
\end{align}
where $(f_1,f_2)$ is a partition of $F$. 
The following proposition is proved in \S \ref{subsec:pf_PFprop}.

\begin{prop}
\label{prop:pair_to_functionl}
Let $\mu,\nu \in \bC$. \\[1mm]
(i) $T^\lambda \in \cD'(\bR)$, 
$\lambda_\infty \in I^{-\infty}_{\mu,\nu}$ and 
$(\lambda_\infty)_\infty =\lambda$ for 
$\lambda \in I^{-\infty}_{\mu,\nu}$. \\[1mm]
(ii) The definition (\ref{eqn:pair_to_fctl}) 
does not depend on the choice of a partition $(f_1,f_2)$ of $F$. 
Moreover, $\Lambda (T_1,T_2)\in I^{-\infty}_{\mu,\nu}$ 
for $(T_1,T_2)\in \cA (J_{\mu,\nu})$. \\[1mm]
(iii) The $\bC$-linear map 
$I^{-\infty}_{\mu,\nu}\ni \lambda \mapsto 
(T^\lambda ,T^{\lambda_\infty} )\in \cA (J_{\mu,\nu})$ 
is bijective, and its inverse map is given 
by $(T_1,T_2)\mapsto \Lambda (T_1,T_2)$. 
\end{prop}

Let $L$ be a shifted lattice in $\bR$. 
We define subspaces $(I^{-\infty}_{\mu ,\nu})_{L}$ and 
$(I^{-\infty}_{\mu ,\nu})_{L}^{\rm quasi}$ of 
$I^{-\infty}_{\mu ,\nu}$ by 
\begin{align*}
&(I^{-\infty}_{\mu ,\nu})_{L}=\{\lambda \in I^{-\infty}_{\mu ,\nu}
\mid \lambda (\rho (\tilde{\ru}(t))F)=\omega_L(t)\lambda (F)\quad 
(t\in L^\vee ,\, F\in I^\infty_{\mu,\nu})\},\\ 
&(I^{-\infty}_{\mu ,\nu})_{L}^{\rm quasi}
=\{\lambda \in I^{-\infty}_{\mu ,\nu}
\mid T^\lambda \in \cD'({L}^\vee \backslash \bR ;\omega_L)\}.
\end{align*}
By Lemma \ref{lem:NA_act_iota_infty} (i) and 
Proposition \ref{prop:pair_to_functionl} (i), we have 
$T^{\lambda}\in \cD'({L}^\vee \backslash \bR ;\omega_L)$ 
for $\lambda \in (I^{-\infty}_{\mu ,\nu})_{L}$. 
Hence, $(I^{-\infty}_{\mu ,\nu})_{L}$ is a subspace of 
$(I^{-\infty}_{\mu ,\nu})_{L}^{\rm quasi}$.

\subsection{The Jacquet integrals and the Fourier expansions}
\label{subsec:Fexp_ILL}

In this subsection, we introduce the Jacquet integral, 
and the Fourier expansions in terms of them. 
Let $\mu \in \bC$, $y\in \bR$ and $f\in C_0^\infty (\bR )$. 
It is easy to see that $f_{\mu,\nu,\infty}$ is integrable 
and $\cF(f_{\mu,\nu,\infty})(y)$ is a holomorphic function 
of $\nu$ on $\mathrm{Re}(\nu)>0$. 
For $m\in \bZ_{\geq 0}$, we denote by $\delta^{(m)}$ 
the $m$-th derivative of the Dirac delta distribution, 
that is, $\delta^{(m)}(f)=(-1)^mf^{(m)}(0)$ 
($f\in C_0^\infty (\bR)$). 

\begin{lem}
\label{lem:pre_twist_Ftrans}
Let $\mu \in \bC$, $y\in \bR^\times$ and $f\in C_0^\infty (\bR )$. 
\vspace{0.5mm}

\noindent (i) As a function of $\nu$, $\cF(f_{\mu,\nu,\infty})(y)$ has 
the holomorphic continuation to the whole $\nu$-plane. \vspace{0.5mm}

\noindent (ii) As a function of $\nu$, $\cF(f_{\mu,\nu,\infty})(0)$ has 
the meromorphic continuation to the whole $\nu$-plane. 
Moreover, 
\begin{align*}
&\cF(f_{\mu,\nu ,\infty})(0)
-\sum_{i=0}^{n}
(\sI)^i\frac{2\cos \bigl(\frac{\pi (i+\mu )}{2}\bigr)}
{i!(2\nu +i)}\delta^{(i)}(f)
\end{align*} 
is holomorphic on $\mathrm{Re}(\nu)>-(n+1)/2$ 
for any $n\in \bZ_{\geq 0}$. 
\end{lem}

A proof of Lemma \ref{lem:pre_twist_Ftrans} is given 
in \S \ref{subsec:AC_fourier}. 
For $\mu,\nu \in \bC$, $y\in \bR$ and $f\in C_0^\infty (\bR )$, 
we define the twisted Fourier transform 
$\cF_{\mu,\nu ,\infty}(f)(y)$ of $f$ by 
\begin{align*}
&\cF_{\mu,\nu ,\infty}(f)(y)
=\left\{\!\begin{array}{l}
\cF(f_{\mu,\nu ,\infty})(y)\hspace{3.3cm}
\text{if $y\neq 0$ or $-2\nu \not\in \bZ_{\geq 0}$},\\[3mm]
\displaystyle 
\lim_{s \to -n/2}
\left(\cF(f_{\mu,s,\infty})(0)
-(\sI)^n\frac{2\cos \bigl(\frac{\pi (n+\mu )}{2}\bigr)}
{n!(2s +n)}\delta^{(n)}(f)\right)\\[4mm]
\hspace{2.4cm}
\text{if $y=0$ and $-2\nu =n$ with some $n\in \bZ_{\geq 0}$}
\end{array}
\right.
\end{align*}
with the meromorphic continuation 
in Lemma \ref{lem:pre_twist_Ftrans}. 
The following lemma is proved in \S \ref{subsec:AC_fourier}.

\begin{lem}
\label{lem:twist_fourier_0}
Let $\mu,\nu \in \bC$ and $y\in \bR$. 
A map 
$C_0^\infty(\bR)\ni f\mapsto \cF_{\mu,\nu,\infty}(f)(y)\in \bC$ 
is a distribution on $\bR$. 
Moreover, we have 
\begin{align}
\label{eqn:fourier_infty}
&\cF_{\mu,\nu,\infty}(f)(y)=\cF (f_{\mu,\nu,\infty})(y)&
&(f\in C_0^\infty (\bR^\times )). 
\end{align}
\end{lem}

Let $\mu ,\nu \in \bC$ and $l\in \bR$. 
When $\mathrm{Re}(\nu)>0$, 
we define the Jacquet integral 
$\cJ_{l}\in I^{-\infty}_{\mu,\nu}$ by 
\begin{align*}
&\cJ_{l}(F)=\int_{-\infty}^\infty F(\tilde{w}\tilde{\ru}(-x))
e^{2\pi \sI lx}dx&&(F\in I^\infty_{\mu,\nu}).
\end{align*}
By (\ref{eqn:iota_characterization}) and (\ref{eqn:pre_infty_change}), we have 
\begin{align}
\label{eqn:jacquet_ext}
&\cJ_{l}(F)=\cF (f_1)(l)+\cF_{\mu,\nu,\infty}(f_2)(l)&
&(F\in I^\infty_{\mu,\nu}), 
\end{align}
where $(f_1,f_2)$ is a partition of $F$. 
By Lemma \ref{lem:twist_fourier_0}, we note that 
a pair of $f\mapsto \cF (f)(l)$ and 
$f\mapsto \cF_{\mu,\nu,\infty}(f)(l)$ is 
in $\cA (J_{\mu,\nu})$. Hence, 
we can extend the definition of 
$\cJ_l\in I^{-\infty}_{\mu,\nu}$ to general $\nu \in \bC$ 
by the expression (\ref{eqn:jacquet_ext}) and 
Proposition \ref{prop:pair_to_functionl}. 
Here this extension of the Jacquet integral $\cJ_l$ 
is essentially same as the extension of the Fourier transformation 
in \cite[\S 2,1 and \S 2.3]{preMSSU} except for 
the case of $l=0$ and $-2\nu \in \bZ_{\geq 0}$.

For $m\in \bZ_{\geq 0}$, 
we define $\delta^{(m)}_{\infty} \in I^{-\infty}_{\mu,\nu}$ by 
\begin{align}
\label{eqn:def_delta_infty}
&\delta^{(m)}_{\infty}(F)
=(\rho (E_+)^{m}F_\infty)(\tilde{w})&
&(F\in I_{\mu,\nu}^\infty ). 
\end{align}
We denote by ${\mathfrak{N}}(\bZ_{\geq 0})$ the space of 
functions $\beta \colon \bZ_{\geq 0} \to \bC$ 
such that  $\beta (m)=0$ for all but finitely many $m\in \bZ_{\geq 0}$. 
For $\alpha \in {\mathfrak{M}} (L)$ and 
$\beta \in {\mathfrak{N}}(\bZ_{\geq 0})$, 
we define a map $\lambda_{\alpha,\beta}\colon 
I_{\mu,\nu}^\infty \to \bC$  by 
\begin{align*}
&\lambda_{\alpha,\beta}(F)=
\sum_{l\in {L}}\alpha (l)\cJ_l (F)
+\sum_{m=0}^\infty \beta (m)\delta^{(m)}_{\infty}(F)
&&(F\in I^\infty_{\mu,\nu}).
\end{align*}

\begin{prop}
\label{prop:QAD_Fourier}
Let $L$ be a shifted lattice in $\bR$. Let $\mu,\nu \in \bC$. 
Then, for  $\alpha \in {\mathfrak{M}} (L)$ and 
$\beta \in {\mathfrak{N}}(\bZ_{\geq 0})$, 
a map $\lambda_{\alpha,\beta}$ is a distribution in 
$(I^{-\infty}_{\mu ,\nu})_{L}^{\rm quasi}$ such that 
$T^{\lambda_{\alpha,\beta}}=T_\alpha$. 
Moreover, the $\bC$-linear map 
\begin{align}
\label{eqn:QAD_Fexp_map}
\mathfrak{M}(L)\times {\mathfrak{N}}(\bZ_{\geq 0})\ni 
(\alpha,\beta) \mapsto \lambda_{\alpha,\beta}
\in (I^{-\infty}_{\mu ,\nu})_{L}^{\rm quasi}
\end{align}
is bijective. 
\end{prop}

A proof of Proposition \ref{prop:QAD_Fourier} is given 
in \S \ref{subsec:Jacquet_integral}. 
By Proposition \ref{prop:QAD_Fourier}, 
for a distribution $\lambda $ in $(I^{-\infty}_{\mu ,\nu})_{L}^{\rm quasi}$, 
there is a unique $(\alpha,\beta)\in 
\mathfrak{M}(L)\times {\mathfrak{N}}(\bZ_{\geq 0})$ such that
\begin{align}
\label{eqn:def_Fexp_ps}
&\lambda (F)=
\sum_{l\in {L}}\alpha (l)\cJ_l (F)
+\sum_{m=0}^\infty \beta (m)\delta^{(m)}_{\infty}(F)
&&(F\in I^\infty_{\mu,\nu}).
\end{align}
We call the expression (\ref{eqn:def_Fexp_ps}) 
the {\it Fourier expansion} of $\lambda$. 

For a subset $S$ of $\bZ_{\geq 0}$, 
we denote by ${\mathfrak{N}}(S)$ a subspace of 
${\mathfrak{N}}(\bZ_{\geq 0})$ consisting of all functions $\beta$ 
such that $\beta (m)=0$ unless $m\in S$. 
Then, for $\alpha \in {\mathfrak{M}} (L)$ and $\beta \in {\mathfrak{N}}(S)$, 
the Fourier expansion of $\lambda_{\alpha,\beta}$ is given as follows: 
\begin{align*}
&\lambda_{\alpha,\beta}(F)=
\sum_{l\in {L}}\alpha (l)\cJ_l (F)
+\sum_{m\in S} \beta (m)\delta^{(m)}_{\infty}(F)
&&(F\in I^\infty_{\mu,\nu}). 
\end{align*}
Let $\mathfrak{M}(L)_{\mu,\nu}^0$ be the subset 
of $\mathfrak{M}(L)$ consisting of all functions $\alpha$ 
satisfying 
\[
\text{( $\alpha (0)=0$ if $0\in L$, $-2\nu \in \bZ_{>0}$ 
and $\mu -2\nu -1\not\in 2\bZ$ )}.
\]
We define 
a subset $S_{\nu}(L)$ of $\bZ_{\geq 0}$ by 
\begin{align*}
&S_{\nu}(L)=\!\left\{\!\begin{array}{ll}
\emptyset &\text{if\, $0\not\in L$},\\
\{0,-2\nu \}\cap \bZ_{\geq 0}&\text{if\, $0\in L$}.
\end{array}\!\right.
\end{align*}
The following proposition is proved in \S \ref{subsec:Jacquet_integral}. 

\begin{prop}
\label{prop:Fexp_functional2}
Let $L$ be a shifted lattice in $\bR$. 
Let $\mu ,\nu \in \bC$. Let 
$\alpha \in {\mathfrak{M}} (L)$ and 
$\beta \in {\mathfrak{N}}(\bZ_{\geq 0})$. Then 
we have 
$\lambda_{\alpha ,\beta} \in (I^{-\infty}_{\mu ,\nu})_{L}$
if and only if $\alpha \in {\mathfrak{M}} (L)^0_{\mu,\nu}$ and 
$\beta \in {\mathfrak{N}}(S_{\nu}(L))$.
\end{prop}

\subsection{Automorphic distributions for shifted lattices}
\label{subsec:quasi_auto}

Let $L_1$ and $L_2$ be two shifted lattices in $\bR$. 
Let $\mu,\nu \in \bC$. 
We define subspaces $(I^{-\infty}_{\mu ,\nu})_{L_1,L_2}$ and 
$(I^{-\infty}_{\mu ,\nu})_{L_1,L_2}^{\rm quasi}$ of 
$I^{-\infty}_{\mu ,\nu}$ by 
\begin{align*}
&(I^{-\infty}_{\mu ,\nu})_{L_1,L_2}
=\{\lambda \mid \lambda \in (I^{-\infty}_{\mu ,\nu})_{L_1},\ 
\lambda_\infty \in (I^{-\infty}_{\mu ,\nu})_{L_2}\}\\
&\hspace{7mm}
=\left\{\lambda \in I^{-\infty}_{\mu ,\nu}
\left| \!\begin{array}{ll}
\lambda (\rho (\tilde{\ru}(t_1))F)=\omega_{L_1}(t_1)\lambda (F)&
(t_1\in L_1^\vee ,\ F\in I^\infty_{\mu,\nu}),\\[1mm] 
\lambda (\rho (\tilde{\overline{\ru}}(t_2))F)=\omega_{L_2}(-t_2)\lambda (F)&
(t_2\in L_2^\vee ,\ F\in I^\infty_{\mu,\nu}) 
\end{array}
\!\right.\!\right\},\!\\[1mm]
&(I^{-\infty}_{\mu ,\nu})_{L_1,L_2}^{\rm quasi}
=\{\lambda \in I^{-\infty}_{\mu ,\nu} \mid 
(T^\lambda ,T^{\lambda_\infty})\in \cA (L_1,L_2;J_{\mu,\nu})\}. 
\end{align*}
Here the second expression of $(I^{-\infty}_{\mu ,\nu})_{L_1,L_2}$ 
is given by Lemma \ref{lem:NA_act_iota_infty} (ii). 
For $\lambda \in (I^{-\infty}_{\mu ,\nu})_{L_1,L_2}$, we have 
$(T^\lambda ,T^{\lambda_\infty})\in \cA (L_1,L_2;J_{\mu,\nu})$ 
by Lemma \ref{lem:NA_act_iota_infty} and 
Proposition \ref{prop:pair_to_functionl}. 
Hence, $(I^{-\infty}_{\mu ,\nu})_{L_1,L_2}$ is a subspace of 
$(I^{-\infty}_{\mu ,\nu})_{L_1,L_2}^{\rm quasi}$. 
We call a distribution in $(I^{-\infty}_{\mu ,\nu})_{L_1,L_2}$ 
(resp. $(I^{-\infty}_{\mu ,\nu})_{L_1,L_2}^{\rm quasi}$) 
an {\it automorphic distribution} 
(resp. a {\it quasi-automorphic distribution}) on $I_{\mu,\nu}^\infty$ 
for $(L_1,L_2)$. 
By Proposition \ref{prop:pair_to_functionl}, we note that 
the natural map 
\[
(I^{-\infty}_{\mu ,\nu})_{L_1,L_2}^{\rm quasi}\ni \lambda \mapsto 
(T^\lambda ,T^{\lambda_\infty} )\in \cA (L_1,L_2;J_{\mu,\nu})
\] 
is bijective. Moreover, 
we have $\lambda_\infty \in (I^{-\infty}_{\mu ,\nu})_{L_2,L_1}$ 
for $\lambda \in (I^{-\infty}_{\mu ,\nu})_{L_1,L_2}$, 
and $\lambda_\infty \in (I^{-\infty}_{\mu ,\nu})_{L_2,L_1}^{\rm quasi}$ 
for $\lambda \in (I^{-\infty}_{\mu ,\nu})_{L_1,L_2}^{\rm quasi}$.

For a shifted lattice $L$ in $\bR$, 
we take $\lambda_{\alpha,\beta}$ ($\alpha \in \mathfrak{M}(L)$, 
$\beta \in \mathfrak{N}(\bZ_{\geq 0})$), 
$\mathfrak{M}(L)_{\mu,\nu}^0$ and 
$S_{\nu}(L)$ as in \S \ref{subsec:Fexp_ILL}. 

\begin{thm}
\label{thm:QAD_Fexp}
Let $L_1$ and $L_2$ be two shifted lattices in $\bR$. 
Let $\mu,\nu \in \bC$. \\[1mm]
(i) Let $\lambda \in (I^{-\infty}_{\mu ,\nu})_{L_1,L_2}^{\rm quasi}$. 
There are unique 
$(\alpha_i,\beta_i)\in \mathfrak{M}(L_i)\times 
{\mathfrak{N}}(\bZ_{\geq 0})$ $(i=1,2)$ 
such that $\lambda =\lambda_{\alpha_1,\beta_1}$ and 
$(\lambda_{\alpha_1,\beta_1})_\infty =\lambda_{\alpha_2,\beta_2}$. 
Moreover, if $\lambda \in (I^{-\infty}_{\mu ,\nu})_{L_1,L_2}$, 
then, for $i=1,2$, 
we have $(\alpha_i,\beta_i)\in \mathfrak{M}(L_i)_{\mu,\nu}^0\times 
{\mathfrak{N}}(S_{\nu}(L_i))$ and 
\begin{align*}
&(\ \beta_i (0)=0
\text{ if $-2\nu \in \bZ_{>0}$, $\mu -2\nu -1\in 2\bZ$ 
and $(0\not\in L_{3-i}$ or $-2\nu >1)$}\ ).
\end{align*}
(ii) Let $(\alpha_i,\beta_i)\in \mathfrak{M}(L_i)\times 
{\mathfrak{N}}(\bZ_{\geq 0})$ $(i=1,2)$  such that 
$(\lambda_{\alpha_1,\beta_1})_\infty =\lambda_{\alpha_2,\beta_2}$. 
Then we have 
$\lambda_{\alpha_1,\beta_1} \in 
(I^{-\infty}_{\mu ,\nu})_{L_1,L_2}^{\rm quasi}$. 
Moreover, if $(\alpha_i,\beta_i)\in \mathfrak{M}(L_i)_{\mu,\nu}^0\times 
{\mathfrak{N}}(S_{\nu}(L_i))$ $(i=1,2)$, 
then we have 
$\lambda_{\alpha_1,\beta_1} \in (I^{-\infty}_{\mu ,\nu})_{L_1,L_2}$. 
\end{thm}
Theorem \ref{thm:QAD_Fexp} (ii) follows immediately from 
Propositions \ref{prop:QAD_Fourier} and \ref{prop:Fexp_functional2}. 
A proof of Theorem \ref{thm:QAD_Fexp} (i) is given 
in \S \ref{subsec:Jac_int_special}.

Let $(\alpha_i,\beta_i)\in 
\mathfrak{M}(L_i)\times {\mathfrak{N}}(\bZ_{\geq 0})$ 
$(i=1,2)$. 
When $\xi_\pm (\alpha_i;s)$ (or their completion 
$\Xi_\pm (\alpha_i;s)$) have 
the meromorphic continuations to the whole $s$-plane for $i=1,2$, 
we consider the following conditions 
{\normalfont [D3]} and {\normalfont [D4]} on 
$\xi_\pm (\alpha_1;s)$, $\xi_\pm (\alpha_2;s)$, $\beta_1$ and $\beta_2$: 
\begin{description}
\item[{\normalfont [D3]}]For any $m\in \bZ_{\geq 0}$ and $i=1,2$, 
the functions $\xi_\pm (\alpha_i;s)$ satisfy 
\[
\displaystyle 
\underset{s=m+1}{\mathrm{Res}}
(\xi_+(\alpha_i;s)+(-1)^m\xi_-(\alpha_i;s))=2(2\pi \sI)^m\beta_{3-i}(m). 
\]
\item[{\normalfont [D4]}]For $i=1,2$, 
the functions $(s-1)(s+2\nu -1)\xi_\pm (\alpha_{i};s)$ are entire 
if $0\in L_{3-i}$, and 
the functions $\xi_\pm (\alpha_{i};s)$ are entire 
if $0\not\in L_{3-i}$. 
\end{description}

Here we note that $\beta_1$ and $\beta_2$ are uniquely determined by 
$\alpha_2$ and $\alpha_1$, respectively, if {\normalfont [D3]} holds.  
If {\normalfont [D3]} and {\normalfont [D4]} hold, 
we have $\beta_i(m)=0$ for $i=1,2$ and $m\in \bZ_{\geq 0}$ unless 
$0\in L_{3-i}$ and $m\in \{0,-2\nu \}$. 
Let {\normalfont [D1]}, {\normalfont [D2-1]} and {\normalfont [D2-2]} be the conditions 
in \S \ref{subsec:DS_AP}. 

\begin{thm}
\label{thm:QAD_DS}
Let $L_1$ and $L_2$ be two shifted lattices in $\bR$. 
Let $\mu,\nu \in \bC$. 
Let $(\alpha_1,\beta_1)\in \mathfrak{M}(L_1)\times 
{\mathfrak{N}}(\bZ_{\geq 0})$ and 
$(\alpha_2,\beta_2)\in \mathfrak{M}(L_2)\times 
{\mathfrak{N}}(\bZ_{\geq 0})$. \\[1mm]
(i)  Assume 
$(\lambda_{\alpha_1,\beta_1})_\infty =\lambda_{\alpha_2,\beta_2}$. 
Then the conditions {\normalfont [D1]}, {\normalfont [D2-1]}  and 
{\normalfont [D3]} hold. Moreover, if $\lambda_{\alpha_1,\beta_1} \in 
(I^{-\infty}_{\mu ,\nu})_{L_1,L_2}$, then the condition 
{\normalfont [D4]} holds. \\[1mm]
(ii) Assume that 
{\normalfont [D1]}, {\normalfont [D2-2]} and {\normalfont [D3]} hold. 
Then $\lambda_{\alpha_1,\beta_1} \in 
(I^{-\infty}_{\mu ,\nu})_{L_1,L_2}^{\rm quasi}$ 
and $(\lambda_{\alpha_1,\beta_1})_\infty =\lambda_{\alpha_2,\beta_2}$. 
If {\normalfont [D4]} also holds, then $\lambda_{\alpha_1,\beta_1} \in 
(I^{-\infty}_{\mu ,\nu})_{L_1,L_2}$. 
\end{thm}

In order to prove Theorem \ref{thm:QAD_DS}, 
we use the following proposition, 
which is proved in \S \ref{subsec:pole_DS}.

\begin{prop}
\label{prop:poles_DS}
Let $L_1$ and $L_2$ be two shifted lattices in $\bR$. 
Let $\mu,\nu \in \bC$. Let 
$(\alpha_i,\beta_i)\in \mathfrak{M}(L_i)\times 
{\mathfrak{N}}(\bZ_{\geq 0})$ $(i=1,2)$ 
such that $(\lambda_{\alpha_1,\beta_1})_\infty =\lambda_{\alpha_2,\beta_2}$. 
Let $\xi_\pm (\alpha_{1};s)$ be the Dirichlet series, 
which is meromorphically continued to the whole $s$-plane 
by Theorem \ref{thm:DS} (i). 
If $-2\nu \not\in \bZ_{\geq 0}$, 
then the functions 
\begin{align*}
&\xi_\pm (\alpha_{1};s)
-\sum_{m=0}^\infty
\frac{(\pm 2\pi \sI)^{m}\beta_2(m)}{s-m-1}
-\frac{2\Gamma (2\nu)\cos \bigl(\tfrac{\pi (2\nu \mp \mu )}{2}\bigr)
\alpha_{2}(0)}
{(2\pi)^{2\nu }(s+2\nu -1)}
\end{align*}
are entire. 
If $-2\nu=n$ with some $n\in \bZ_{\geq 0}$, 
then the functions 
\begin{align*}
&\xi_\pm (\alpha_{1};s) 
-\sum_{m=0}^\infty
\frac{(\pm 2\pi \sI)^{m}\beta_2(m)}{s-m-1}\\
&+\frac{\pi (2\pi)^{n}\sin \bigl(\tfrac{\pi (n\mp \mu )}{2}\bigr)
\alpha_{2}(0)}{n!(s-n-1)}
+\frac{2(2\pi)^{n}\cos \bigl(\tfrac{\pi (n\mp \mu )}{2}\bigr)
\alpha_{2}(0)}{n!(s-n-1)^2}
\end{align*}
are entire. Here we understand $\alpha_{2}(0)=0$ if $0\not\in L_2$, 
and double signs are in the same order. 
\end{prop}

\begin{proof}[Proof of Theorem \ref{thm:QAD_DS}]
First, we will prove the statement (i). 
Assume that 
$(\lambda_{\alpha_1,\beta_1})_\infty =\lambda_{\alpha_2,\beta_2}$ 
holds. 
Then we have $(T_{\alpha_1},T_{\alpha_2})\in 
\cA (L_1,L_2;J_{\mu,\nu})$ 
by Propositions \ref{prop:pair_to_functionl} and 
\ref{prop:QAD_Fourier}. 
Hence, {\normalfont [D1]} and {\normalfont [D2-1]} hold by Theorem \ref{thm:DS} (i). 
Since $(\lambda_{\alpha_2,\beta_2})_\infty =\lambda_{\alpha_1,\beta_1}$ 
also holds, we know that {\normalfont [D3]} holds 
by Proposition \ref{prop:poles_DS}. 
If $\lambda_{\alpha_1,\beta_1} \in 
(I^{-\infty}_{\mu ,\nu})_{L_1,L_2}$, then 
{\normalfont [D4]} holds by 
Theorem \ref{thm:QAD_Fexp} (i) and 
Proposition \ref{prop:poles_DS}. 

Next, we will prove the statement (ii). 
Assume that {\normalfont [D1]}, {\normalfont [D2-2]} and {\normalfont [D3]} hold. 
Then we have $(T_{\alpha_1},T_{\alpha_2})\in \cA (L_1,L_2;J_{\mu,\nu})$ 
by {\normalfont [D1]}, {\normalfont [D2-2]} and Theorem \ref{thm:DS} (ii). 
Let $\lambda =\Lambda (T_{\alpha_1},T_{\alpha_2})\in 
(I^{-\infty}_{\mu ,\nu})_{L_1,L_2}^{\rm quasi}$. 
By Proposition \ref{prop:QAD_Fourier}, 
there are $\beta_1',\beta_2'\in 
\mathfrak{N}(\bZ_{\geq 0})$ such that 
$\lambda =\lambda_{\alpha_1,\beta_1'}$ 
and $(\lambda_{\alpha_1,\beta_1'})_\infty 
=\lambda_{\alpha_2,\beta_2'}$. 
By Proposition \ref{prop:poles_DS} and {\normalfont [D3]}, 
we have $\beta_1=\beta_1'$ and $\beta_2=\beta_2'$. 
Hence, $\lambda_{\alpha_1,\beta_1} \in 
(I^{-\infty}_{\mu ,\nu})_{L_1,L_2}^{\rm quasi}$ 
and $(\lambda_{\alpha_1,\beta_1})_\infty =\lambda_{\alpha_2,\beta_2}$. 
If {\normalfont [D4]} also holds, we have $\lambda_{\alpha_1,\beta_1} \in 
(I^{-\infty}_{\mu ,\nu})_{L_1,L_2}$ 
by Theorem \ref{thm:QAD_Fexp} (ii) and Proposition \ref{prop:poles_DS}. 
\end{proof}


\subsection{Knopp's result}
\label{subsec:quasi_auto2}

In this subsection, we show that there exist quasi-automorphic distributions, 
which are not automorphic distributions, by using the ``Mittag--Leffler'' 
theorem of Knopp. 
The result in this subsection is given by Professor Fumihiro Sato.

\begin{thm}[{\cite[Theorem 2]{Knopp_001}}]
\label{thm:Knopp}
Let $\kappa \in \bR_{\geq 2}$. For a 
rational function $q(s)$ satisfying $q(\kappa -s)=q(s)$, 
there is a Dirichlet series 
$\xi_q(s)=\sum_{m=1}^\infty a_q(m)m^{-s}$ satisfying the following 
conditions {\normalfont [K1]}, {\normalfont [K2]} and 
{\normalfont [K3]}:
\begin{description}
\item[{\normalfont [K1]}]There is $r>0$ such that 
$a_q(m)=O(m^r)$ $(m\to \infty )$. 
In particular, the Dirichlet series 
$\xi_q(s)$ converges absolutely on $\mathrm{Re}(s)>r+1$. 

\item[{\normalfont [K2]}]The Dirichlet series $\xi_q(s)$ has 
the meromorphic continuation to the whole $s$-plane, 
and there is a polynomial function $p(s)$ such that $p(s)\xi_q(s)$ is 
an entire function of finite genus. 

\item[{\normalfont [K3]}]The function $\Xi_q(s)=\pi^{-s}\Gamma (s)\xi_q(s)$ 
satisfies $\Xi_q(\kappa -s)=\Xi_q(s)$, 
and the function $\Xi_q(s)-q(s)$ is entire. 

\end{description}
\end{thm}

Let $i\in \{0,1\}$ and $n\in \bZ_{> 0}$. 
Let $\beta \in \mathfrak{N}(\bZ_{\geq 0})$. 
For $\kappa =2n$ and  
\[
q(s)=\sum_{m=0}^\infty \beta (m)
\left(
\frac{1}{s-n-i-2m}+\frac{1}{-s+n-i-2m}\right),
\]
we take a Dirichlet series 
$\xi_q(s)=\sum_{m=1}^\infty a_q(m)m^{-s}$ 
as in Theorem \ref{thm:Knopp}. 
Set 
\begin{align*}
&\mu =0,&
&\nu =-n+\tfrac{1}{2},& 
&L_1=L_2=2^{-1}\bZ 
\end{align*}
and 
\begin{align*}
&\alpha_1(l)=\left\{\begin{array}{ll}
a_q(2l)&\text{if}\ l>0,\\
0&\text{if}\ l=0,\\
(-1)^{n+i+1}a_q(-2l)&\text{if}\ l<0
\end{array}\right.
\quad (l\in 2^{-1}\bZ ),
\hspace{10mm}
\alpha_2=(-1)^n\alpha_1.
\end{align*}
Then {\normalfont [K1]} implies $\alpha_1\in \mathfrak{M}(L_1)$ and 
$\alpha_2\in \mathfrak{M}(L_2)$. 
Moreover, we have 
\begin{align*}
&\Xi_+(\alpha_1;s)=(-1)^{n+i+1}\Xi_-(\alpha_1 ;s)\\
&=(-1)^n\Xi_+(\alpha_2;s)=(-1)^{i+1}\Xi_-(\alpha_2;s)
=\Xi_q(s).
\end{align*}
Hence, {\normalfont [K2]} and {\normalfont [K3]} imply that 
$\alpha_1$, $\alpha_2$ satisfy {\normalfont [D1]} 
and {\normalfont [D2-2]}. 
By Theorem \ref{thm:DS} and Proposition \ref{prop:pair_to_functionl}, 
we have $(T_{\alpha_1},T_{\alpha_2})\in \cA (L_1,L_2;J_{0,\nu})$ 
and $\Lambda (T_{\alpha_1},T_{\alpha_2})\in 
(I_{0,\nu}^{-\infty})_{L_1,L_2}^{\rm quasi}$. 
If we take $\beta$ so that $q(s)$ has a pole on 
$\bC \smallsetminus (\{2n\}\cup \bZ_{\leq 1})$, 
then $\Lambda (T_{\alpha_1},T_{\alpha_2})\not\in 
(I_{0,\nu}^{-\infty})_{L_1,L_2}$ 
by Theorems \ref{thm:QAD_Fexp} (i) and \ref{thm:QAD_DS} (i). 
In particular, we know that 
the complement of $(I_{0,\nu}^{-\infty})_{L_1,L_2}$ in 
$(I_{0,\nu}^{-\infty})_{L_1,L_2}^{\rm quasi}$ is 
infinite dimensional.

\subsection{The Poisson transform}
\label{subsec:poisson}

Let $\mu,\nu \in \bC$ and $\kappa \in \mu +2\bZ$. 
Let $C^\infty (\gH )$ be the space of smooth functions on $\gH$. 
For $\phi \in C^\infty (\gH )$ and $g\in G$, we set 
\begin{align*}
&(\phi \big|_\kappa g)(z)=\left(\frac{J(g,z)}{|J(g,z)|}\right)^{-\kappa}
\phi (gz)&&(z\in \gH).
\end{align*}
We define the hyperbolic Laplacian 
$\Omega_\kappa $ of weight $\kappa$ on $\gH$ by  
\begin{align}
\label{eqn:def_Omega}
&\Omega_\kappa 
=-y^2\left(\frac{\partial^2}{\partial x^2}+
\frac{\partial^2}{\partial y^2}\right)
+\sI \kappa y\frac{\partial}{\partial x}&
&(z=x+\sI y\in \gH).
\end{align}
Let $\cM_\nu (\gH ;\kappa )$ be a subspace 
of $C^\infty (\gH )$ consisting of all functions $\phi$ 
satisfying 
\[
(\Omega_\kappa \phi )(z)
=\left(\tfrac{1}{4}-\nu^2\right)\phi (z)
\]
and the following condition 
{\normalfont [M1]}: 
\begin{description}
\item[{\normalfont [M1]}]
There are $c,r>0$ such that 
$\displaystyle 
|\phi (z)|\leq c \left(\frac{|z|^2+1}{y}\right)^{\!r}$\quad 
($z=x+\sI y\in \gH $). 
\end{description}

For a distribution $\lambda $ on $I^{\infty}_{\mu,\nu}$, 
we define the Poisson transform 
$\cP_{\nu,\kappa}(\lambda)$ of $\lambda$ by 
\begin{align}
\label{eqn:def_poisson}
&\cP_{\nu,\kappa}(\lambda)(z)
=e^{\pi \sI \kappa /2}
\lambda \bigl(\rho (\tilde{\ru}(x)\tilde{\ra}(y))F_{\nu,\kappa}\bigr)&
&(z=x+\sI y\in \gH ),
\end{align}
where $F_{\nu,\kappa}$ is an element of $I_{\mu,\nu}^\infty$ defined by 
\begin{align}
\label{eqn:def_F_nu_kappa}
&F_{\nu,\kappa }(\tilde{g})=\mathrm{Im}(g\sI)^{\nu +1/2}e^{\sI \kappa \theta}&
&(\tilde{g}=(g,\theta )\in \widetilde{G}).
\end{align}

\begin{prop}
\label{prop:poisson}
Let $\mu,\nu \in \bC$, and $\kappa \in \mu +2\bZ$. \\[1mm]
(i) Let $\lambda \in I^{-\infty}_{\mu,\nu}$. 
Then 
we have $\cP_{\nu,\kappa}(\lambda)\in \cM_\nu (\gH ;\kappa )$ and 
\begin{align}
\label{eqn:poisson_Gact}
&(\cP_{\nu,\kappa}(\lambda)\big|_\kappa g) (z)
=e^{\pi \sI \kappa /2}
\lambda \bigl(\rho ({}^s\!g\tilde{\ru}(x)\tilde{\ra}(y))F_{\nu,\kappa}\bigr)
\end{align}
for $g\in G$ and $z=x+\sI y\in \gH$. \\[1mm]
(ii) Let $\alpha \in \mathfrak{M}(L)_{\mu,\nu}^0$ and  
$\beta \in {\mathfrak{N}}(S_{\nu}(L))$. 
For $z=x+\sI y\in \gH $, we have 
\begin{align}
\label{eqn:poisson_Fexp}
\begin{split}
&\cP_{\nu,\kappa}(\lambda_{\alpha,\beta})(z)=
\sum_{0\neq l\in L}
\frac{\pi^{\nu +\frac{1}{2}}|l|^{\nu -\frac{1}{2}}\alpha (l)
}{\Gamma \bigl(\frac{2\nu +1+\sgn (l) \kappa }{2}\bigr)}\,
W_{\sgn (l)\frac{\kappa}{2},\nu}
(4\pi |l|y)e^{2\pi \sI l x}\\
&\hspace{12mm}
+cy^{-\nu +\frac{1}{2}}
+\left\{\begin{array}{ll}
(-1)^{\frac{\mu -\kappa }{2}}\beta (0)y^{\nu +\frac{1}{2}}
&\text{if $0\in L$ and $\nu \neq 0$},\\[1mm]
-2\cos \bigl(\tfrac{\pi \kappa}{2}\bigr)\alpha (0)
y^{\frac{1}{2}}\log y
&\text{if $0\in L$ and $\nu =0$},\\[1mm]
0&\text{if $0\not\in L$}.
\end{array}\right.
\end{split}
\end{align} 
Here $W_{\kappa ,\nu}(y)$ is Whittaker's function 
(\cite[Chapter 16]{Whittaker_Watson_001}), 
and $c$ is the constant determined by 
\begin{align*}
&c
=\left\{\!\begin{array}{l}
\displaystyle 
\frac{2^{1-2\nu }\pi \Gamma (2\nu )}
{\Gamma \bigl(\tfrac{2\nu +1-\kappa}{2}\bigr)
\Gamma \bigl(\tfrac{2\nu +1+\kappa }{2}\bigr)}\,
\alpha (0)
\hspace{15mm}\text{ if $0\in L$ and $-2\nu \not\in \bZ_{\geq 0}$},\\[4mm]
\displaystyle 
(-1)^{\frac{|\kappa|-n-1}{2}}\frac{2^{n}
\pi \,\bigl(\tfrac{|\kappa|+n-1}{2}\bigr)!}
{n!\,\bigl(\tfrac{|\kappa|-n-1}{2}\bigr)!}\,\alpha (0)
+(-1)^{\frac{\mu -\kappa }{2}}\beta (n)\mathbf{d}(n,\kappa )\\[4mm]
\hspace{11.6mm}
\text{if $0\in L$, $-2\nu =n$ and $|\kappa |-n+1\in 2\bZ_{> 0}$ 
with some $n\in \bZ_{\geq 0}$, }\\[2mm]
\displaystyle 
(-1)^{\frac{\mu -\kappa }{2}}\mathbf{d}(n,\kappa )\beta (n)\\
\hspace{11.6mm}
\text{if $0\in L$, $-2\nu =n$ and $\mu +n-1\not\in 2\bZ $ 
with some $n\in \bZ_{>0}$, }\\[2mm]
2\mathbf{j} (\kappa )\alpha (0)\cos \bigl(\tfrac{\pi \kappa }{2}\bigr)
+(-1)^{\frac{\mu -\kappa }{2}}\beta (0)\hspace{7mm} 
\text{if $0\in L$, $\nu =0$ and $\mu -1 \not\in 2\bZ$},\\[1mm]
0\hspace{55mm}\text{ otherwise}
\end{array}\!\right.
\end{align*}
with  
\begin{align}
\label{eqn:def_d_n_kappa}
&\mathbf{d}(n,\kappa )
=(\sI )^n\prod_{j=0}^{n-1}(\kappa +n-1-2j),\\
\label{eqn:def_j_kappa}
&\mathbf{j} (\kappa )
=\sum_{i=0}^\infty 
\left(\frac{1}{2i+1-\kappa }
+\frac{1}{2i+1+\kappa }-\frac{1}{i+1}\right)
-\log 2.
\end{align}
\end{prop}

Proofs of the statements (i) and (ii) of 
Proposition \ref{prop:poisson} are given in 
\S \ref{subsec:ps_mod_growth} and \S \ref{subsec:Jac_int_special}, 
respectively.


\subsection{Maass forms of real weights}
\label{subsec:Maass_forms}

Let $\Gamma $ be a cofinite subgroup of $SL(2,\bZ)$ 
such that $-1_2\in \Gamma$. 
Let $\kappa \in \bR$. 
We call $v$ a {\it multiplier system} 
on $\Gamma$ of weight $\kappa $ 
if and only if $v$ is a map from $\Gamma$ to $\bC^\times$ 
satisfying the conditions 
$v(-1_2)=e^{-\pi \kappa \sI }$, 
$|v(\gamma )|=1$ $(\gamma \in \Gamma )$ 
and 
\begin{align}
\label{eqn:cdn_multiplier}
&v(\gamma_1\gamma_2)=v(\gamma_1)v(\gamma_2)
\frac{J(\gamma_1,\gamma_2z)^\kappa J(\gamma_2,z)^\kappa }
{J(\gamma_1\gamma_2,z)^\kappa}&
&(\gamma_1,\gamma_2\in \Gamma,\ z\in \gH). 
\end{align}
Here we note that the right hand side 
does not depend on $z\in \gH$.

Let $\gD=\{z=x+\sI y \mid  -1/2\leq x\leq 1/2,\ 
y\geq \sqrt{1-x^2}\}$, 
which is the closure of the standard fundamental domain 
of $SL(2,\bZ)\backslash \gH$. 
Let $v$ be a multiplier system on $\Gamma$ of weight $\kappa$. 
Let $\nu \in \bC$. 
Let $\cM_{\nu }(\Gamma \backslash \gH ;v,\kappa )$ be 
a subspace of $C^\infty (\gH )$ consisting all functions $\phi$ 
satisfying 
\begin{align*}
&(\phi \big|_\kappa \gamma )(z)
=v(\gamma )\phi (z)\quad (\gamma \in \Gamma ),&
&(\Omega_\kappa \phi )(z)
=\left(\tfrac{1}{4}-\nu^2\right)\phi (z)
\end{align*}
and the following condition {\normalfont [M2]}:
\begin{description}
\item[{\normalfont [M2]}]
$\phi$ is of moderate growth at any cusp of $\Gamma$, that is, 
for any $\gamma \in SL(2,\bZ)$, there exist $c,r>0$ such that 
$(\phi \big|_\kappa \gamma )(z)\leq c\,\mathrm{Im}(z)^{r}$ $(z\in \gD )$. 
\end{description}
We call a function in 
$\cM_{\nu }(\Gamma \backslash \gH ;v,\kappa )$ 
a {\it Maass form} for $\Gamma $ of weight $\kappa$ 
with multiplier system $v$ and eigenvalue $\tfrac{1}{4}-\nu^2$. 
The following proposition is proved 
in \S \ref{subsec:mg_ftn_G}.

\begin{prop}
\label{prop:Maass_subspace}
Retain the notation. Then we have 
\[
\cM_{\nu }(\Gamma \backslash \gH ;v,\kappa )
=\left\{\phi \in 
\cM_\nu (\gH ;\kappa )\ \left| \ 
\phi \big|_\kappa \gamma 
=v(\gamma )\phi \quad (\gamma \in \Gamma )\right.\right\}.
\]
\end{prop}

In \cite[Chapter IV, \S 2]{Maass_003}, Maass introduce 
the Fourier expansion of Maass forms 
at each cusps. 
Let $\phi \in \cM_{\nu }(\Gamma \backslash \gH ;v,\kappa )$, 
and we introduce the Fourier expansion of $\phi$ at $\infty$ here. 
Let $L=R^{-1}(u+\bZ )$ be a shifted lattice 
with the real numbers $R$, $u$ determined by 
\begin{align}
\label{eqn:def_R_u}
&R=\min \{t\in \bZ_{>0}\mid \ru (t)\in \Gamma \},&
&0\leq u<1,&
&v(\ru (R))=e^{2\pi \sI u}.
\end{align}
Then $\phi$ has the following Fourier expansion at $\infty$ 
\begin{align}
\label{eqn:Fexp_Maass}
\begin{split}
\phi (z)=
&\sum_{0\neq l\in L}a(\phi ;l)
W_{\sgn (l)\frac{\kappa}{2},\nu}
(4\pi |l|y)e^{2\pi \sI l x}
+a(\phi ;0)y^{-\nu +\frac{1}{2}}\\
&+\left\{\begin{array}{ll}
b(\phi ;0)y^{\nu +\frac{1}{2}}&\text{if }\nu \neq 0,\\[1mm]
b(\phi ;0)y^{\frac{1}{2}}\log y&\text{if }\nu =0
\end{array}\right.\hspace{8mm}
(z=x+\sI y\in \gH )
\end{split}
\end{align}
with $b(\phi ;0),a(\phi ;l)\in \bC$ ($l\in L$), 
where $a(\phi ;0)=b(\phi ;0)=0$ if $0\not\in L$. 
Here $W_{\kappa ,\nu}(y)$ is Whittaker's function 
(\cite[Chapter 16]{Whittaker_Watson_001}).

\begin{rem}
Our definition of Maass forms is slightly different from that in 
Maass \cite[Chapter IV, \S 2]{Maass_003}. 
For a smooth function $\phi$ on $\gH$, 
we note that $\phi \in \cM_{\nu }(\Gamma \backslash \gH ;v,\kappa )$ 
if and only if 
$f(z,\overline{z})=\mathrm{Im}(z)^{-\nu -\frac{1}{2}}\phi (z)$ 
is an automorphic form of the type 
$\{\Gamma,\nu +\frac{1+\kappa }{2},\nu +\frac{1-\kappa }{2},v\}$ 
in the sense of Maass. 
\end{rem}

\subsection{Automorphic distributions for discrete groups}
\label{sec:AD_discrete}

Let $\Gamma $ be a cofinite subgroup of $SL(2,\bZ)$ 
such that $-1_2\in \Gamma$. 
We define a subgroup $\widetilde{\Gamma}$ of $\widetilde{G}$ by 
$\widetilde{\Gamma}=\varpi^{-1}(\Gamma)
=\{(\gamma ,\theta )\in \widetilde{G}\mid \gamma \in \Gamma\}$. 
Let $\kappa \in \bR$. 
For a multiplier system $v$ on $\Gamma$ of weight $\kappa $, 
we set 
\begin{align}
\label{eqn:def_char_mult_sys}
&\tilde{\chi}_v(\tilde{\gamma})=v(\gamma)
e^{\sI \kappa(\theta +\arg J(\gamma ,\sI))}
&(\tilde{\gamma}=(\gamma ,\theta )\in \widetilde{\Gamma}).
\end{align}
Then $\tilde{\chi}_v$ is a unitary character of $\widetilde{\Gamma}$. 
Moreover, 
$v\mapsto \tilde{\chi}_v$ defines a bijection from 
the set of multiplier systems on $\Gamma$ of weight $\kappa $ to the set 
of unitary characters $\tilde{\chi}$ of $\widetilde{\Gamma}$ satisfying 
$\tilde{\chi}({}^s\!(-1_2))=e^{-\pi \kappa \sI }$, 
since its inverse map $\tilde{\chi}\mapsto v_{\tilde{\chi}}$ is given by 
$v_{\tilde{\chi}}(\gamma)=\tilde{\chi}({}^s\!\gamma)$ 
($\gamma \in \Gamma$).

Let $\mu \in \bR$ and $\nu\in \bC$. 
Let $v$ be a multiplier system on $\Gamma$ of weight $\kappa $. 
We define a subspace 
$(I^{-\infty}_{\mu,\nu})^{\widetilde{\Gamma},\tilde{\chi}_v}$ 
of $I^{-\infty}_{\mu,\nu}$ by 
\begin{align*}
&(I^{-\infty}_{\mu,\nu})^{\widetilde{\Gamma},\tilde{\chi}_v}
=\{\lambda \in I^{-\infty}_{\mu,\nu}\mid 
\lambda (\rho (\tilde{\gamma})F)
=\tilde{\chi}_v(\tilde{\gamma})\lambda (F)\quad 
(\tilde{\gamma}\in \widetilde{\Gamma},\ 
F\in I^{\infty}_{\mu,\nu})\}. 
\end{align*}
We call a distribution in 
$(I^{-\infty}_{\mu,\nu})^{\widetilde{\Gamma},\tilde{\chi}_v}$ 
an {\it automorphic distribution} 
on $I^{\infty}_{\mu,\nu}$ for $\widetilde{\Gamma}$ 
with character $\tilde{\chi}_v$.  
For $m\in \bZ$, we have 
\begin{align}
\label{eqn:wt_compatibility}
&\rho (\tilde{\rk}(m\pi ))F=e^{\pi m \sI \mu }F\quad 
(F\in I_{\mu,\nu}^\infty),&
&\tilde{\chi}_v(\tilde{\rk}(m\pi ))=e^{\pi m \sI \kappa }. 
\end{align}
Hence, 
$(I^{-\infty}_{\mu,\nu})^{\widetilde{\Gamma},\tilde{\chi}_v}=\{0\}$ 
unless $\kappa \in \mu +2\bZ$. We assume $\kappa \in \mu +2\bZ$ 
until the end of this subsection. 
Then we have the following proposition as an immediate consequence of 
Propositions \ref{prop:poisson} (i) and \ref{prop:Maass_subspace}. 

\begin{prop}
\label{prop:dist_to_form}
Retain the notation. Let $\lambda \in 
(I^{-\infty}_{\mu,\nu})^{\widetilde{\Gamma},\tilde{\chi}_v}$. 
Then the Poisson transform $\cP_{\nu,\kappa}(\lambda )$ of $\lambda$ 
is a Maass form in $\cM_{\nu }(\Gamma \backslash \gH ;v,\kappa )$. 
\end{prop}

Let $L=R^{-1}(u+\bZ )$ be a shifted lattice in $\bR$ 
with the real numbers $R$, $u$ determined by 
(\ref{eqn:def_R_u}). 
Let $\hat{L}=\hat{R}^{-1}(\hat{u}+\bZ )$ 
be a shifted lattice in $\bR$ 
with the real numbers $\hat{R}$, $\hat{u}$ determined by 
\begin{align}
\label{eqn:def_hat_R_u}
&\hat{R}=\min \{t\in \bZ_{>0}\mid \overline{\ru} (-t)\in \Gamma \},&
&0\leq \hat{u}<1,&
&v(\overline{\ru}(-\hat{R}))=e^{2\pi \sI \hat{u}}.
\end{align}
Then it is obvious that 
$(I^{-\infty}_{\mu,\nu})^{\widetilde{\Gamma},\tilde{\chi}_v}$ 
is a subspace of $(I^{-\infty}_{\mu,\nu})_{L,\hat{L}}$. 
Hence, by Theorem \ref{thm:QAD_Fexp} (i), 
for $\lambda \in (I^{-\infty}_{\mu,\nu})^{\widetilde{\Gamma},\tilde{\chi}_v}$, 
there are unique 
$\alpha \in \mathfrak{M}(L)_{\mu,\nu}^0$, 
$\beta \in {\mathfrak{N}}(S_{\nu}(L))$, 
$\hat{\alpha}\in \mathfrak{M}(\hat{L})_{\mu,\nu}^0$ 
and $\hat{\beta}\in {\mathfrak{N}}(S_{\nu}(\hat{L}))$ 
such that $\lambda =\lambda_{\alpha ,\beta }$ and 
$\lambda_\infty =\lambda_{\hat{\alpha},\hat{\beta}}$, 
that is, 
\begin{align*}
\begin{split}
&\lambda (F)=\lambda_{\alpha ,\beta }(F)
=\sum_{l\in L }\alpha (l)\cJ_l(F)
+\sum_{m\in S_{\nu}(L)}\beta (m)\delta^{(m)}_{\infty}(F),\\
&\lambda_\infty (F)=\lambda_{\hat{\alpha},\hat{\beta}}(F)
=\sum_{l\in \hat{L}}\hat{\alpha}(l)\cJ_l(F)
+\sum_{m\in S_{\nu}(\hat{L})}\hat{\beta}(m)\delta^{(m)}_{\infty}(F)
\end{split}&
\end{align*}
for $F\in I_{\mu,\nu}^\infty$. 
We note that (\ref{eqn:poisson_Fexp}) is 
the Fourier expansion (\ref{eqn:Fexp_Maass}) at $\infty$ 
for the Maass form $\cP_{\nu,\kappa}(\lambda )
=\cP_{\nu,\kappa}(\lambda_{\alpha,\beta})$.

\subsection{Automorphic distributions for $\widetilde{\Gamma}_0(N)$}
\label{subsec:Auto_Dist_group}

Let $N\in \bZ_{>0}$. 
We define a subgroup $\Gamma_0(N)$ of $SL(2,\bZ)$ by 
\begin{align*}
\Gamma_0(N)=\left\{\left.
\left(\begin{array}{cc}
a&b\\ Nc&d
\end{array}\right)
\ \right|
a,b,c,d\in \bZ,\ ad-Nbc=1
\right\},
\end{align*}
and set $\widetilde{\Gamma}_0(N)=\varpi^{-1}(\Gamma_0(N))
=\{(\gamma ,\theta)\in \widetilde{G}\mid \gamma \in \Gamma_0(N)\}$. 
Let $\kappa \in \bR$. 
Let $v$ be a multiplier system on $\Gamma_0(N)$ of weight $\kappa$, 
and define the unitary character 
$\tilde{\chi}_v$ of $\widetilde{\Gamma}_0(N)$ by 
(\ref{eqn:def_char_mult_sys}) with $\Gamma =\Gamma_0(N)$. 
For $\Gamma =\Gamma_0(N)$ and $v$, 
we take shifted lattices $L$, $\hat{L}$ as in the previous subsection, 
that is, $L=u+\bZ$, $\hat{L}=N^{-1}(\hat{u}+\bZ )$ with 
$0\leq u,\hat{u}<1$ determined by 
$v(\ru(1))=e^{2\pi \sI u}$ and 
$v(\overline{\ru}(-N))=e^{2\pi \sI \hat{u}}$. 
Let {\normalfont [D1]}, {\normalfont [D2-1]}, {\normalfont [D2-2]}, {\normalfont [D3]} and {\normalfont [D4]} 
be the conditions for $\alpha_i\in \mathfrak{M}(L_i)$, 
$\beta_i\in \mathfrak{N}(\bZ_{\geq 0})$ ($i=1,2$), 
which are given in \S \ref{subsec:DS_AP} 
and \S \ref{subsec:quasi_auto}. 

When $N=1$, we note $v(\ru(1))=v(\overline{\ru}(-1))$ and $L=\hat{L}$. 
As a corollary of Theorem \ref{thm:QAD_DS}, 
we obtain the following.

\begin{cor}
\label{cor:SL2Z_converse}
Retain the notation and assume $N=1$. 
Let $\mu \in \kappa +2\bZ$ and $\nu \in \bC$. 
Let $(\alpha ,\beta )\in \mathfrak{M}(L)\times 
{\mathfrak{N}}(\bZ_{\geq 0})$. 
Set $\alpha_1=\alpha$, $\beta_{1}=\beta$, 
$\alpha_2=v(w)e^{\pi \sI \mu /2}\alpha$ and 
$\beta_{2}=v(w)e^{\pi \sI \mu /2}\beta$ with 
$L_1=L_2=L$. 
\\[1mm]
(i)  Assume 
$\lambda_{\alpha ,\beta }\in 
(I^{-\infty}_{\mu ,\nu})^{\widetilde{\Gamma}_0(1),\tilde{\chi}_v}$. 
Then the conditions {\normalfont [D1]}, {\normalfont [D2-1]}, {\normalfont [D3]} and {\normalfont [D4]} hold 
for the above $\alpha_1$, $\alpha_2$, $\beta_1$ and $\beta_2$. \\[1mm]
(ii) Assume that 
{\normalfont [D1]}, {\normalfont [D2-2]}, {\normalfont [D3]} and {\normalfont [D4]} hold 
for the above $\alpha_1$, $\alpha_2$, $\beta_1$ and $\beta_2$. 
Then $\lambda_{\alpha ,\beta } \in 
(I^{-\infty}_{\mu ,\nu})^{\widetilde{\Gamma}_0(1),\tilde{\chi}_v}$ 
and $(\lambda_{\alpha ,\beta })_\infty =
v(w)e^{\pi \sI \mu /2}\lambda_{\alpha ,\beta }$. 
Moreover, 
we have $\cP_{\nu ,\kappa}(\lambda_{\alpha ,\beta })\in 
\cM_{\nu }(\Gamma_0(1) \backslash \gH ;v,\kappa )$ 
with the Fourier expansion (\ref{eqn:poisson_Fexp}). 
\end{cor}
\begin{proof}
It is wellknown that 
$\Gamma_0(1)=SL(2,\bZ )$ is generated by $\ru (1)$ and $w$. 
Hence, by (\ref{eqn:wt_compatibility}), 
for $\lambda \in (I_{\mu,\nu}^{-\infty})_{L,L}$, 
we know that $\lambda \in 
(I^{-\infty}_{\mu ,\nu})^{\widetilde{\Gamma}_0(1),\tilde{\chi}_v}$ 
if and only if 
\begin{align*}
&\lambda (\rho (\tilde{w})F)=v(w)\lambda (F)&
&(F\in I_{\mu,\nu}^\infty ). 
\end{align*}
This equality is equivalent to 
$\lambda_\infty =v(w)e^{\pi \sI \mu /2}\lambda $. 
Hence, the assertion follows from 
Theorem \ref{thm:QAD_DS} together with 
Propositions \ref{prop:dist_to_form} and \ref{prop:poisson} (ii). 
\end{proof}

Until the end of this subsection, 
we assume $N>1$. 
Let $d$ be a positive integer coprime to $N$. 
Let $\psi$ be a Dirichlet character modulo $d$.  
For $m\in \bZ$, we define character sums $\tau_{\psi}(m)$ 
and $\hat{\tau}_{\psi ,v}(m)$ by 
\begin{align}
&\tau_{\psi}(m)=
\underset{\mathrm{gcd}(c,d)=1}{\sum_{c\bmod d}}\psi (c)
e^{2\pi \sI mc/d},\\
\label{eqn:def_char_mult_sum}
&\hat{\tau}_{\psi ,v}(m)
=\underset{\mathrm{gcd}(c,d)=1}{\sum_{c\bmod d}}
\overline{\psi (c)}v(\gamma )
e^{2\pi \sI (mc+\hat{u}c-ub)/d}
\end{align}
with $\displaystyle 
\gamma =\left(\begin{array}{cc}
a&b\\
Nc&d
\end{array}\right)\!\in \Gamma_0(N)$. 
Here we note that these definitions do not 
depend on the choice of a reduced residue system modulo $d$ 
and $\gamma \in \Gamma_0(N)$.

Let $(\alpha ,\beta )\in \mathfrak{M}(L)\times {\mathfrak{N}}(\bZ_{\geq 0})$ 
and  $(\hat{\alpha},\hat{\beta})\in \mathfrak{M}(\hat{L})
\times {\mathfrak{N}}(\bZ_{\geq 0})$. 
Let $(\alpha_{\psi} ,\beta_\psi )$ be an element of 
$\mathfrak{M}(L)\times {\mathfrak{N}}(\bZ_{\geq 0})$ determined by 
\begin{align*}
&\alpha_\psi (l)=\tau_{\psi}(l-u)\alpha (l)\quad (l\in L),&
&\beta_\psi =\tau_{\psi}(0)\beta .
\end{align*}
Let $(\hat{\alpha}_{\psi,v} ,\hat{\beta}_{\psi,v} )$ be an element of 
$\mathfrak{M}(d^{-2}\hat{L})\times {\mathfrak{N}}(\bZ_{\geq 0})$ determined by 
\begin{align*}
&\hat{\alpha}_{\psi,v}(d^{-2}l)=
\overline{\psi (-N)}
\hat{\tau}_{\psi ,v}(Nl-\hat{u})
d^{2\nu -1}\hat{\alpha}(l),\\
&\hat{\beta}_{\psi,v}(m)
=\overline{\psi (-N)}
\hat{\tau}_{\psi ,v}(0)
d^{-2\nu -2m -1}\\
&\phantom{\hat{\beta}_{\psi,v}(m)=}\times 
\left\{\begin{array}{ll}
\bigl(\hat{\beta}(0)+4\cos \bigl(\tfrac{\pi \mu }{2}\bigr)
\hat{\alpha}(0)\log d\bigr)
&\text{if $0\in \hat{L}$ and $m=\nu =0$},
\\[1mm]
\hat{\beta}(m)&\text{otherwise}
\end{array}\right.
\end{align*}
for $l\in \hat{L}$ and $m\in \bZ_{\geq 0}$. 
For $i=1,2$, we consider the following conditions 
{\normalfont [W1-$i$]} and {\normalfont [W2-$i$]${}_{d,\psi}$} 
for $\alpha ,\beta ,\hat{\alpha}$ and $\hat{\beta}$: 
\begin{description}
\item[{\normalfont [W1-$i$]}]The conditions {\normalfont [D1]}, 
{\normalfont [D2-$i$]}, {\normalfont [D3]} and {\normalfont [D4]} 
hold for $\alpha_1=\alpha$, $\beta_{1}=\beta$, $\alpha_2=\hat{\alpha}$ 
and $\beta_{2}=\hat{\beta}$ with $L_1=L$ and $L_2=\hat{L}$. 

\item[{\normalfont [W2-$i$]}${}_{d,\psi}$]The conditions 
{\normalfont [D1]}, {\normalfont  [D2-$i$]},  
{\normalfont [D3]} and {\normalfont [D4]} hold for 
$\alpha_1=\alpha_\psi$, $\beta_{1}=\beta_\psi$, 
$\alpha_2=\hat{\alpha}_{\psi,v}$ 
and $\beta_{2}=\hat{\beta}_{\psi,v}$ with $L_1=L$ and 
$L_2=d^{-2}\hat{L}$. 

\end{description}
Let $\bP_N$ be a subset of the set of 
positive odd prime integers not dividing $N$, 
such that 
$\bP_N\cap \{am+b\mid m\in \bZ\}\neq \emptyset $ 
for any integers $a$, $b$ coprime to each other 
(the existence of such $\bP_N$ follows from 
Dirichlet's theorem on arithmetic progressions).

\begin{thm}
\label{thm:Weil_converse}
Let $\mu \in \bR$, $\nu \in \bC$ and $N\in \bZ_{>1}$. 
Let $\kappa \in \mu +2\bZ$. 
Let $v$ be a multiplier system on $\Gamma_0(N)$ of weight $\kappa$. 
Let $L=u+\bZ$ and $\hat{L}=N^{-1}(\hat{u}+\bZ )$ with 
$0\leq u,\hat{u}<1$ determined by 
$v(\ru(1))=e^{2\pi \sI u}$ and 
$v(\overline{\ru}(-N))=e^{2\pi \sI \hat{u}}$. 
Let $(\alpha ,\beta )\in \mathfrak{M}(L)\times 
{\mathfrak{N}}(\bZ_{\geq 0})$ and 
$(\hat{\alpha},\hat{\beta})\in \mathfrak{M}(\hat{L})\times 
{\mathfrak{N}}(\bZ_{\geq 0})$. \\[1mm]
(i)  Assume $\lambda_{\alpha ,\beta } \in 
(I^{-\infty}_{\mu ,\nu})^{\widetilde{\Gamma}_0(N),\tilde{\chi}_v}$ and 
$(\lambda_{\alpha ,\beta})_\infty =\lambda_{\hat{\alpha},\hat{\beta}}$. 
Then the condition {\normalfont [W1-1]} holds. Moreover, 
the condition {\normalfont [W2-1]}${}_{d,\psi}$ holds 
for any $d\in \bZ_{>0}$ coprime to $N$ 
and any Dirichlet character $\psi$ modulo $d$. \\[1mm]
(ii) Assume that {\normalfont [W1-2]} holds. 
Assume furthermore that {\normalfont [W2-2]}${}_{d,\psi}$ holds 
for any $d\in \bP_N$ 
and any Dirichlet character $\psi$ modulo $d$. 
Then $\lambda_{\alpha ,\beta } \in 
(I^{-\infty}_{\mu ,\nu})^{\widetilde{\Gamma}_0(N),\tilde{\chi}_v}$ and 
$(\lambda_{\alpha ,\beta})_\infty =\lambda_{\hat{\alpha},\hat{\beta}}$. 
Moreover, 
we have $\cP_{\nu ,\kappa}(\lambda_{\alpha ,\beta })\in 
\cM_{\nu }(\Gamma_0(N) \backslash \gH ;v,\kappa )$ 
with the Fourier expansion (\ref{eqn:poisson_Fexp}). 
\end{thm}

This theorem is proved in \S \ref{subsec:proof_weil}. 
Theorem \ref{thm:Weil_converse} (ii) is 
a Weil type converse theorem for automorphic distributions 
and Maass forms of real weights, 
which is a generalization of \cite[Theorems 4.2 and 4.3]{preMSSU}.

\section{Preparation}
\label{sec:junbi}

\subsection{Some lemmas}
\label{subsec:lem_comp}

\begin{lem}
\label{lem:fourier_diff}
Let $S$ be a finite subset of $\bR$. 
For $f\in C(\bR)\cap L^1(\bR)$ 
such that $f\in C^1(\bR \smallsetminus S)$ 
and $f'\in L^1(\bR)$, we have 
\begin{align*}
&(2\pi \sI y)\cF (f)(y)=-\cF (f')(y)&
&(y\in \bR).
\end{align*}
\end{lem}
\begin{proof}
The assertion follows immediately from integration by parts.
\end{proof}

\begin{lem}
\label{lem:F_lambda_abcd}
Let $a,b,c,d\in \bR$ such that 
$\{x\in \bR \mid cx+d>0\}\neq \emptyset$. 
Let $s \in \bC$, $n\in \bZ_{\geq 0}$ and $f\in C^n(\bR )$. 
Then we have 
\begin{align*}
&\frac{d^n}{dx^n}\left\{(cx+d)^{s}
f\left(\frac{ax+b}{cx+d}\right)\right\}\\
&=\sum_{i=0}^n\binom{n}{i}(ad-bc)^ic^{n-i}
(s-n+1)_{n-i}(cx+d)^{s-n-i}
f^{(i)}\left(\frac{ax+b}{cx+d}\right)
\end{align*}
on $\{x\in \bR \mid cx+d>0\}$. 
Here $\displaystyle 
\binom{n}{i}=\frac{n!}{i!(n-i)!}$ is the binomial coefficient, 
and $(z)_i={\Gamma (z+i)}/{\Gamma (z)}=z(z+1)\cdots (z+i-1)$ is the 
Pochhammer symbol. 
\end{lem}
\begin{proof}
We obtain the assertion by induction with respect to $n$. 
\end{proof}

\subsection{Test functions}
\label{subsec:test_ftn}

We define a function $\Delta$ on $\bR$ by 
\begin{align*}
&\Delta (t)
=\left\{\begin{array}{ll}
\displaystyle 
\frac{\me (4-t^2)}{\me (4-t^2)+\me (4-1/t^2)}
&\text{ if }\ t\neq 0,\\[3mm]
1&\text{ if }\ t=0,
\end{array}\right.&
&\me (t)=\left\{\begin{array}{ll}
e^{-1/t}&\text{ if }\ t>0,\\
0&\text{ if }\ t\leq 0.
\end{array}\right.
\end{align*}
Then we note that $\Delta $ is a smooth function on $\bR$ 
with $\supp (\Delta)\subset \{t\in \bR \mid |t|\leq 2\}$ 
and has the following properties.
\begin{enumerate}
\item[(1)]$0\leq \Delta (t)\leq 1$ and $\Delta (t)=\Delta (-t)$ for $t\in \bR$.
\item[(2)]$\Delta (t)=1$ if and only if $|t|\leq 1/2$.
\item[(3)]$\Delta (t)+\Delta (-1/t)=1$ for $t\in \bR^\times$.
\end{enumerate}
We define a function $\varphi$ on $\bR$ by 
$\varphi (t)=0$ $(t\leq 0)$ and 
\begin{align*}
&\varphi (t)= 
\frac{1}{2\pi \sI}\int_{\mathrm{Re}(s)=\sigma_0}
\exp\!\left(-\frac{1}{\pi}\int_{-\infty}^{\infty}
\frac{(1-\sI sr)\sqrt{|r|}}{(\sI r-s)(1+r^2)}\,dr\right)
t^{-s}ds\ \ (t>0).
\end{align*}
Here the path of the integration 
$\int_{\mathrm{Re}(s)=\sigma_0}$ 
is the vertical line 
from $\sigma_0-\sI \infty$ to $\sigma_0+\sI \infty$ 
with $\sigma_0<0$. 

By the proof of \cite[Lemma 4.30]{Kimura_001}, 
we know that $\varphi $ is a smooth function on $\bR$ 
with $\supp (\varphi )\subset \{t\in \bR \mid t\geq 1\}$ 
and has the following properties. 
\begin{enumerate}
\item[(1)]For $n\in \bZ_{\geq 0}$, the $n$-th derivative 
$\displaystyle \varphi^{(n)}$ of $\varphi$ is bounded. 
\item[(2)]For $s\in \bC$ such that $\mathrm{Re}(s)<0$, 
we have 
\begin{align*}
\left|\int_{0}^\infty \varphi (t)t^{s-1}dt\right|
>\exp\!\left(-\sqrt{|\mathrm{Im}(s)|}
-\frac{\sqrt{|\mathrm{Re}(s)|}}{\pi}\int_{-\infty}^{\infty}\frac{\sqrt{|t|}}{1+t^2}dt
\right).
\end{align*}
\end{enumerate}

Let $\varepsilon \in \{\pm 1\}$, $\sigma \in \bR$ and $u>1$. 
We define functions $\varphi_{\varepsilon ,\sigma }\in C^\infty (\bR)$ and 
$\varphi_{\varepsilon ,\sigma ,u}\in C^\infty_0(\bR^\times )$ by 
\begin{align*}
&\varphi_{\varepsilon ,\sigma} (x)
=|x|^{-\sigma}\varphi (\varepsilon x),&
&\varphi_{\varepsilon ,\sigma ,u}(x)
=|x|^{-\sigma}\varphi (\varepsilon x)
\Delta \!\left(\frac{\varepsilon x}{2u}\right)&
&(x\in \bR^\times ).
\end{align*}
By the above properties of $\Delta$ and $\varphi$, we have 
\begin{align}
\label{eqn:supp_phies}
&\supp (\varphi_{\varepsilon ,\sigma ,u})\subset 
\supp (\varphi_{\varepsilon ,\sigma})\subset 
\{x\in \bR \mid \varepsilon x\geq 1\},\\
\label{eqn:supp_phies_minus}
&\supp (\varphi_{\varepsilon ,\sigma}-\varphi_{\varepsilon ,\sigma ,u})
\subset \{x\in \bR\mid \varepsilon x\geq u\},\\
\label{eqn:phies_estimate}
&(\varphi_{\varepsilon ,\sigma})^{(n)}(x)
=O(|x|^{-\sigma})\ \ (|x|\to \infty )
\hspace{10mm}
(n\in \bZ_{\geq 0}).
\end{align}

\begin{lem}
\label{lem:test_Shintani}
Let $\sigma \in \bR$ and $n\in \bZ_{\geq 0}$. 
There exists $\tilde{c}_{\sigma ,n}>0$ such that  
\begin{align}
\label{eqn:lem_shintani}
&\left|(\varphi_{\varepsilon ,\sigma}-
\varphi_{\varepsilon ,\sigma ,u})^{(n)}(x)\right|\leq 
\tilde{c}_{\sigma ,n}|x|^{-\sigma}&
&(x\in \bR^\times ,\ \varepsilon \in \{\pm 1\},\ u>1).
\end{align}
\end{lem}
\begin{proof}
By direct computation, we have 
\begin{align*}
&(\varphi_{\varepsilon ,\sigma}-
\varphi_{\varepsilon ,\sigma ,u})^{(n)}(x)
=\frac{d^n}{dx^n}\left\{\varphi_{1,\sigma }(\varepsilon x)- 
\varphi_{1,\sigma }(\varepsilon x)
\Delta\!\left(\frac{\varepsilon x}{2u}\right)\right\}\\
&=\varepsilon^n(\varphi_{1,\sigma })^{(n)}(\varepsilon x)
-\varepsilon^n\sum_{i=1}^n\binom{n}{i}
(\varphi_{1,\sigma })^{(i)}(\varepsilon x)
\Delta^{(n-i)}\!\left(\frac{\varepsilon x}{2u}\right)(2u)^{-n+i}.
\end{align*}
Hence, if we put 
\begin{align*}
\tilde{c}_{\sigma ,n}
=\sup_{x\in \bR}\left||x|^\sigma (\varphi_{1,\sigma })^{(n)}(x)\right|
+\sum_{i=1}^n\binom{n}{i}\sup_{x\in \bR}\left||x|^\sigma 
(\varphi_{1,\sigma })^{(i)}(x)\right| \,
\sup_{x\in \bR}\left|\Delta^{(n-i)}(x)\right|,
\end{align*}
the inequality (\ref{eqn:lem_shintani}) holds. 
\end{proof}
\begin{prop}
\label{prop:test_Shintani1}
Let $\sigma >1$ and $n\in \bZ_{\geq 0}$. 
There exists $c_{\sigma ,n}>0$ such that  
\begin{align}
\label{eqn:prop_shintani1}
&|y|^n|\cF (\varphi_{\varepsilon ,\sigma}-
\varphi_{\varepsilon ,\sigma ,u})(y)|
\leq c_{\sigma ,n}u^{1-\sigma }&
&(y\in \bR ,\ \varepsilon \in \{\pm 1\},\ u>1).
\end{align}
\end{prop}
\begin{proof}
By repeated application of Lemma \ref{lem:fourier_diff}, 
we have 
\begin{align*}
&(2\pi \sI y)^n\cF (\varphi_{\varepsilon ,\sigma}-
\varphi_{\varepsilon ,\sigma ,u})(y)
=(-1)^n\cF ((\varphi_{\varepsilon ,\sigma}-
\varphi_{\varepsilon ,\sigma ,u})^{(n)})(y).
\end{align*}
Hence we have 
\begin{align*}
&|y|^n\left|\cF (\varphi_{\varepsilon ,\sigma}-
\varphi_{\varepsilon ,\sigma ,u})(y)\right|
=\frac{1}{(2\pi)^n}
\left|\cF ((\varphi_{\varepsilon ,\sigma}-
\varphi_{\varepsilon ,\sigma ,u})^{(n)})(y)\right|\\
&\leq \frac{1}{(2\pi)^n}\int_{\bR }\left|
(\varphi_{\varepsilon ,\sigma}-
\varphi_{\varepsilon ,\sigma ,u})^{(n)}(x)\right|\,dx
\leq \frac{\tilde{c}_{\sigma ,n}u^{1-\sigma}}{(2\pi)^n(\sigma -1)}.
\end{align*}
Here the last inequality follows from (\ref{eqn:supp_phies_minus}), 
Lemma \ref{lem:test_Shintani} and 
\begin{align*}
&\int_u^\infty |x|^{-\sigma}\,dx=\int_{-\infty}^{-u} |x|^{-\sigma}\,dx
=\frac{u^{1-\sigma}}{\sigma -1}. 
\end{align*}
Therefore, we obtain (\ref{eqn:prop_shintani1}) 
with $\displaystyle {c}_{\sigma ,n}
=\frac{\tilde{c}_{\sigma ,n}}{(2\pi)^n(\sigma -1)}$. 
\end{proof}

\begin{prop}
\label{prop:test_Shintani2}
Let $n\in \bZ_{\geq 0}$, $\mu ,\nu \in \bC$ and 
$\sigma >2\mathrm{Re}(\nu )+2n$. 
There exists $c_{\mu ,\nu,\sigma ,n}>0$ such that 
\begin{align}
\label{eqn:prop_shintani2}
\begin{split}
|y|^n|\cF ((\varphi_{\varepsilon ,\sigma}-
\varphi_{\varepsilon ,\sigma ,u})_{\mu ,\nu,\infty})(y)|
\leq c_{\mu ,\nu,\sigma ,n}u^{-\sigma +2\mathrm{Re}(\nu )+2n}&\\
(y\in \bR ,\ \varepsilon \in \{\pm 1\},\ u>1).&
\end{split}
\end{align}
\end{prop}
\begin{proof}
By (\ref{eqn:supp_phies_minus}), we have
\begin{align}
\label{eqn:pf_prop_shintani2_002}
\supp ((\varphi_{\varepsilon ,\sigma}-
\varphi_{\varepsilon ,\sigma ,u})_{\mu ,\nu,\infty})
\subset \{x\in \bR\mid 0\leq -\varepsilon x\leq 1/u\}.
\end{align}
For $-\varepsilon x>0$, we have 
\begin{align*}
(\varphi_{\varepsilon ,\sigma}-
\varphi_{\varepsilon ,\sigma ,u})_{\mu ,\nu,\infty}(x)
=e^{\varepsilon \pi \sI \mu /2}
(-\varepsilon x)^{-2\nu -1}
(\varphi_{1,\sigma}-
\varphi_{1,\sigma ,u})\left(\frac{1}{-\varepsilon x}\right).
\end{align*}
By Lemma \ref{lem:F_lambda_abcd} with 
$a=d=0$, $b=1$, $c=-\varepsilon$, $s=-2\nu -1$ and 
$f=\varphi_{1,\sigma}-\varphi_{1,\sigma ,u}$, 
we have 
\begin{align*}
((\varphi_{\varepsilon ,\sigma}-
\varphi_{\varepsilon ,\sigma ,u})_{\mu ,\nu,\infty})^{(n)}(x)
=&\varepsilon^ne^{\varepsilon \pi \sI \mu /2}
\sum_{i=0}^n\binom{n}{i}(-1)^{n-i}
(-2\nu -n)_{n-i}\\
&\times (-\varepsilon x)^{-2\nu -1-n-i}
(\varphi_{1,\sigma}-\varphi_{1,\sigma ,u})^{(i)}
\left(\frac{1}{-\varepsilon x}\right)
\end{align*}
for $-\varepsilon x>0$. 
By Lemma \ref{lem:test_Shintani}, for $0<-\varepsilon x<1$, we have 
\begin{align}
\label{eqn:pf_prop_shintani2_001}
\left|((\varphi_{\varepsilon ,\sigma}-
\varphi_{\varepsilon ,\sigma ,u})_{\mu ,\nu,\infty})^{(n)}(x)\right|
\leq &\tilde{c}_{\mu ,\nu ,\sigma ,n}
|x|^{\sigma -2\mathrm{Re}(\nu) -1-2n}
\end{align}
with 
\begin{align*}
\displaystyle \tilde{c}_{\mu ,\nu ,\sigma ,n}=
e^{\pi |\mathrm{Im}(\mu)|/2}\sum_{i=0}^n\binom{n}{i}
\left|(-2\nu -n)_{n-i}\right|\tilde{c}_{\sigma ,i}.
\end{align*}
By repeated application of Lemma \ref{lem:fourier_diff}, 
we have 
\begin{align*}
&(2\pi \sI y)^n\cF ((\varphi_{\varepsilon ,\sigma}-
\varphi_{\varepsilon ,\sigma ,u})_{\mu,\nu,\infty})(y)
=(-1)^n\cF (((\varphi_{\varepsilon ,\sigma}-
\varphi_{\varepsilon ,\sigma ,u})_{\mu,\nu,\infty})^{(n)})(y).
\end{align*}
Hence we have 
\begin{align*}
&|y|^n\left|\cF ((\varphi_{\varepsilon ,\sigma}-
\varphi_{\varepsilon ,\sigma ,u})_{\mu,\nu,\infty})(y)\right|
=\frac{1}{(2\pi)^n}
\left|\cF (((\varphi_{\varepsilon ,\sigma}-
\varphi_{\varepsilon ,\sigma ,u})_{\mu,\nu,\infty})^{(n)})(y)\right|\\
&\leq \frac{1}{(2\pi)^n}\int_{\bR }\left|
((\varphi_{\varepsilon ,\sigma}-
\varphi_{\varepsilon ,\sigma ,u})_{\mu,\nu,\infty})^{(n)}(x)\right|\,dx
\leq \frac{\tilde{c}_{\mu ,\nu ,\sigma ,n}u^{-\sigma +2\mathrm{Re}(\nu)+2n}}
{(2\pi)^n(\sigma -2\mathrm{Re}(\nu)-2n)}.
\end{align*}
Here the last inequality follows from (\ref{eqn:pf_prop_shintani2_002}), 
(\ref{eqn:pf_prop_shintani2_001}) and 
\begin{align*}
\int_0^{1/u} |x|^{\sigma -2\mathrm{Re}(\nu) -1-2n}\,dx
=\int_{-1/u}^{0} |x|^{\sigma -2\mathrm{Re}(\nu) -1-2n}\,dx
=\frac{u^{-\sigma +2\mathrm{Re}(\nu) +2n}}{\sigma -2\mathrm{Re}(\nu) -2n}. 
\end{align*}
Therefore, we obtain (\ref{eqn:prop_shintani2}) 
with $\displaystyle {c}_{\mu ,\nu ,\sigma ,n}
=\frac{\tilde{c}_{\mu ,\nu ,\sigma ,n}}
{(2\pi)^n(\sigma -2\mathrm{Re}(\nu)-2n)}$. 
\end{proof}

\begin{cor}
\label{cor:test_Shintani3}
Let $\mu ,\nu \in \bC$, $\varepsilon \in \{\pm 1\}$ 
and $\sigma >2\mathrm{Re}(\nu )$. 
For $n\in \bZ$ such that $\sigma -2\mathrm{Re}(\nu )>2n\geq 0$, 
we have 
\begin{align*}
&\cF ((\varphi_{\varepsilon ,\sigma})_{\mu ,\nu,\infty})(y)
=O(|y|^{-n})\quad (|y|\to \infty )
\end{align*}
\end{cor}
\begin{proof}
Since $\varphi_{\varepsilon ,\sigma ,u}\in C^\infty_0(\bR^\times )$, 
we have $\cF ((\varphi_{\varepsilon ,\sigma ,u})_{\mu ,\nu,\infty})
\in \cS (\bR )$ for $u>1$. Hence, the assertion 
follows immediately from Proposition \ref{prop:test_Shintani2}. 
\end{proof}

\section{The local zeta functions}
\label{sec:local_zeta}

\subsection{The definition of the local zeta functions}
\label{subsec:def_local_zeta}

Let $\varepsilon \in \{\pm 1\}$ and $f\in C(\bR^\times )$. 
For $s\in \bC$, we set 
\begin{align*}
&\Phi_{\varepsilon}(f;s)=\int_0^\infty f(\varepsilon t)t^{s-1}\,dt
=\Phi_{\varepsilon ,+}(f;s)+\Phi_{\varepsilon ,-}(f;s)
\end{align*}
with 
\begin{align*}
&\Phi_{\varepsilon ,+}(f;s)=\int_1^\infty f(\varepsilon t)t^{s-1}\,dt,& 
&\Phi_{\varepsilon ,-}(f;s)=\int_0^1f(\varepsilon t)t^{s-1}\,dt. 
\end{align*}
The integral $\Phi_{\varepsilon ,+}(f;s)$ 
converges absolutely and defines a holomorphic function 
on $\mathrm{Re}(s)<-r$ if there is $r\in \bR$ such that 
$f(x)=O(|x|^r)\ (|x|\to \infty )$. 
The integral $\Phi_{\varepsilon ,-}(f;s)$ 
converges absolutely and defines a holomorphic function 
on $\mathrm{Re}(s)>-r$ if there is $r\in \bR$ such that 
$f(x)=O(|x|^r)\ (|x|\to 0)$. 
We call $\Phi_{\varepsilon }(f;s)$ the {\it local zeta function}.

\subsection{The meromorphic continuation of $\Phi_{\pm 1}(f;s)$}
\label{subsec:AC_local_zeta}

Let $\varepsilon \in \{\pm 1\}$ and $f\in \cS (\bR )$. 
Then $\Phi_{\varepsilon}(f;s)$ converges absolutely and defines 
a holomorphic function on $\mathrm{Re}(s)>0$. 
In this subsection, we give the meromorphic continuation of 
$\Phi_{\varepsilon }(f;s)$. 
For $s\in \bC$ and $m\in \bZ_{\geq 0}$, we set 
\begin{align*}
\Phi_{\varepsilon ,-}^{m}(f;s)=
\int_0^1\left(
f(\varepsilon t)-\sum_{n=0}^{m-1}\frac{\delta^{(n)}(f)}{n!}
(-\varepsilon t)^n
\right)t^{s-1}\,dt 
\end{align*}
with $\delta^{(n)}(f)=(-1)^nf^{(n)}(0)$. 
Here we note that $\Phi_{\varepsilon ,-}^{0}(f;s)=
\Phi_{\varepsilon ,-}(f;s)$. 
Maclaurin's theorem asserts that, 
for $m\in \bZ_{\geq 0}$ and $x\in \bR $, 
there exists $0<\theta <1$ such that 
\begin{align}
\label{eqn:maclaurin}
&f(x)-\sum_{n=0}^{m-1}\frac{\delta^{(n)}(f)}{n!}(-x)^n
=\frac{f^{(m)}(\theta x)}{m!}x^m.
\end{align}
Hence, we know that the integral $\Phi_{\varepsilon ,-}^{m}(f;s)$ 
converges absolutely and defines a holomorphic function 
on $\mathrm{Re}(s)>-m$. 

Using the equality 
$\displaystyle \frac{1}{s+n}=\int_{0}^1t^{s+n-1}\,dt$, 
we have 
\begin{align}
\label{eqn:AC_Phi}
\Phi_{\varepsilon }(f;s)
=&\Phi_{\varepsilon ,+}(f;s)+\Phi_{\varepsilon ,-}^{m}(f;s)+
\sum_{n=0}^{m-1}\frac{(-\varepsilon )^n\delta^{(n)}(f)}{n!(s+n)} 
\end{align}
for $m\in \bZ_{\geq 0}$ and $s\in \bC$ such that $\mathrm{Re}(s)>0$. 
The expression (\ref{eqn:AC_Phi}) gives the meromorphic continuation 
of $\Phi_{\varepsilon }(f;s)$ to $\mathrm{Re}(s)>-m$. 
Since a non-negative integer $m$ can be chosen arbitrarily, 
we have the meromorphic continuation of $\Phi_{\varepsilon }(f;s)$ 
to the whole $s$-plane. 
The local zeta function $\Phi_{\varepsilon }(f;s)$ 
is holomorphic on $\bC \smallsetminus \bZ_{\leq 0}$ and 
has poles of order at most $1$ at $s=-n\ (n\in \bZ_{\geq 0})$.

\subsection{The local zeta distributions}
\label{subsec:dist_local_zeta}

In this subsection, we explain that 
the local zeta functions can be regarded as distributions. 
Let $\varepsilon \in \{\pm 1\}$. 
Let $u>1$. 
Then, for $s\in \bC$ and $f\in C^\infty_0(\bR )$ such that 
$\supp (f)\subset \{x\in \bR \mid |x|\leq u\}$, we have 
\begin{align*}
&\left|\Phi_{\varepsilon ,+}(f;s)\right|
=\left|\int_1^u f(\varepsilon t)t^{s-1}\,dt\right|
\leq \left(\int_1^u t^{\mathrm{Re}(s)-1}\,dt\right)
\sup_{x\in \bR}|f(x)|.
\end{align*}
Hence 
$C^\infty_0(\bR )\ni f\mapsto \Phi_{\varepsilon ,+}(f;s)\in \bC$ 
is a distribution on $\bR$ for $s\in \bC$. 

Let $m\in \bZ_{\geq 0}$. 
By Maclaurin's theorem (\ref{eqn:maclaurin}), we have 
\begin{align*}
&\left|\Phi_{\varepsilon ,-}^m(f;s)\right|
=\left|\int_0^1 \left(
f(\varepsilon t)-\sum_{n=0}^{m-1}\frac{\delta^{(n)}(f)}{n!}
(-\varepsilon t)^n
\right)t^{s-1}\,dt\right|\\
&\leq \frac{1}{m!}\left(\int_0^1 t^{\mathrm{Re}(s)+m-1}\,dt\right)
\sup_{x\in \bR}|f^{(m)}(x)|
=\frac{1}{m!(\mathrm{Re}(s)+m)}\sup_{x\in \bR}|f^{(m)}(x)|
\end{align*}
for $f\in C^\infty_0(\bR )$ and $s\in \bC$ 
such that $\mathrm{Re}(s)>-m$.
Hence 
$C^\infty_0(\bR )\ni f\mapsto \Phi_{\varepsilon ,-}^{m}(f;s)\in \bC$ 
is a distribution on $\bR$ for $s\in \bC$ such that $\mathrm{Re}(s)>-m$.

Therefore, because of the expression (\ref{eqn:AC_Phi}), 
we know that $C^\infty_0(\bR )\ni f\mapsto \Phi_{\varepsilon }(f;s)\in \bC$ 
is a distribution on $\bR$ for $s\in \bC \smallsetminus \bZ_{\leq 0}$.

\subsection{The local functional equation}
\label{subsec:LFE}

In this subsection, we introduce some equations 
for the local zeta functions. 
The equation in the following proposition is called 
the {\it local functional equation}. 
\begin{prop}[{\cite[Proposition 4.21]{Kimura_001}}]
\label{prop:LFE}
For $f\in \cS (\bR )$, we have 
\begin{align*}
(\Phi_1(\cF (f);s),\Phi_{-1}(\cF (f);s))
=(2\pi )^{-s}\Gamma (s)(\Phi_1(f;-s+1),\Phi_{-1}(f;-s+1)){\mathrm{E}}(s).
\end{align*}
Here ${\mathrm{E}}(s)$ is the matrix defined by (\ref{eqn:def_matrix_gamma}).
\end{prop}

\begin{lem}
\label{lem:LFE_infty}
Let $\varepsilon \in \{\pm 1\}$ and $f\in C_0(\bR^\times )$. 
Then $\Phi_{\varepsilon }(f;s)$ is entire and is bounded on a vertical strip 
$\sigma_1\leq \mathrm{Re}(s)\leq \sigma_2$ for $\sigma_1,\sigma_2\in \bR$ 
such that $\sigma_1<\sigma_2$. 
Moreover, we have the equality 
\begin{align*}
(\Phi_{1}(f_{\mu,\nu,\infty};s),\Phi_{-1}(f_{\mu,\nu,\infty};s))
=&(\Phi_{1}(f;-s+2\nu +1),\Phi_{-1}(f;-s+2\nu +1))\Sigma_{\mu}
\end{align*}
for $\mu ,\nu \in \bC$. 
Here $\Sigma_{\mu}$ is the matrix defined by (\ref{eqn:def_matrix_sigma}). 
\end{lem}
\begin{proof}
The entireness of $\Phi_{\varepsilon }(f;s)$ is obvious. 
For $\sigma_1,\sigma_2\in \bR$ such that $\sigma_1<\sigma_2$, 
the function $\Phi_{\varepsilon }(f;s)$ is bounded on 
a vertical strip $\sigma_1\leq \mathrm{Re}(s)\leq \sigma_2$, 
since 
\begin{align*}
|\Phi_{\varepsilon }(f;s)|
=&\left|\int_0^\infty f(\varepsilon t)t^{s-1}\,dt\right|
\leq \int_1^\infty |f(\varepsilon t)|t^{\sigma_2-1}\,dt
+\int_0^1 |f(\varepsilon t)|t^{\sigma_1-1}\,dt. 
\end{align*}
By direct computation, we have 
\begin{align*}
\Phi_{\varepsilon }(f_{\mu ,\nu ,\infty};s)
=&\int_0^\infty f_{\mu ,\nu ,\infty}(\varepsilon t)t^{s-1}\,dt
%
=e^{-\varepsilon \pi \sI \mu /2} \int_0^\infty 
f(-\varepsilon t^{-1})t^{s-2\nu -2}\,dt\\
%
%
=&e^{-\varepsilon \pi \sI \mu /2} \Phi_{-\varepsilon }(f;-s+2\nu +1).
\end{align*}
Here the third equality follows from the substitution $t\to t^{-1}$. 
This implies that the equality in the statement holds.  
\end{proof}

\subsection{The local zeta function for $\varphi_{\varepsilon ,\sigma}$}

Let $\varepsilon \in \{\pm 1\}$ and $\sigma >1$. 
Let $\varphi_{\varepsilon ,\sigma}\in C^\infty (\bR )$ and 
$\varphi_{\varepsilon ,\sigma ,u}\in C^\infty_0(\bR^\times)$ 
$(u>1)$ be the functions in \S \ref{subsec:test_ftn}. 
We consider here the local zeta function for 
$\varphi_{\varepsilon ,\sigma}$. 
By (\ref{eqn:supp_phies}) and (\ref{eqn:phies_estimate}), 
the integral 
$\Phi_{1}(\varphi_{1,\sigma};s)=\Phi_{-1}(\varphi_{-1,\sigma};s)$ 
converges absolutely and defines a holomorphic function 
on $\mathrm{Re}(s)<\sigma $. By (\ref{eqn:supp_phies}), we have 
 $\Phi_{-1}(\varphi_{1,\sigma};s)=\Phi_{1}(\varphi_{-1,\sigma};s)=0\ 
(s\in \bC )$. By the property (2) of $\varphi$ 
in \S \ref{subsec:test_ftn}, 
for $s\in \bC$ such that $\mathrm{Re}(s)<\sigma $, we have 
\begin{align}
\label{eqn:est_LZ_test}
\left|\Phi_{1}(\varphi_{1,\sigma};s)\right|
>\exp\!\left(-\sqrt{|\mathrm{Im}(s)|}
-\frac{\sqrt{\sigma -\mathrm{Re}(s)}}{\pi}
\int_{-\infty}^{\infty}\frac{\sqrt{|t|}}{1+t^2}dt
\right).
\end{align}

By (\ref{eqn:supp_phies}), (\ref{eqn:phies_estimate}) and 
Lemma \ref{lem:fourier_diff}, we have 
\begin{align}
\label{eqn:Fourier_phies_estimate}
&\cF (\varphi_{\varepsilon ,\sigma})(y)=O(|y|^{-n})\quad 
(|y|\to \infty)&&(n\in \bZ_{\geq 0}). 
\end{align}
Hence, 
for $\varepsilon_1\in \{\pm 1\}$, the integral 
$\Phi_{\varepsilon_1}(\cF (\varphi_{\varepsilon ,\sigma});s)$ 
converges absolutely and defines a holomorphic function 
on $\mathrm{Re}(s)>0$. 

\begin{lem}
\label{lem:LFE_test}
Retain the notation. 
For $s\in \bC$ such that $\mathrm{Re}(s)>0$, we have 
\begin{align*}
&(\Phi_1(\cF (\varphi_{\varepsilon ,\sigma});s),
\Phi_{-1}(\cF (\varphi_{\varepsilon ,\sigma});s))\\
&=(2\pi )^{-s}\Gamma (s)(\Phi_1(\varphi_{\varepsilon ,\sigma};-s+1),
\Phi_{-1}(\varphi_{\varepsilon ,\sigma};-s+1)){\mathrm{E}}(s). 
\end{align*}
\end{lem}
\begin{proof}
Let $\varepsilon,\varepsilon_1\in \{\pm 1\}$ and $u>1$. 
Let $s\in \bC$ such that $\mathrm{Re}(s)>0$. 
By (\ref{eqn:supp_phies}), 
we have 
$\Phi_{-\varepsilon}(\varphi_{\varepsilon,\sigma ,u};s)=0$. 
Hence, applying Proposition \ref{prop:LFE} to 
$f=\varphi_{\varepsilon ,\sigma ,u}$, we have 
\begin{align}
\label{eqn:pf_testLFE001}
\Phi_{\varepsilon_1}(\cF (\varphi_{\varepsilon ,\sigma ,u});s)
=(2\pi)^{-s}\Gamma (s)
e^{\varepsilon \varepsilon_1\pi \sI s/2}
\Phi_{\varepsilon }(\varphi_{\varepsilon ,\sigma ,u};-s+1).
\end{align}
Since 
$|x|^\sigma |(\varphi_{\varepsilon ,\sigma }
-\varphi_{\varepsilon ,\sigma ,u})(x)|
\leq |\varphi (\varepsilon x)|$ for $x\in \bR$, 
we have  
\begin{align}
\label{eqn:pf_testLFE002}
&|x|^\sigma |(\varphi_{\varepsilon ,\sigma }
-\varphi_{\varepsilon ,\sigma ,u})(x)|\leq c_\varphi &
&(x\in \bR)
\end{align}
with $c_\varphi =\sup_{x\in \bR}|\varphi (x)|$. 
We take $n\in \bZ_{\geq 0}$ so that $\mathrm{Re}(s)<n$. 
Using (\ref{eqn:supp_phies_minus}), (\ref{eqn:pf_testLFE001}), 
(\ref{eqn:pf_testLFE002}) and Proposition \ref{prop:test_Shintani1}, 
we have 
\begin{align*}
&\left|\Phi_{\varepsilon_1}(\cF (\varphi_{\varepsilon ,\sigma });s)
-(2\pi)^{-s}\Gamma (s)
e^{\varepsilon \varepsilon_1\pi \sI s/2}
\Phi_{\varepsilon }(\varphi_{\varepsilon ,\sigma };-s+1)\right|\\
&=\left|\Phi_{\varepsilon_1}(\cF (\varphi_{\varepsilon ,\sigma }
-\varphi_{\varepsilon ,\sigma ,u});s)
-(2\pi)^{-s}\Gamma (s)
e^{\varepsilon \varepsilon_1\pi \sI s/2}
\Phi_{\varepsilon }(\varphi_{\varepsilon ,\sigma }
-\varphi_{\varepsilon ,\sigma ,u};-s+1)\right|\\
&\leq \int_0^\infty \left|
\cF (\varphi_{\varepsilon ,\sigma }
-\varphi_{\varepsilon ,\sigma ,u})
(\varepsilon_1 t)\right|t^{\mathrm{Re}(s)-1}\,dt\\
&\hphantom{=.}
+(2\pi)^{-\mathrm{Re}(s)}
e^{-\varepsilon \varepsilon_1\pi \mathrm{Im}(s)/2}
|\Gamma (s)|
\int_0^\infty \left|
(\varphi_{\varepsilon ,\sigma }
-\varphi_{\varepsilon ,\sigma ,u})
(\varepsilon t)\right|t^{-\mathrm{Re}(s)}\,dt\\
&\leq \left(\int_0^1 c_{\sigma ,0}u^{1-\sigma }t^{\mathrm{Re}(s)-1}\,dt
+\int_1^\infty c_{\sigma ,n}u^{1-\sigma }t^{\mathrm{Re}(s)-n-1}\,dt\right)\\
&\hphantom{=.}
+(2\pi)^{-\mathrm{Re}(s)}e^{-\varepsilon \varepsilon_1\pi \mathrm{Im}(s)/2}
|\Gamma (s)|
\int_u^\infty 
c_\varphi t^{-\mathrm{Re}(s)-\sigma}\,dt\\
&\leq \left(\frac{c_{\sigma ,0}}{\mathrm{Re}(s)}
+\frac{c_{\sigma ,n}}{n-\mathrm{Re}(s)}\right)u^{1-\sigma }
+\frac{c_\varphi (2\pi)^{-\mathrm{Re}(s)}
e^{-\varepsilon \varepsilon_1\pi \mathrm{Im}(s)/2}
|\Gamma (s)|}{\mathrm{Re}(s)+\sigma -1}
u^{-\mathrm{Re}(s)-\sigma +1}.
\end{align*}
Since this inequality holds for any $u>1$,
we have 
\begin{align*}
&\Phi_{\varepsilon_1}(\cF (\varphi_{\varepsilon ,\sigma });s)
=(2\pi)^{-s}\Gamma (s)
e^{\varepsilon \varepsilon_1\pi \sI s/2}
\Phi_{\varepsilon }(\varphi_{\varepsilon ,\sigma };-s+1).
\end{align*}
This implies that the equality in the statement holds. 
\end{proof}

\section{Dirichlet series}
\label{sec:DS}

In this section, we give a proof of Theorem \ref{thm:DS}.

\subsection{Definition of zeta integrals}
\label{subsec:zeta_integrals}

Let ${L}$ be a shifted lattice in $\bR$. 
Let $r\in \bR$, $\alpha \in {\mathfrak{M}}_r({L})$ and $\sigma >r+1$. 
Let $f$ be a function in $C (\bR^\times )$ satisfying 
the following condition:  
\begin{align}
\label{eqn:cdn_zeta_def}
\text{$x\mapsto |x|^{\sigma}f(x)$ is a bounded function 
on $\bR^\times$. }
\end{align}
It is convenient to introduce a series 
\begin{align*}
&\vartheta({\alpha ,f};t)
=\sum_{0\neq l\in {L}}\alpha (l)f(tl)&
&(t>0).
\end{align*}
\begin{lem}
\label{lem:theta_modoki1}
We take ${L}$, $r$, $\alpha$ and $\sigma $ as above. 
For $f\in C (\bR^\times )$ satisfying (\ref{eqn:cdn_zeta_def}), 
the series $\vartheta({\alpha ,f};t)$ converges absolutely and defines 
a continuous function on $\bR_{>0}$. 
Moreover, there is a constant $C_{\alpha ,{\sigma}}>0$ such that 
\begin{align}
\label{eqn:theta_est}
\begin{split}
\sum_{0\neq l\in {L}}|\alpha (l)f(tl)|
\leq C_{\alpha ,{\sigma}}\sup_{x\in \bR^\times}(|x|^\sigma |f(x)|)
t^{-{\sigma}}&\\
(t>0,\ \text{$f\in C (\bR^\times )$ satisfying (\ref{eqn:cdn_zeta_def}) })&.
\end{split}
\end{align} 
\end{lem}
\begin{proof}
Since $\alpha \in {\mathfrak{M}}_r({L})$, 
there is $C_\alpha >0$ such that 
$|\alpha (l)|\leq C_\alpha  |l|^r \ (0\neq l\in {L})$. 
Hence, for $t>0$ and $0\neq l\in {L}$, 
we have 
\begin{align*}
\left|\alpha (l)f(tl)\right|
&\leq C_\alpha |l|^r
\sup_{x\in \bR^\times}(|x|^\sigma |f(x)|)
|tl|^{-{\sigma}}
=C_\alpha |l|^{-({\sigma}-r)}\sup_{x\in \bR^\times }
(|x|^\sigma |f(x)|)t^{-{\sigma}}.
\end{align*}
Since $\sum_{0\neq l\in {L}}|l|^{-({\sigma}-r)}$ converges, 
we obtain the assertion. 
\end{proof}

For $s\in \bC$, we define the zeta integral $Z(\alpha ,f;s)$ by 
\begin{align*}
&Z(\alpha ,f;s)
=\int_0^\infty \vartheta({\alpha ,f};t)t^{s}\frac{dt}{t}
=Z_+(\alpha ,f;s)+Z_-(\alpha ,f;s)
\end{align*}
with 
\begin{align*}
&Z_+(\alpha ,f;s)
=\int_1^\infty \vartheta({\alpha ,f};t)t^{s}\frac{dt}{t},& 
&Z_-(\alpha ,f;s)
=\int_0^1 \vartheta({\alpha ,f};t)t^{s}\frac{dt}{t}. 
\end{align*}
\begin{lem}
\label{lem:par_zeta_converge}
We take ${L}$, $r$, $\alpha$ and $\sigma $ as above. 
Let $f$ be a function in $C (\bR^\times )$ 
satisfying (\ref{eqn:cdn_zeta_def}). 
We take $C_{\alpha ,\sigma}>0$ as in Lemma \ref{lem:theta_modoki1}. 

\noindent (i) 
The integral $Z_+(\alpha ,f;s)$ converges absolutely 
and defines a holomorphic function on $\mathrm{Re} (s)< {\sigma}$. 
For $s\in \bC$ such that $\mathrm{Re}(s)<\sigma $, 
we have 
\begin{align*}
&|Z_+(\alpha ,f;s)|\leq 
\frac{C_{\alpha ,\sigma}}{\sigma -\mathrm{Re}(s)}
\sup_{x\in \bR^\times }(|x|^\sigma |f(x)|).
\end{align*}

\noindent (ii) 
The integral $Z_-(\alpha ,f;s)$ converges absolutely and 
defines a holomorphic function on $\mathrm{Re} (s)>\sigma$. 
For $s\in \bC$ such that $\mathrm{Re}(s)>\sigma $, 
we have 
\begin{align*}
&|Z_-(\alpha ,f;s)|\leq 
\frac{C_{\alpha ,\sigma}}{\mathrm{Re}(s)-\sigma }
\sup_{x\in \bR^\times }(|x|^\sigma |f(x)|).
\end{align*}

\end{lem}
\begin{proof}
The assertion follows immediately from 
Lemma \ref{lem:theta_modoki1}.
\end{proof}

\begin{lem}
\label{lem:zeta_converge}
We take ${L}$, $r$ and $\alpha$ as above. 
Let $\sigma_2>\sigma_1>r+1$. 
Let $f$ be a function in $C (\bR^\times )$ such that 
$x\mapsto |x|^{\sigma_i}f(x)$ is a bounded function 
on $\bR^\times$ for $i\in \{1,2\}$. 
Then the zeta integral $Z(\alpha ,f;s)$ 
converges absolutely 
and defines a holomorphic function on 
$\sigma_1<\mathrm{Re} (s)<\sigma_2$. 
Moreover, for $s\in \bC$ such that $\sigma_1<\mathrm{Re} (s)<\sigma_2$,
we have 
\begin{align}
\label{eqn:rel_zeta_DS}
Z(\alpha ,f;s)
=(\Phi_1(f;s),\Phi_{-1}(f;s))\left(\begin{array}{c}
\xi_+(\alpha ;s)\\
\xi_-(\alpha ;s)
\end{array}\right).
\end{align}
\end{lem}
\begin{proof}
The former part of the statement follows immediately 
from Lemma \ref{lem:par_zeta_converge}. 
The equality (\ref{eqn:rel_zeta_DS}) is obtained as follows:
\begin{align*}
&Z(\alpha ,f;s)
=\int_0^\infty \vartheta({\alpha ,f};t)t^{s}\frac{dt}{t}
=\int_0^\infty \sum_{0\neq l\in L}\alpha (l)f(tl)t^{s}\frac{dt}{t}\\
&=\sum_{0<l\in L}\alpha (l)\int_0^\infty f(tl)t^{s}\frac{dt}{t}
+\sum_{0>l\in L}\alpha (l)\int_0^\infty f(tl)t^{s}\frac{dt}{t}\\
&=\xi_+(\alpha ;s)\Phi_{1}(f;s)
+\xi_-(\alpha ;s)\Phi_{-1}(f;s)
=(\Phi_1(f;s),\Phi_{-1}(f;s))\left(\begin{array}{c}
\xi_+(\alpha ;s)\\
\xi_-(\alpha ;s)
\end{array}\right).
\end{align*}
Here the fourth equality follows from 
the substitution $t\to t|l|^{-1}$. 
\end{proof}

\subsection{Functional equations}
\label{subsec:DS_FE}

In this subsection, we give a proof of the following 
proposition, which is the former part of Theorem \ref{thm:DS} (i). 
\begin{prop}
\label{prop:DS_FE}
Let ${L}_1$ and ${L}_2$ be two shifted lattices in $\bR$. 
Let $\mu,\nu \in \bC$. 
Let $(T_{\alpha_1},T_{\alpha_2})\in 
\cA (L_1,L_2;J_{\mu ,\nu})$ with 
$\alpha_1\in \mathfrak{M}(L_1)$ and $\alpha_2\in \mathfrak{M}(L_2)$. 
Then $\Xi_\pm (\alpha_1;s)$ and $\Xi_\pm (\alpha_2;s)$ satisfy 
the condition {\normalfont [D1]} in \S \ref{subsec:DS_AP}. 
\end{prop}
In order to prove this proposition, we prepare some lemmas. 
For $t>0$ and a function $f$ on $\bR$ or $\bR^\times$, 
we set $f_{[t]}(x)=f(tx)$. 
\begin{lem}
\label{lem:rel_f_t}
Let $\mu ,\nu \in \bC$, $y\in \bR$ and $t>0$. 

\noindent (i) 
$\cF (f)(ty)=t^{-1}\cF (f_{[t^{-1}]})(y)$ for $f\in L^{1}(\bR)$. 

\noindent (ii) 
$(f_{[t^{-1}]})_{\mu,\nu,\infty} =t^{2\nu +1}(f_{\mu,\nu,\infty})_{[t]}$ 
for $f\in C(\bR^\times )$.

\end{lem}
\begin{proof}
For $f\in L^{1}(\bR)$, we have 
\begin{align*}
\cF (f)(ty)&=\int_{-\infty}^{\infty}f(x){e}^{2\pi \sI xty}\, dx
=t^{-1}\cF(f_{[t^{-1}]})(y).
\end{align*}
Here the second equality follows from the substitution $x\to t^{-1}x$. 
For $x\in \bR^\times$ and $f\in C(\bR^\times )$, we have 
\begin{align*}
t^{2\nu +1}(f_{\mu,\nu,\infty})_{[t]}(x)
&=
e^{-\sgn (x) \pi \sI \mu /2}|x|^{-2\nu -1}
f(-1/(tx))
=(f_{[t^{-1}]})_{\mu,\nu,\infty}(x).
\end{align*}
Here the first equality follows from $\sgn (tx)=\sgn (x)$. 
\end{proof}
\begin{lem}
\label{lem:Po_sum_modoki}
We use the notation in Proposition \ref{prop:DS_FE}. 
Then we have 
\begin{align*}
&\sum_{l\in {L}_1}\alpha_{1}(l)\cF (f)(tl)
=t^{2\nu -1}\sum_{l\in {L}_2}\alpha_{2}(l)
\cF (f_{\mu,\nu,\infty} )(t^{-1}l)&
&(t>0,\ f\in C^\infty_0(\bR^\times )).
\end{align*}
\end{lem}
\begin{proof}
By Lemma \ref{lem:rel_f_t} 
and the relation (\ref{eqn:autom_pair_def}), we have 
\begin{align*}
&\sum_{l\in {L}_1}\alpha_{1}(l)\cF (f)(tl)
=t^{-1}\sum_{l\in {L}_1}\alpha_{1}(l)\cF (f_{[t^{-1}]})(l)\\
&=t^{-1}T_{\alpha_1}(f_{[t^{-1}]})
=t^{-1}T_{\alpha_2}((f_{[t^{-1}]})_{\mu,\nu,\infty})
=t^{2\nu}T_{\alpha_2}((f_{\mu,\nu,\infty})_{[t]})\\
&=t^{2\nu}\sum_{l\in {L}_2}
\alpha_{2}(l)\cF ((f_{\mu,\nu,\infty})_{[t]})(l)
=t^{2\nu -1}\sum_{l\in {L}_2}\alpha_{2}(l)
\cF (f_{\mu,\nu,\infty})(t^{-1}l)
\end{align*}
for $t>0$ and $f\in C^\infty_0(\bR^\times )$. 
Hence, we obtain the assertion. 
\end{proof}

\begin{lem}
\label{lem:AC_global_zeta}
We use the notation in Proposition \ref{prop:DS_FE}. 
Let $f\in C^\infty_0(\bR^\times)$. 
Take $r\in \bR$ so that $\alpha_{1}\in 
\mathfrak{M}_r({L}_1)$. 
Then, for $s\in \bC$ such that $\mathrm{Re}(s)>
\max \{r+1,0\}$, we have 
\begin{align*}
Z(\alpha_{1},\cF (f);s)
=
&Z_+(\alpha_{1},\cF (f);s)
+Z_+(\alpha_{2},\cF (f_{\mu,\nu,\infty});-s-2\nu +1)\\
&
-\frac{\alpha_{1}(0)\cF (f)(0)}{s}
+\frac{\alpha_{2}(0)\cF (f_{\mu,\nu,\infty})(0)}{s+2\nu -1}.
\end{align*}
\end{lem}
\begin{proof}
By Lemma \ref{lem:Po_sum_modoki}, we have 
\begin{align}
\nonumber 
\vartheta(\alpha_{1},\cF (f);t)
=&t^{2\nu -1}\vartheta(\alpha_{2},\cF (f_{\mu,\nu,\infty});t^{-1})\\
\label{eqn:pf_zeta_FE_001}
&-\alpha_{1}(0)\cF (f)(0)
+t^{2\nu -1}\alpha_{2}(0)\cF (f_{\mu,\nu,\infty})(0).
\end{align}
If $\mathrm{Re}(s)$ is sufficiently large, we have 
\begin{align*}
&Z_-(\alpha_{1},\cF (f);s)
=
\int_0^1 \vartheta(\alpha_{1},\cF (f);t)t^{s}
\frac{dt}{t}\\
&=
\int_0^1 \vartheta(\alpha_{2},\cF (f_{\mu,\nu,\infty});t^{-1})
t^{s+2\nu -1}
\frac{dt}{t}\\
&\hphantom{=}
-\int_0^1 \alpha_{1}(0)\cF (f)(0)t^{s}\frac{dt}{t}
+\int_0^1 \alpha_{2}(0)\cF (f_{\mu,\nu,\infty})(0)t^{s+2\nu -1}\frac{dt}{t}\\
&=
Z_+(\alpha_{2},\cF (f_{\mu,\nu,\infty}) ;-s-2\nu +1)
-\frac{\alpha_{1}(0)\cF (f)(0)}{s}
+\frac{\alpha_{2}(0)\cF (f_{\mu,\nu,\infty})(0)}{s+2\nu -1}.
\end{align*}
Here the second equality follows from (\ref{eqn:pf_zeta_FE_001}), 
and the third equality follows from the substitution $t\to t^{-1}$ 
in the first term. 
The assertion follows from this equality 
and Lemma \ref{lem:par_zeta_converge}. 
\end{proof}

\begin{proof}[Proof of Proposition \ref{prop:DS_FE}]
Let $f\in C^\infty_0(\bR^\times )$. 
By Lemma \ref{lem:AC_global_zeta}, we have 
\begin{equation}
\begin{aligned}
Z(\alpha_{1},\cF (f);s)
=\,
&Z_+(\alpha_{1},\cF (f);s)
+Z_+(\alpha_{2},\cF (f_{\mu,\nu,\infty});-s-2\nu +1)\\
&
\label{eqn:pf_zeta_FE_002}
-\frac{\alpha_{1}(0)\cF (f)(0)}{s}
+\frac{\alpha_{2}(0)\cF (f_{\mu,\nu,\infty})(0)}{s+2\nu -1}
\end{aligned}
\end{equation}
for $s\in \bC$ such that $\mathrm{Re}(s)$ is sufficiently large. 
Since the first 2 terms on the right-hand side of 
(\ref{eqn:pf_zeta_FE_002}) are entire functions 
by Lemma \ref{lem:par_zeta_converge} (i), 
the expression (\ref{eqn:pf_zeta_FE_002})
gives the meromorphic continuation of 
$Z(\alpha_{1},\cF (f);s)$ to the whole $s$-plane. 
By Lemma \ref{lem:zeta_converge} and 
Proposition \ref{prop:LFE}, we have 
\begin{align}
\label{eqn:pf_zeta_FE_003}
&Z(\alpha_{1},\cF (f);s)
=\left(\,\Phi_{1}(f;-s+1),\,\Phi_{-1}(f;-s+1)\,\right){\mathrm{E}}(s)
\left(\begin{array}{c}
\Xi_+(\alpha_{1} ;s)\\
\Xi_-(\alpha_{1} ;s)
\end{array}\right)
\end{align}
for $s\in \bC$ such that $\mathrm{Re}(s)$ is sufficiently large. 
Let $s_0\in \bC$. For $\varepsilon \in \{\pm 1\}$, we define 
the function $\Delta_{s_0,\varepsilon}$ in $C^\infty_0(\bR^\times)$ by  
\begin{align*}
&\Delta_{s_0,\varepsilon}(x)
=|x|^{s_0}\Delta (\varepsilon x-3)&
&(x\in \bR^\times ), 
\end{align*}
where $\Delta$ is the function defined in \S \ref{subsec:test_ftn}. 
Then 
$\displaystyle \Phi_1(\Delta_{s_0,1};-s_0+1)>0$ and 
\begin{align}
\label{eqn:pf_zeta_FE_006}
&\left(\begin{array}{cc}
\Phi_1(\Delta_{s_0,1};s)&
\Phi_{-1}(\Delta_{s_0,1};s)\\
\Phi_1(\Delta_{s_0,-1};s)&
\Phi_{-1}(\Delta_{s_0,-1};s)
\end{array}\right)
=\Phi_1(\Delta_{s_0,1};s)1_2
\end{align}
for $s\in \bC$. 
Hence, we have  
\begin{align*}
\Phi_1(\Delta_{s_0,1};-s+1)\,
{\mathrm{E}}(s)
\left(\begin{array}{c}
\Xi_+(\alpha_{1} ;s)\\
\Xi_-(\alpha_{1} ;s)
\end{array}\right)
=\left(\begin{array}{c}
Z(\alpha_{1},\cF (\Delta_{s_0,1});s)\\
Z(\alpha_{1},\cF (\Delta_{s_0,-1});s)
\end{array}\right),
\end{align*}
and $\Phi_1(\Delta_{s_0,1};-s+1)\neq 0$ 
on a sufficiently small neighborhood of $s_0$. 
Since $s_0$ can be chosen arbitrarily, 
we obtain the meromorphic continuations 
of $\Xi_+(\alpha_{1} ;s)$ and $\Xi_-(\alpha_{1} ;s)$ 
to the whole $s$-plane. 
Moreover, since 
\begin{align*}
&\cF (\Delta_{0,1})(0)=\cF (\Delta_{0,-1})(0)
=\Phi_1(\Delta_{0,1};1),\\
&\cF ((\Delta_{-2\nu +1,\varepsilon })_{\mu ,\nu ,\infty})(0)
=e^{\varepsilon \pi \sI \mu /2}\Phi_1(\Delta_{-2\nu +1,1};2\nu )&
&(\varepsilon \in \{\pm 1\}),
\end{align*}
we know that (\ref{eqn:DS_entire}) is entire.

Since $(T_{\alpha_1},T_{\alpha_2})\in \mathcal{A} ({L}_1,{L}_2;J_{\mu ,\nu})$ 
if and only if 
$(T_{\alpha_2},T_{\alpha_1})\in \mathcal{A} ({L}_2,{L}_1;J_{\mu ,\nu})$, 
by Lemma \ref{lem:AC_global_zeta}, we note that 
$Z(\alpha_{2},\cF (f_{\mu,\nu,\infty});-s-2\nu +1)$ is equal to 
the right-hand side of (\ref{eqn:pf_zeta_FE_002}). 
Hence we obtain the equality
\begin{align}
\label{eqn:pf_zeta_FE_004}
Z(\alpha_{1},\cF (f);s)
=Z(\alpha_{2},\cF (f_{\mu,\nu,\infty});-s-2\nu +1). 
\end{align}
By Lemmas \ref{lem:LFE_infty}, \ref{lem:zeta_converge} and 
Proposition \ref{prop:LFE}, we have 
\begin{align}
\label{eqn:pf_zeta_FE_005}
\begin{split}
&Z(\alpha_{2},\cF (f_{\mu ,\nu ,\infty});s)\\
&=(\Phi_{1}(f;s+2\nu ),\Phi_{-1}(f;s+2\nu ))\Sigma_{\mu}
{\mathrm{E}}(s)
\left(\begin{array}{c}
\Xi_+(\alpha_{2} ;s)\\
\Xi_-(\alpha_{2} ;s)
\end{array}\right). 
\end{split}
\end{align}
By (\ref{eqn:pf_zeta_FE_003}), 
(\ref{eqn:pf_zeta_FE_004}) and (\ref{eqn:pf_zeta_FE_005}) 
with $f=\Delta_{s_0,\varepsilon}$ 
($s_0\in \bC$, $\varepsilon \in \{\pm 1\}$) and (\ref{eqn:pf_zeta_FE_006}), 
we obtain the functional equation in {\normalfont [D1]}. 
\end{proof}

\subsection{The estimate on vertical strips}
\label{subsec:est_DS}

In this subsection, we complete a proof of Theorem \ref{thm:DS} (i).  
We use the test functions $\varphi_{\varepsilon ,\sigma}\in C^\infty (\bR)$ 
and $\varphi_{\varepsilon ,\sigma ,u}\in C^\infty_0(\bR^\times )$ 
in \S \ref{subsec:test_ftn}. 

\begin{lem}
\label{lem:test_Z_ACexp}
We use the notation in Proposition \ref{prop:DS_FE}. 
Let $r$ be a real number such that 
$\alpha_{1}\in \mathfrak{M}_r(L_1)$ and 
$\alpha_{2}\in \mathfrak{M}_r(L_2)$. 
Let $\varepsilon \in \{\pm 1\}$, and 
take $\sigma >1$ and $n\in \bZ$ so that 
$\sigma -2\mathrm{Re}(\nu )>2n
\geq 2\max\{r+1,0\}$. 
For $s\in \bC$ such that 
$\mathrm{Re}(s)>\max \{n,-n-2\mathrm{Re}(\nu )+1\}$, 
we have 
\begin{align*}
Z(\alpha_{1},\cF (\varphi_{\varepsilon ,\sigma});s)
=
&Z_+(\alpha_{1},\cF (\varphi_{\varepsilon ,\sigma});s)
+Z_+(\alpha_{2},
\cF ((\varphi_{\varepsilon ,\sigma})_{\mu,\nu,\infty});
-s-2\nu +1)\\
&-\frac{\alpha_{1}(0)\cF (\varphi_{\varepsilon ,\sigma})(0)}{s}
+\frac{\alpha_{2}(0)
\cF ((\varphi_{\varepsilon ,\sigma})_{\mu,\nu,\infty})(0)}{s+2\nu -1}.
\end{align*}
\end{lem}
\begin{proof}
By (\ref{eqn:Fourier_phies_estimate}), 
Corollary \ref{cor:test_Shintani3} and Lemma \ref{lem:par_zeta_converge}, 
we know that the integrals 
$Z_+(\alpha_{1},\cF (\varphi_{\varepsilon ,\sigma});s)$, 
$Z_-(\alpha_{1},\cF (\varphi_{\varepsilon ,\sigma});s)$ and 
$Z_+(\alpha_{2},\cF ((\varphi_{\varepsilon ,\sigma})_{\mu,\nu,\infty});s)$ 
converge absolutely and define holomorphic functions 
on the whole $s$-plane, $\mathrm{Re}(s)>n$ and $\mathrm{Re}(s)<n$, 
respectively. 
Let $u>1$. 
Because of the proof of Lemma \ref{lem:AC_global_zeta}, 
for $s\in \bC$ such that $\mathrm{Re}(s)>\max\{r+1,0\}$, 
we have 
\begin{align}
\label{eqn:pf_test_Z_ACexp}
\begin{split}
Z_-(\alpha_{1},\cF (\varphi_{\varepsilon ,\sigma ,u});s)
=
&Z_+(\alpha_{2},
\cF ((\varphi_{\varepsilon ,\sigma ,u})_{\mu,\nu,\infty});-s-2\nu +1)\\
&-\frac{\alpha_{1}(0)\cF (\varphi_{\varepsilon ,\sigma ,u})(0)}{s}
+\frac{\alpha_{2}(0)
\cF ((\varphi_{\varepsilon ,\sigma ,u})_{\mu,\nu,\infty})(0)}{s+2\nu -1}.
\end{split}
\end{align}
Let $s\in \bC$ such that $\mathrm{Re}(s)>\max \{n,-n-2\mathrm{Re}(\nu )+1\}$. 
Then we have 
\begin{align*}
\nonumber 
&\biggl| Z_-(\alpha_{1},\cF (\varphi_{\varepsilon ,\sigma} );s)
-Z_+(\alpha_{2},
\cF ((\varphi_{\varepsilon ,\sigma})_{\mu,\nu,\infty});-s-2\nu +1)\\
&
+\frac{\alpha_{1}(0)\cF (\varphi_{\varepsilon ,\sigma})(0)}{s}
-\frac{\alpha_{2}(0)
\cF ((\varphi_{\varepsilon ,\sigma})_{\mu,\nu,\infty})(0)}{s+2\nu -1}\biggr|\\
&=\biggl| Z_-(\alpha_{1},\cF (\varphi_{\varepsilon ,\sigma} 
-\varphi_{\varepsilon ,\sigma,u});s)
-Z_+(\alpha_{2},\cF ((\varphi_{\varepsilon ,\sigma}
-\varphi_{\varepsilon ,\sigma,u})_{\mu,\nu,\infty});-s-2\nu +1)\\
&\hphantom{=.}
+\frac{\alpha_{1}(0)\cF (\varphi_{\varepsilon ,\sigma} 
-\varphi_{\varepsilon ,\sigma,u})(0)}{s}
-\frac{\alpha_{2}(0)\cF ((\varphi_{\varepsilon ,\sigma}
-\varphi_{\varepsilon ,\sigma,u})_{\mu,\nu,\infty})(0)}
{s+2\nu -1}\biggr|\\
&\leq 
\frac{C_{\alpha_1,n}}{\mathrm{Re}(s)-n }
\underset{\ x\in \bR^\times \hspace{-1mm}}{\sup}
(|x|^{n}|\cF (\varphi_{\varepsilon ,\sigma} 
-\varphi_{\varepsilon ,\sigma,u})(x)|)\\
&\hphantom{=.}
+\frac{C_{\alpha_2,{n}}}{n -\mathrm{Re}(-s-2\nu +1)}
\underset{\ x\in \bR^\times \hspace{-1mm}}{\sup}
(|x|^{n}|\cF ((\varphi_{\varepsilon ,\sigma}
-\varphi_{\varepsilon ,\sigma,u})_{\mu,\nu,\infty})(x)|)\\
&\hphantom{=.}
+\frac{|\alpha_{1}(0)|
}{|s|}|\cF (\varphi_{\varepsilon ,\sigma} 
-\varphi_{\varepsilon ,\sigma,u})(0)|
+\frac{|\alpha_{2}(0)|}{|s+2\nu -1|}
|\cF ((\varphi_{\varepsilon ,\sigma}
-\varphi_{\varepsilon ,\sigma,u})_{\mu,\nu,\infty})(0)|\\
&\leq 
\frac{C_{\alpha_1,n}c_{\sigma ,n}}{\mathrm{Re}(s)-n }u^{1-\sigma }
+\frac{C_{\alpha_2,{n}}c_{\mu ,\nu ,\sigma ,n}}{n -\mathrm{Re}(-s-2\nu +1)}
u^{-(\sigma -2\mathrm{Re}(\nu )-2n)}\\
&\hphantom{=.}
+\frac{|\alpha_{1}(0)|
}{|s|}c_{\sigma ,0}u^{1-\sigma }
+\frac{|\alpha_{2}(0)|}{|s+2\nu -1|}
c_{\mu ,\nu ,\sigma ,0}u^{-(\sigma -2\mathrm{Re}(\nu ))}.
\end{align*}
Here the first equality follows from (\ref{eqn:pf_test_Z_ACexp}), 
the first inequality follows from Lemma \ref{lem:par_zeta_converge}, 
and the second inequality follows from 
Propositions \ref{prop:test_Shintani1} 
and \ref{prop:test_Shintani2}. 
Since this inequality holds for any $u>1$,
we have 
\begin{align*}
\nonumber 
Z_-(\alpha_{1},\cF (\varphi_{\varepsilon ,\sigma} );s)
=&Z_+(\alpha_{2},
\cF ((\varphi_{\varepsilon ,\sigma})_{\mu,\nu,\infty});-s-2\nu +1)\\
&-\frac{\alpha_{1}(0)\cF (\varphi_{\varepsilon ,\sigma})(0)}{s}
+\frac{\alpha_{2}(0)
\cF ((\varphi_{\varepsilon ,\sigma})_{\mu,\nu,\infty})(0)}{s+2\nu -1}. 
\end{align*}
By this equality, we obtain the assertion. 
\end{proof}

\begin{proof}[Proof of Theorem \ref{thm:DS} (i)]
Because of Proposition \ref{prop:DS_FE}, it suffices to prove that 
the meromorphically continued function $\Xi_\pm (\alpha_1;s)$ satisfies 
the condition {\normalfont [D2-1]} in \S \ref{subsec:DS_AP}. 
Let $\sigma_1,\sigma_2\in \bR $ such that 
$\sigma_1<\sigma_2$. 
We take $r\in \bR$ so that $\alpha_{1}\in \mathfrak{M}_r(L_1)$ 
and $\alpha_{2}\in \mathfrak{M}_r(L_2)$. 
Let $\varepsilon \in \{\pm 1\}$. 
We take 
$\sigma >1$ and $n\in \bZ$ so that 
$\sigma -2\mathrm{Re}(\nu )>2n
\geq 2\max\{r+1,0\}$ and 
$\max \{ 1-\sigma ,\, -n-2\mathrm{Re}(\nu )+1\}<\sigma_1$.

By Lemma \ref{lem:zeta_converge} and Lemma \ref{lem:LFE_test}, 
for $s\in \bC$ such that $\mathrm{Re}(s)>n$, we have 
\begin{align*}
&\nonumber 
Z(\alpha_{1},\cF (\varphi_{\varepsilon ,\sigma});s)
=(\Phi_{1}(\cF (\varphi_{\varepsilon ,\sigma});s),\ 
\Phi_{-1}(\cF (\varphi_{\varepsilon ,\sigma});s))
\left(\begin{array}{c}
\xi_+(\alpha_{1} ;s)\\
\xi_-(\alpha_{1} ;s)
\end{array}\right)\\
&=\left(\,\Phi_{1}(\varphi_{\varepsilon ,\sigma};-s+1),\,
\Phi_{-1}(\varphi_{\varepsilon ,\sigma};-s+1)\,\right){\mathrm{E}}(s)
\left(\begin{array}{c}
\Xi_+(\alpha_{1} ;s)\\
\Xi_-(\alpha_{1} ;s)
\end{array}\right).
\end{align*}
Since 
$\Phi_{1}(\varphi_{1,\sigma};s)=\Phi_{-1}(\varphi_{-1,\sigma};s)$ 
and $\Phi_{-1}(\varphi_{1,\sigma};s)=\Phi_{1}(\varphi_{-1,\sigma};s)=0$, 
we have 
\begin{align}
\label{eqn:pf_AP_to_DS005}
\left(\begin{array}{c}
\Xi_+(\alpha_{1} ;s)\\
\Xi_-(\alpha_{1} ;s)
\end{array}\right)
=
\frac{1}{\Phi_{1}(\varphi_{1,\sigma};-s+1)}{\mathrm{E}}(s)^{-1}
\left(\begin{array}{c}
Z(\alpha_{1},\cF (\varphi_{1,\sigma});s)\\
Z(\alpha_{1},\cF (\varphi_{-1,\sigma});s)
\end{array}\right).
\end{align}

By Lemma \ref{lem:test_Z_ACexp}, we have 
\begin{align*}
\begin{split}
Z(\alpha_{1},\cF (\varphi_{\varepsilon ,\sigma});s)
=
&Z_+(\alpha_{1},\cF (\varphi_{\varepsilon ,\sigma});s)
+Z_+(\alpha_{2},
\cF ((\varphi_{\varepsilon ,\sigma})_{\mu ,\nu,\infty});
-s-2\nu +1)\\
&-\frac{\alpha_{1}(0)\cF (\varphi_{\varepsilon ,\sigma})(0)}{s}
+\frac{\alpha_{2}(0)
\cF ((\varphi_{\varepsilon ,\sigma})_{\mu ,\nu,\infty})(0)}{s+2\nu -1} 
\end{split}
\end{align*}
for $s\in \bC$ such that 
$\mathrm{Re}(s)>\max \{n,-n-2\mathrm{Re}(\nu )+1\}$. 
This expression gives the meromorphic continuation of 
$Z(\alpha_{1},\cF (\varphi_{\varepsilon ,\sigma});s)$ 
to $\mathrm{Re}(s)>-n-2\mathrm{Re}(\nu )+1$, 
and the inequality  
\begin{align*}
\left|Z(\alpha_{1},\cF (\varphi_{\varepsilon ,\sigma});s)\right|
%
&\leq 
\int_1^\infty \left|\vartheta ({\alpha_{1} ,
\cF (\varphi_{\varepsilon ,\sigma})};t)\right|
t^{{\sigma}_2}\,\frac{dt}{t}\\
&\hphantom{=.}
+\int_1^\infty 
\left|\vartheta(\alpha_{2},
\cF ((\varphi_{\varepsilon ,\sigma})_{\mu ,\nu,\infty});t)\right|
t^{-{\sigma}_1-2\mathrm{Re}(\nu) +1}\,\frac{dt}{t}\\
&\hphantom{=.}
+\left| \frac{\alpha_{1}(0)\cF (\varphi_{\varepsilon ,\sigma})(0)}{s}\right|
+\left| \frac{\alpha_{2}(0)
\cF ((\varphi_{\varepsilon ,\sigma})_{\mu ,\nu,\infty})(0)}{s+2\nu -1}\right| 
\end{align*}
for $s\in \bC$ such that ${\sigma}_1\leq \mathrm{Re}(s)\leq {\sigma}_2$. 
Hence, we have 
\begin{align*}
&Z(\alpha_{1},\cF (\varphi_{\varepsilon ,\sigma});s)
=O(1)&&(|s|\to \infty)&
&\text{uniformly on \ $\sigma_1\leq \mathrm{Re}(s)\leq \sigma_2$}.
\end{align*}
By (\ref{eqn:est_LZ_test}), (\ref{eqn:pf_AP_to_DS005}) and 
this estimate, we obtain the assertion. 
\end{proof}

\subsection{The converse theorem}
\label{subsec:converse}

In this subsection, 
we give a proof of Theorem \ref{thm:DS} (ii). 

\begin{lem}
\label{lem:cdn_zeta_converse}
We use the notation in Theorem \ref{thm:DS} (ii). 
Then, for any $f\in C^\infty_0(\bR^\times)$, 
the zeta integrals $Z(\alpha_{1},\cF (f);s)$ and 
$Z(\alpha_2,\cF (f_{\mu,\nu,\infty});s)$ satisfy 
the following conditions {\normalfont [Z1]} and {\normalfont [Z2]}: 
\begin{description}
\item[{\normalfont [Z1]}]
The zeta integral $Z(\alpha_{1},\cF (f);s)$ 
has the meromorphic continuation 
to the whole $s$-plane, and the function 
\begin{align}
\label{eqn:zeta_entire}
Z(\alpha_{1},\cF (f);s)+\frac{\alpha_{1}(0)\cF (f)(0)}{s}
-\frac{\alpha_{2}(0)\cF (f_{\mu,\nu,\infty})(0)}{s+2\nu -1}
\end{align}
is entire. 
Moreover, $Z(\alpha_{1},\cF (f);s)$ and 
$Z(\alpha_2,\cF (f_{\mu,\nu,\infty});s)$ satisfy 
the functional equation 
\begin{align}
\label{eqn:zeta_FE}
Z(\alpha_{1},\cF (f);s)
=Z(\alpha_{2},\cF (f_{\mu,\nu,\infty});-s-2\nu+1).
\end{align}
\item[{\normalfont [Z2]}]For any ${\sigma}_1,{\sigma}_2\in \bR$ 
such that $\sigma_1<\sigma_2$, 
there is some $c_0\in \bR_{>0}$ such that 
\begin{align*}
&Z(\alpha_{1},\cF (f);s)=O({e}^{|s|^{c_0}})&
&(|s|\to \infty )&
&\text{uniformly on ${\sigma}_1\leq \mathrm{Re}(s)\leq {\sigma}_2$}. 
\end{align*} 
\end{description}
\end{lem}
\begin{proof}
By Lemmas \ref{lem:LFE_infty}, \ref{lem:zeta_converge} and 
Proposition \ref{prop:LFE}, 
we have 
\begin{align*}
&Z(\alpha_{1},\cF (f);s)
=\left(\,\Phi_{1}(f;-s+1),\,\Phi_{-1}(f;-s+1)\,\right){\mathrm{E}}(s)
\left(\begin{array}{c}
\Xi_+(\alpha_{1};s)\\
\Xi_-(\alpha_{1};s)
\end{array}\right),\\
&Z(\alpha_{2},\cF (f_{\mu ,\nu ,\infty});s)
=(\Phi_{1}(f;s+2\nu ),\Phi_{-1}(f;s+2\nu ))\Sigma_{\mu}
{\mathrm{E}}(s)
\left(\begin{array}{c}
\Xi_+(\alpha_{2};s)\\
\Xi_-(\alpha_{2};s)
\end{array}\right) 
\end{align*}
for $s\in \bC$ with a sufficiently large real part. 
Using these equalities, Lemma \ref{lem:LFE_infty} and 
$\cF (f)(0)=\Phi_{1}(f;1)+\Phi_{-1}(f;1)$, 
we obtain {\normalfont [Z1]} and {\normalfont [Z2]} 
from the conditions {\normalfont [D1]} and {\normalfont [D2-2]} 
in \S \ref{subsec:DS_AP}. 
\end{proof}

We recall the Mellin inversion formula. 
For a continuous function $\vartheta$ on $\bR_{>0}$, 
we define the Mellin transform $\cM (\vartheta )$ by 
\begin{align*}
&\cM (\vartheta )(s)=\int_0^\infty \vartheta (t)t^{s-1}dt&
&(s\in \bC).
\end{align*}

\begin{lem}[The Mellin inversion formula]
\label{lem:Mellin_inv}
Let $\vartheta$ be a continuous function on $\bR_{>0}$. 
Assume that there exists $\sigma \in \bR$ such that 
\begin{align}
\label{eqn:cdn_Mellin_001}
&\int_0^\infty |\vartheta (t)|t^{\sigma -1}\,dt<\infty ,&
&\int_{\mathrm{Re}(s)=\sigma }|\cM (\vartheta )(s)|\,ds<\infty .
\end{align}
Here the path of the integration $\int_{\mathrm{Re}(s)=\sigma }$ 
is the vertical line from $\sigma -\sI \infty$ to $\sigma +\sI \infty$. 
Then we have $\displaystyle \vartheta (t)=\frac{1}{2\pi \sI}
\int_{\mathrm{Re}(s)=\sigma}\cM (\vartheta )(s)t^{-s}\,ds$ 
$(t>0)$. 
\end{lem}
\begin{proof}
We derive this lemma from the Fourier inversion formula 
(See, for example, \cite[\S 1.5]{Bump_004}). 
\end{proof}

\begin{lem}
\label{lem:cthm_junbi}
Let ${L}$ be a shifted lattice in $\bR$. 
Let $r\in \bR $, $\alpha \in {\mathfrak{M}}_r({L})$, $f\in \cS (\bR)$ 
and ${\sigma}>\max \{0,r+1\}$. 
Then 
\begin{align}
\label{eqn:cthm_junbi_estimate}
&Z(\alpha ,f;s)=O(|s|^{-n})\ \ (|s|\to \infty)\ \ 
\text{ on \ $\mathrm{Re}(s)=\sigma$}
\end{align}
for any $n\in \bZ_{\geq 0}$. Moreover, we have 
\begin{align}
\label{eqn:cthm_junbi_Mellin}
&\vartheta (\alpha ,f;t)=\frac{1}{2\pi \sI}
\int_{\mathrm{Re}(s)=\sigma}Z(\alpha ,f;s)t^{-s}\,ds&
&(t>0).
\end{align}
\end{lem}
\begin{proof}
We prove (\ref{eqn:cthm_junbi_estimate}) by induction with respect to $n$. 
Let $\sigma >\max \{0,r+1\}$. 
When $n=0$, the estimate (\ref{eqn:cthm_junbi_estimate}) holds since 
\begin{align*}
|Z(\alpha ,f;s)|\leq \int_{0}^\infty 
\left| \vartheta ({\alpha ,f};t)\right|t^{\sigma }\,\frac{dt}{t} 
\end{align*}
for $s\in \bC$ such that $\mathrm{Re}(s)=\sigma$. 
If $\mathrm{Re}(s)=\sigma$, we have 
\begin{align}
\nonumber 
Z(\alpha ,f;s)
=&\int_0^\infty \sum_{0\neq l\in {L}}\alpha (l)f(tl)t^{s-1}\,dt
=\int_0^\infty \sum_{0\neq l\in {L}}\alpha (l)f(tl)
\left(\frac{d}{dt}\frac{t^{s}}{s}\right)dt\\
\label{eqn:pf_cthm_junbi_001}
=&-\int_0^\infty 
\sum_{0\neq l\in {L}}\alpha (l)lf'(tl)
\frac{t^{s}}{s}\,dt
=-\frac{1}{s}Z({\alpha',f'};s+1)
\end{align}
with $\alpha'\in {\mathfrak{M}}_{r+1}({L})$ defined by 
$\alpha'(l)=l\alpha (l)\ (l\in {L} )$. 
Here the third equality follows from integration by parts. 
From the expression (\ref{eqn:pf_cthm_junbi_001}), we know that 
(\ref{eqn:cthm_junbi_estimate}) for $n=n_0+1$ holds 
if (\ref{eqn:cthm_junbi_estimate}) for $n=n_0$ holds. 
This complete the proof of (\ref{eqn:cthm_junbi_estimate}) 
for all $n\in \bZ_{\geq 0}$. 

Because of Lemma \ref{lem:theta_modoki1} and 
the estimate (\ref{eqn:cthm_junbi_estimate}), 
we note that $\vartheta (t)=\vartheta (\alpha ,f;t)$ satisfies 
the condition (\ref{eqn:cdn_Mellin_001}). 
Since $\cM (\vartheta )(s)=Z(\alpha ,f;s)$, 
we obtain (\ref{eqn:cthm_junbi_Mellin}) by 
the Mellin inversion formula (Lemma \ref{lem:Mellin_inv}). 
\end{proof}

\begin{lem}[The Phragmen--Lindel\"of theorem]
\label{lem:PL_theorem}
Let ${\sigma}_1,{\sigma}_2\in \bR$ such that ${\sigma}_1<{\sigma}_2$. 
Let $\Theta (s)$ be a holomorphic function on a domain containing 
a vertical strip $\sigma_1\leq \mathrm{Re}(s)\leq \sigma_2$. 
Assume that there are $c_0\in \bR_{>0}$ and $c_1\in \bR$ such that 
\begin{align*}
&\Theta (s)=O({e}^{|s|^{c_0}})&&(|s|\to \infty)&
&\text{uniformly on \ $\sigma_1\leq \mathrm{Re}(s)\leq \sigma_2$}, \\
&\Theta (s)=O(|s|^{c_1})&&(|s|\to \infty)& 
&\text{on \ $\mathrm{Re}(s)={\sigma}_1$ and $\mathrm{Re}(s)={\sigma}_2$}. 
\end{align*}
Then $\Theta (s)=O(|s|^{c_1})\ (|s|\to \infty)$ 
uniformly on $\sigma_1\leq \mathrm{Re}(s)\leq \sigma_2$. 
\end{lem}
\begin{proof}
See, for example, Miyake \cite[Lemma 4.3.4]{Miyake_001}. 
\end{proof}

\begin{proof}[Proof of Theorem \ref{thm:DS} (ii)]
Let $f\in C^\infty_0(\bR^\times)$. 
By Lemma \ref{lem:cdn_zeta_converse}, 
we know that $\alpha_1$, $\alpha_2$ and $f$ satisfy 
the conditions {\normalfont [Z1]} and {\normalfont [Z2]}. 
Let $\sigma $ be a sufficiently large real number. 
By Lemma \ref{lem:cthm_junbi}, we have 
\begin{align*}
&Z(\alpha_{1},\cF (f);s)=O(|s|^{-n})&&(|s|\to \infty) &
&\text{on $\mathrm{Re}(s)=\sigma$, }
\end{align*}
for any $n\in \bZ_{\geq 0}$. By the functional equation 
(\ref{eqn:zeta_FE}) in {\normalfont [Z1]}, 
we note that this estimate also holds on 
$\mathrm{Re}(s)=-\sigma -2\mathrm{Re}(\nu)+1$. 
By these estimates, {\normalfont [Z2]} and the entireness of 
(\ref{eqn:zeta_entire}) in {\normalfont [Z1]}, 
we know that 
\begin{align*}
\Theta (s)=s(s+2\nu -1)Z(\alpha_{1},\cF (f);s) 
\end{align*}
satisfies the assumption in the Phragmen--Lindel\"of theorem 
(Lemma \ref{lem:PL_theorem}) with 
$\sigma_1=-\sigma -2\mathrm{Re}(\nu)+1$, $\sigma_2=\sigma $ 
and $c_1=-n$ for any $n\in \bZ_{\geq 0}$. 
This implies that  
\begin{align*}
&Z(\alpha_{1},\cF (f);s)=O(|s|^{-n})&&(|s|\to \infty) 
\end{align*}
uniformly on $-\sigma -2\mathrm{Re}(\nu)+1\leq \mathrm{Re}(s)\leq \sigma$, 
for any $n\in \bZ_{\geq 0}$. 
Hence, by Lemma \ref{lem:cthm_junbi} and 
the shift of the contour, we have 
\begin{align}
\nonumber 
\vartheta(\alpha_{1},\cF (f);t)
=&\frac{1}{2\pi \sI}\int_{\mathrm{Re}(s)=\sigma}
Z(\alpha_{1},\cF (f);s)t^{-s}\,ds\\
\nonumber 
=&\frac{1}{2\pi \sI}\int_{\mathrm{Re}(s)=-\sigma -2\mathrm{Re}(\nu)+1}
Z(\alpha_{1},\cF (f);s)t^{-s}\,ds\\
\label{eqn:pf_cthm_zeta_001}
&-\alpha_{1}(0)\cF (f)(0)
+\alpha_{2}(0)\cF (f_{\mu,\nu,\infty})(0)t^{2\nu -1}.
\end{align}
Moreover, we have 
\begin{align}
\nonumber 
&\frac{1}{2\pi \sI}\int_{\mathrm{Re}(s)=-\sigma -2\mathrm{Re}(\nu)+1}
Z(\alpha_{1},\cF (f);s)t^{-s}\,ds\\
\nonumber 
&=\frac{1}{2\pi \sI}\int_{\mathrm{Re}(s)=-\sigma -2\mathrm{Re}(\nu)+1}
Z(\alpha_{2},\cF (f_{\mu,\nu,\infty});-s-2\nu+1)t^{-s}\,ds\\
\nonumber 
&=\frac{1}{2\pi \sI}\int_{\mathrm{Re}(s)=\sigma}
Z(\alpha_{2},\cF (f_{\mu,\nu,\infty});s)t^{s+2\nu -1}\,ds\\
\label{eqn:pf_cthm_zeta_002}
&=t^{2\nu -1}\vartheta(\alpha_{2},\cF (f_{\mu,\nu,\infty});t^{-1}). 
\end{align}
Here the first equality follows from the functional equation 
(\ref{eqn:zeta_FE}) in {\normalfont [Z1]}, 
the second equality follows from the substitution 
$s\to -s-2\nu+1$, and the third equality follows 
from Lemma \ref{lem:cthm_junbi}. 
Combining the equalities (\ref{eqn:pf_cthm_zeta_001}) 
and (\ref{eqn:pf_cthm_zeta_002}), we have
\begin{align*}
&\vartheta(\alpha_{1},\cF (f);t)+\alpha_{1}(0)\cF (f)(0)\\
&=t^{2\nu -1}
\{\vartheta(\alpha_{2},\cF (f_{\mu,\nu,\infty});t^{-1})
+\alpha_{2}(0)\cF (f_{\mu,\nu,\infty})(0)\}.
\end{align*}
This equality for $t=1$ is 
$T_{\alpha_1}(f)=T_{\alpha_2}(f_{\mu ,\nu ,\infty })$, 
and complete the proof. 
\end{proof}

\section{Distributions on $I_{\mu,\nu}^\infty$}
\label{sec:CF_ps}

\subsection{The correspondence between $\cA (J_{\mu,\nu})$ and $I_{\mu,\nu}^{-\infty}$. }
\label{subsec:pf_PFprop}

In this subsection, we give a proof of 
Proposition \ref{prop:pair_to_functionl}. 
For $j\in \bZ_{\geq 0}$, we denote by $\cU_j(\g_\bC)$ 
the subspace of $\cU(\g_\bC)$ spanned by the products of 
$j$ or less elements of $\g_{\bC}$. 
We denote by $\Ad$ the adjoint action of $G$ on $\g_{\bC}$, 
that is, $\Ad (g)X=gXg^{-1}$ $(g\in G,\ X\in \g_{\bC})$. 
We denote the action of $G$ on $\cU(\g_\bC)$ 
induced from $\Ad$ by the same symbol. 
Let $\exp \colon \g \to G$ be the exponential map, that is, 
\begin{align*}
&\exp (X)=\sum_{i=0}^\infty \frac{1}{i!}X^i&
&(X\in \g). 
\end{align*}

\begin{lem}
\label{lem:univ_exp1}
Let $X\in \g$. \\[1mm]
\noindent (i) 
$\displaystyle 
\widetilde{\exp}(X)=\bigl({}^s\!\exp \bigl(\tfrac{1}{m}X\bigr)\bigr)^m$
for a sufficiently large $m\in \bZ_{\geq 0}$. \\[1mm]
\noindent (ii) 
$\tilde{g}\,\widetilde{\exp}(X)\tilde{g}^{-1}
=\widetilde{\exp}(\Ad (g)X)$ 
for $\tilde{g}=(g,\theta )\in \widetilde{G}$.\\[1mm]
\noindent (iii) 
$\widetilde{\exp}(tH)=\tilde{\ra}(e^{2t})$, 
$\widetilde{\exp}(tE_+)=\tilde{\ru}(t)$ and 
$\widetilde{\exp}(tE_-)=\tilde{\overline{\ru}}(t)$ for $t\in \bR$. \\[1mm]
\noindent (iv) 
$\widetilde{\exp}(t(E_+-E_-))=\tilde{\rk}(t)$ for $t\in \bR$. 
\end{lem}
\begin{proof}
Let $X\in \g$ and $\tilde{g}=(g,\theta )\in \widetilde{G}$. 
Since $\bR \ni t\mapsto 
\tilde{g}\,\widetilde{\exp}(tX)\tilde{g}^{-1}\in \widetilde{G}$ is a 
continuous group homomorphism, we note that 
$\tilde{g}\,\widetilde{\exp}(tX)\tilde{g}^{-1}$ is contained in 
a neighborhood ${\mathstrut }^s\!G$ of $(1_2,0)$ 
for $t\in \bR$ such that $|t|$ is sufficiently small. 
Since 
\[
\varpi (\tilde{g}\widetilde{\exp}(tX)\tilde{g}^{-1})
=g\exp (tX)g^{-1}=\exp (t\Ad (g)X), 
\]
we have 
$\tilde{g}\,\widetilde{\exp}(tX)\tilde{g}^{-1}
={}^s\!\exp (t\Ad (g)X)$
for $t\in \bR$ such that $|t|$ is sufficiently small. 
Hence, for a sufficiently large positive integer $m$, we have 
\[
\tilde{g}\,\widetilde{\exp}(X)\tilde{g}^{-1}=
\left(\tilde{g}\,\widetilde{\exp}\!\left(\tfrac{1}{m}X\right)
\tilde{g}^{-1}\right)^m=
\left({}^s\!\exp\!\left(\tfrac{1}{m}\Ad (g)X\right)\right)^m.
\]
The statement (i) follows from this equality in the case of 
$\tilde{g}=(1_2,0)$. The statement (ii) follows from 
this equality and the statement (i). 
The statement (iii) and (iv) follow from  the statement (i). 
\end{proof}

\begin{lem}
\label{lem:partition_exist}
Let $\mu,\nu \in \bC$. 
Let $F\in I^\infty_{\mu ,\nu}$, and set 
\begin{align}
\label{eqn:partition_example}
&f_1(x)=F(\tilde{w}\tilde{\ru}(-x))\Delta (x),&
&f_2(x)=F_\infty (\tilde{w}\tilde{\ru}(-x)){\Delta}(x)&
&(x\in \bR),
\end{align}
where $\Delta$ is the function defined in \S \ref{subsec:test_ftn}. 
Then $(f_1,f_2)$ is a partition of $F$, and satisfies 
$\supp (f_i)\subset \{t\in \bR \mid |t|\leq 2\}$ $(i=1,2)$.
\end{lem}
\begin{proof}
Let  
$\tilde{g}=(g,\theta )\in \widetilde{G}$ 
with $g=\left(\begin{array}{cc}
a&b\\
c&d
\end{array}\right)\in G$. 
By definition, we have 
\begin{align*}
&\iota_{\mu,\nu }(f_1)(\tilde{g})
=\left\{\begin{array}{ll}
\displaystyle 
F(\tilde{g})\Delta (-d/c)&
\text{ if }\ c\neq 0,\\[2mm]
\displaystyle 
0&\text{ if }\ c=0,
\end{array}\right.\\
&(\iota_{\mu,\nu }(f_2))_\infty (\tilde{g})
=\left\{\begin{array}{ll}
\displaystyle 
F (\tilde{g}){\Delta}(c/d)&
\text{ if }\ d\neq 0,\\[2mm]
\displaystyle 
0&\text{ if }\ d=0.
\end{array}\right.&
\end{align*}
By $\Delta (0)=1$, 
$\Delta (x)+{\Delta}(-1/x)=1$ $(x\in \bR^\times )$, 
$\supp (\Delta)\subset \{t\in \bR \mid |t|\leq 2\}$ and 
these equalities, we know that $(f_1,f_2)$ is a partition of $F$. 
\end{proof}

\begin{lem}
\label{lem:aimai_partitions}
Let $\mu,\nu \in \bC$. Let $(T_1,T_2)\in \cA (J_{\mu,\nu})$. 
Let $(f_{1},f_{2})$ and $(f_{3},f_{4})$ 
be two partitions of $F\in I^\infty_{\mu,\nu}$. 
Then $T_1(f_1)+T_2(f_2)=T_1(f_3)+T_2(f_4)$. 
\end{lem}
\begin{proof}
By (\ref{eqn:iota_characterization}), (\ref{eqn:pre_infty_change}) 
and the definition of partitions, 
we have 
\begin{align*}
&F(\tilde{w}\tilde{\ru}(-x))=f_1(x)+(f_2)_{\mu,\nu,\infty}(x)
=f_3(x)+(f_4)_{\mu,\nu,\infty}(x)&&(x\in \bR).
\end{align*}
Hence, $f_{1}-f_{3}=-(f_{2}-f_{4})_{\mu,\nu,\infty}$. 
This equality implies $f_2-f_4\in C^\infty_0(\bR^\times )$ since 
$f_1,f_2,f_3,f_4\in C^\infty_0(\bR)$. 
Therefore, by (\ref{eqn:autom_pair_def}), we have 
\begin{align*}
&(T_1(f_{1})+T_2(f_{2}))-(T_1(f_{3})+T_2(f_{4}))\\
&=T_1(f_{1}-f_{3})+T_2(f_{2}-f_{4})
=T_1(f_{1}-f_{3}+(f_{2}-f_{4})_{\mu,\nu,\infty})=T_1(0)=0. 
\end{align*}
Hence, we obtain the assertion. 
\end{proof}

\begin{lem}
\label{lem:norm_K_to_N}
Let $\mu,\nu \in \bC$ and $u>0$. 
For $F\in I^\infty_{\mu,\nu }$, 
the inequalities 
\begin{align}
\label{eqn:Knorm_to_N}
\sup_{-u\leq x\leq u}
|F(\tilde{w}\tilde{\ru}(-x))|
&\leq (1+u^2)^{|\mathrm{Re}(\nu)+1/2|}|F|_K,\\
\label{eqn:Knorm_infty}
|F_\infty|_K&\leq e^{{\pi}|\mathrm{Im}(\mu)|/{2}}|F|_K
\end{align}
hold. For $f\in C^\infty_0(\bR)$ with 
$\supp (f)\subset \{x\in \bR \mid |x|\leq u\}$, 
we have 
\begin{align}
\label{eqn:Nnorm_K}
&|\iota_{\mu,\nu}(f)|_K\leq 
e^{\pi |\mathrm{Im}(\mu )|}(1+u^2)^{|\mathrm{Re}(\nu)+1/2|}
\sup_{x\in \bR}|f(x)|.
\end{align}
\end{lem}
\begin{proof}
Let $F\in I_{\mu,\nu}^\infty$. Since 
\begin{align*}
\tilde{w}\tilde{\ru}(-x)=
\tilde{\ru}\!\left(\frac{x}{1+x^2}\right)
\tilde{\ra}\!\left(\frac{1}{1+x^2}\right)
\tilde{\rk}(-\arg (\sI -x)),
\end{align*}
we have $F(\tilde{w}\tilde{\ru}(-x))=(1+x^2)^{-\nu-1/2}
F(\tilde{\rk}(-\arg (\sI -x)))$ for $x\in \bR$. 
By this equality, we obtain (\ref{eqn:Knorm_to_N}).  
The inequality (\ref{eqn:Knorm_infty}) follows from 
\begin{align*}
&F_\infty (\tilde{\rk}(\theta))
=e^{\pi \sI \mu /2}F(\tilde{\rk}(\theta -\pi /2))
=e^{-\pi \sI \mu /2}F(\tilde{\rk}(\theta +\pi /2))&
&(\theta \in \bR).
\end{align*}
The inequality (\ref{eqn:Nnorm_K}) follows from 
\begin{align*}
&\iota_{\mu,\nu }(f)(\tilde{\rk}(\theta))
=\left\{\begin{array}{ll}
\displaystyle 
e^{\pi \sI \mu }
|\sin \theta |^{-2\nu -1}f\!\left(1/\tan \theta \right)
&\text{ if }\ 0<\theta <\pi,\\[2mm]
|\sin \theta |^{-2\nu -1}f\!\left(1/\tan \theta \right)
&\text{ if }-\pi <\theta <0,\\[2mm]
\displaystyle 
0&\text{ if }\ \theta \in \pi \bZ
\end{array}\right.\\
&\frac{1}{1+u^2}\leq \sin^2\theta \leq 1\quad 
\text{if }\left|\frac{1}{\tan \theta}\right|\leq u,
\end{align*}
for $\theta \in \bR$ and $f\in C^\infty_0(\bR)$. 
\end{proof}
\begin{lem}
\label{lem:gact_fct_on_R}
Let $\mu,\nu \in \bC$. 

\noindent (i) 
$\displaystyle 
(\rho (E_+)F)(\tilde{w}\tilde{\ru}(-x))
=-\frac{d}{dx}F(\tilde{w}\tilde{\ru}(-x))$ 
$(x\in \bR)$ for $F\in I^\infty_{\mu,\nu}$. \\[2mm]
\noindent (ii) 
Let $f\in C^\infty_0(\bR)$. Then 
$\rho (E_+)\iota_{\mu,\nu}(f)=-\iota_{\mu,\nu}(f')$ and 
\begin{align*}
&\rho (H)\iota_{\mu,\nu}(f)=\iota_{\mu,\nu}(f_H)&
&\text{with}\ f_H(x)=-2xf'(x)-(2\nu +1)f(x),\\
&\rho (E_-)\iota_{\mu,\nu}(f)=
\iota_{\mu,\nu}(f_{E_-})&
&\text{with}\ f_{E_-}(x)=x^2f'(x)+(2\nu +1)xf(x).
\end{align*}
\end{lem}
\begin{proof}
The statement (i) follows immediately from Lemma \ref{lem:univ_exp1} (iii).
Let $x\in \bR$. 
By the statement (i) and (\ref{eqn:iota_characterization}), we have 
$(\rho (E_+)\iota_{\mu,\nu}(f))(\tilde{w}\tilde{\ru}(-x))=-f'(x)$. 
Hence, we have $\rho (E_+)\iota_{\mu,\nu}(f)=-\iota_{\mu,\nu}(f')$ 
by the characterization (\ref{eqn:iota_iff}). 
By Lemma \ref{lem:univ_exp1} (iii) and (\ref{eqn:Bruhat_g}), 
we have 
\begin{align*}
&\tilde{w}\tilde{\ru}(-x)\,\widetilde{\exp}(tE_-)=
\tilde{w}\tilde{\ru}(-x)\tilde{\overline{\ru}}(t)
=\tilde{\ru}\!\left(\frac{-t}{1-tx}\right)
\tilde{\ra}\!\left(\frac{1}{(1-tx)^2}\right)
\tilde{w}\,\tilde{\ru}\!\left(\frac{-x}{1-tx}\right)
\end{align*}
for $t\in \bR$ such that $|t|$ is sufficiently small. 
Using this equality, we have 
\begin{align*}
(\rho (E_-)\iota_{\mu,\nu}(f))(\tilde{w}\tilde{\ru}(-x))
&=\frac{d}{dt}\bigg|_{t=0}
\left((1-tx)^{-2\nu -1}f\left(\frac{x}{1-tx}\right)\right)\\
&=x^2f'(x)+(2\nu +1)xf(x)=f_{E_-}(x). 
\end{align*}
Hence, we have  
$\rho (E_-)\iota_{\mu,\nu}(f)=\iota_{\mu,\nu}(f_{E_-})$ 
by the characterization (\ref{eqn:iota_iff}). 
We obtain $\rho (H)\iota_{\mu,\nu}(f)=\iota_{\mu,\nu}(f_{H})$ 
by the formulas for $\rho (E_+)$ and $\rho (E_-)$ 
with $\rho (H)=\rho (E_+)\circ \rho (E_-)-\rho (E_-)\circ \rho (E_+)$. 
\end{proof}

\begin{lem}
\label{lem:seminorm_Finfty_leq_F}
Let $\mu,\nu \in \bC$. 
Then $\cQ_X(F_\infty)
\leq e^{{\pi}|\mathrm{Im}(\mu)|/{2}}\cQ_{\Ad(w^{-1})X}(F)$ for 
$F\in I^\infty_{\mu,\nu}$ and $X\in U(\g_{\bC})$. 
\end{lem}
\begin{proof}
The assertion follows immediately from 
(\ref{eqn:Knorm_infty}) and Lemma \ref{lem:univ_exp1} (ii). 
\end{proof}

\begin{lem}
\label{lem:seminorm_f_leq_F}
Let $\mu,\nu \in \bC$. 
For $i\in \bZ_{\geq 0}$, 
there is a constant $c_{\Delta,\nu,i}$ such that 
\begin{align*}
&\sup_{x\in \bR}\left|\frac{d^i}{dx^i}
(F(\tilde{w}\tilde{\ru}(-x))\Delta (x))\right|\leq 
c_{\Delta,\nu,i}\sum_{j=0}^i\cQ_{(E_+)^j}(F)&
&(F\in I_{\mu,\nu}^\infty ).
\end{align*}
\end{lem}
\begin{proof}
Let $F\in I^\infty_{\mu,\nu}$ and $i\in \bZ_{\geq 0}$. 
By Lemma \ref{lem:gact_fct_on_R} (i), we have 
\begin{align}
\nonumber 
&\frac{d^i}{dx^i}
(F(\tilde{w}\tilde{\ru}(-x))\Delta (x))
=\sum_{j=0}^i\binom{i}{j}(-1)^j(\rho (E_+)^jF)(\tilde{w}\tilde{\ru}(-x))
\Delta^{(i-j)}(x)
\end{align}
for $x\in \bR $. 
By (\ref{eqn:Knorm_to_N}) and 
$\supp (\Delta )\subset \{x \in \bR \mid |x|\leq 2\}$, we have 
\begin{align*}
\sup_{x\in \bR}\left|\frac{d^i}{dx^i}
(F(\tilde{w}\tilde{\ru}(-x))\Delta (x))\right|
\leq 
&\sum_{j=0}^i\binom{i}{j}\left(\sup_{x\in \bR}\Delta^{(i-j)}(x)\right)\\
&\times 5^{|\mathrm{Re}(\nu)+1/2|}\cQ_{(E_+)^j}(F).
\end{align*}
Hence, we obtain the assertion. 
\end{proof}

\begin{lem}
\label{lem:seminorm_F_leq_f}
Let $\mu,\nu \in \bC$ and $j\in \bZ_{\geq 0}$. 
For $X\in \cU_j(\g_{\bC})$ 
and $u>0$, there is a constant $c_{X,u}>0$ such that 
\begin{align*}
\begin{split}
&\cQ_X(\iota_{\mu,\nu}(f))\leq c_{X,u}\sum_{i=0}^{j}
\sup_{x\in \bR}\bigl|f^{(i)}(x)\bigr|\\
&(f\in C^\infty_0(\bR)\ \text{with}\ 
\supp (f)\subset \{x\in \bR \mid |x|\leq u\}).
\end{split}
\end{align*}
\end{lem}
\begin{proof}
Let $X\in \cU_j(\g_{\bC})$ and $u>0$. 
Since $\{H,E_+,E_-\}$ is a basis of $\g_{\bC}$, it follows 
from Lemma \ref{lem:gact_fct_on_R} (ii) that 
there are polynomial functions $p_{X,i}$ 
$(0\leq i\leq j)$ on $\bR$ such that  
\begin{align*}
&\rho (X)\iota_{\mu,\nu}(f)=\sum_{i=0}^{j}
\iota_{\mu,\nu}(p_{X,i}f^{(i)})&
&(f\in C^\infty_{0}(\bR)).
\end{align*}
Hence, for $f\in C^\infty_0(\bR)$ 
with $\supp (f)\subset \{x\in \bR \mid |x|\leq u\}$, 
we have 
\begin{align*}
\cQ_{X}(\iota_{\mu,\nu}(f))
\leq &\,e^{\pi |\mathrm{Im}(\mu )|}(1+u^2)^{|\mathrm{Re}(\nu)+1/2|}
\sum_{i=0}^{j}\left(\sup_{-u\leq x\leq u}|p_{X,i}(x)|\right)
\sup_{x\in \bR}|f^{(i)}(x)|
\end{align*}
by (\ref{eqn:Nnorm_K}). 
The assertion follows from this inequality. 
\end{proof}

\begin{proof}[Proof of Proposition \ref{prop:pair_to_functionl}]
First, we will prove the statement (i). 
Let $\lambda\in I_{\mu,\nu}^{-\infty}$ and $u>0$. 
Then there exist $c>0$, $m\in \bZ_{>0}$ and 
$X_1,X_2, \cdots ,X_m\in \cU (\g_{\bC})$ satisfying 
(\ref{eqn:cdn_conti_ps}). 
Take $j\in \bZ_{\geq 0}$ so that 
$X_1,X_2, \cdots ,X_m\in \cU_j(\g_{\bC})$. 
By Lemma \ref{lem:seminorm_F_leq_f}, there are 
constants $c_{X_i,u}>0$ $(1\leq i\leq m)$ such that  
\begin{align*}
|T^\lambda (f)|=|\lambda (\iota_{\mu,\nu}(f))|
\leq c\sum_{i=1}^m
\cQ_{X_i}(\iota_{\mu,\nu}(f))\leq 
\left(c\sum_{i=1}^mc_{X_i,u}\right)
\sum_{i=0}^{j}
\sup_{x\in \bR}\bigl|f^{(i)}(x)\bigr|&\\
(\text{$f\in C^\infty_0(\bR)$ with 
$\supp (f)\subset \{x\in \bR \mid |x|\leq u\}$})&.
\end{align*}
This implies $T^\lambda \in \cD'(\bR)$. 
By Lemma \ref{lem:seminorm_Finfty_leq_F}, we have 
\begin{align*}
&|\lambda_\infty (F)|=|\lambda (F_\infty )|
\leq c\sum_{i=1}^m
\cQ_{X_i}(F_\infty) 
\leq ce^{{\pi}|\mathrm{Im}(\mu)|/{2}}\sum_{i=1}^m
\cQ_{\Ad(w^{-1})X_i}(F)
\end{align*}
for $F\in I_{\mu,\nu}^\infty$. This implies 
$\lambda_\infty \in I_{\mu,\nu}^{-\infty}$. 
The equality $(\lambda_\infty)_\infty =\lambda$ follows from 
$(F_\infty)_\infty =F$ $(F\in I_{\mu,\nu}^\infty)$.

Next, we will prove the statement (ii). By Lemma \ref{lem:aimai_partitions}, 
we know that the definition (\ref{eqn:pair_to_fctl}) 
does not depend on the choice of a partition. 
Let $(T_1,T_2)\in \cA (J_{\mu,\nu})$. 
Then there are $m\in \bZ_{\geq 0}$ and $c>0$ such that 
\begin{align*}
|T_1(f)|\leq c\sum_{i=0}^{m}
\sup_{x\in \bR}\left|f^{(i)}(x)\right|,
\hspace{10mm} 
|T_2(f)|\leq c\sum_{i=0}^{m}
\sup_{x\in \bR}\left|f^{(i)}(x)\right|&\\
(f\in C^\infty_0(\bR)\ \text{with}\ 
\supp (f)\subset \{x\in \bR \mid |x|\leq 2\})&.
\end{align*}
By Lemmas \ref{lem:seminorm_Finfty_leq_F} and \ref{lem:seminorm_f_leq_F}, 
there are constants $c_{\Delta,\nu,i}>0$ $(1\leq i\leq j)$ 
such that 
\begin{align*}
|\Lambda (T_1,T_2)(F)|
&=|T_1(f_1)+T_2(f_2)|\leq 
c\sum_{i=0}^{m}\left(
\sup_{x\in \bR}\left|f_1^{(i)}(x)\right|
+\sup_{x\in \bR}\left|f_2^{(i)}(x)\right|\right)\\
&\leq 
c\sum_{i=0}^{m}
c_{\Delta,\nu,i}\sum_{j=0}^i
\left(\cQ_{(E_+)^j}(F)
+e^{{\pi}|\mathrm{Im}(\mu)|/{2}}
\cQ_{(-E_-)^j}(F)\right)
\end{align*}
for $F\in I^{\infty}_{\mu,\nu}$ 
with the partition $(f_1,f_2)$ of $F$ defined by 
(\ref{eqn:partition_example}). 
This implies 
$\Lambda (T_1,T_2) \in I_{\mu,\nu}^{-\infty}$. 

Finally, we will prove the statement (iii). 
By definition, we have 
\begin{align*}
&\Lambda (T^{\lambda },T^{\lambda_\infty})=\lambda,&
&T^{\Lambda (T_1,T_2)}=T_1,&
&(\Lambda (T_1,T_2))_\infty =\Lambda (T_2,T_1)
\end{align*}
for $\lambda \in I_{\mu,\nu}^{-\infty}$ and $(T_1,T_2)\in \cA (J_{\mu,\nu})$. 
The statement (iii) follows from the statement (i), (ii) and these equalities. 
\end{proof}

\subsection{The twisted Fourier transforms}
\label{subsec:AC_fourier}

In this subsection, 
we give proofs of 
Lemmas \ref{lem:pre_twist_Ftrans}, \ref{lem:twist_fourier_0}, 
and prepare some lemmas for the twisted Fourier transformation.

\begin{proof}[{Proof of Lemma \ref{lem:pre_twist_Ftrans}}]
Let $\mu \in \bC$, $y\in \bR^\times$ and $f\in C_0^\infty (\bR )$. 
By direct computation, for $\nu \in \bC$, we have 
$(f_{\mu,\nu,\infty})''=(D_{\nu}f)_{\mu,\nu +1,\infty }$, 
where $D_\nu$ is the differential operator defined by 
\begin{align}
\label{eqn:def_Dnu}
&(D_{\nu}f)(x)
=x^2f''(x)+(2\nu +2)\left\{2xf'(x)+(2\nu +1) f(x)\right\}.
\end{align}
By Lemma \ref{lem:fourier_diff}, we have 
$\displaystyle 
\cF(f_{\mu,\nu,\infty})(y)
=\frac{1}{(2\pi \sI y)^2}
\cF((D_{\nu}f)_{\mu,\nu +1,\infty})(y)$ 
for $\nu \in \bC$ such that $\mathrm{Re}(\nu)>0$. 
By repeated application of this equality, we have 
\begin{align}
\label{eqn:twist_fourier}
&\cF(f_{\mu,\nu,\infty})(y)
=\frac{1}{(2\pi \sI y)^{2m}}\cF 
((D_{\nu +m-1}D_{\nu +m-2}\cdots D_{\nu}f )_{\mu,\nu +m,\infty})(y)
\end{align}
for $m\in \bZ_{\geq 0}$ and $\nu \in \bC$ such that $\mathrm{Re}(\nu)>0$. 
This expression gives the holomorphic continuation of 
$\cF (f_{\mu,\nu,\infty})(y)$ to $\mathrm{Re}(\nu)>-m$. 
Since a non-negative integer $m$ can be chosen arbitrarily, 
we obtain the statement (i). 

By direct computation, 
for $\nu \in \bC$ such that $\mathrm{Re}(\nu)>0$, we have 
\begin{align}
\label{eqn:twist_fourier0}
&\cF(f_{\mu,\nu,\infty})(0)
=e^{\pi \sI \mu /2}\Phi_{1}(f;2\nu )
+e^{-\pi \sI \mu /2}\Phi_{-1}(f;2\nu ).
\end{align}
Hence, by the results in \S \ref{subsec:AC_local_zeta}, 
we obtain the statement (ii). 
\end{proof}

The twisted Fourier transform 
$\cF_{\mu,\nu ,\infty}(f)$ of $f\in C^\infty_0(\bR)$ is defined 
in \S \ref{subsec:Fexp_ILL} using 
the meromorphic continuation in Lemma \ref{lem:pre_twist_Ftrans}.

\begin{lem}
\label{lem:twist_fourier}
Let $\mu ,\nu \in \bC$ and $y\in \bR^\times $. 
For $u>0$ and $m\in \bZ_{\geq 0}$ such that $\mathrm{Re}(\nu)>-m$, 
we have 
\begin{align*}
|\cF_{\mu,\nu,\infty}(f)(y)|\leq 
\frac{e^{\pi |\mathrm{Im}(\mu)|/2}
u^{2\mathrm{Re}(\nu)+2m}
}{(\mathrm{Re}(\nu)+m)(2\pi |y|)^{2m}}
\sup_{x\in \bR}|(D_{\nu +m-1}D_{\nu +m-2}\cdots D_{\nu }f)(x)|&\\
(f\in C_0^\infty (\bR)\ \text{with}\ \supp (f)\subset 
\{x\in \bR \mid |x|\leq u\}),&
\end{align*}
where $D_\nu$ is the differential operator defined by (\ref{eqn:def_Dnu}). 
\end{lem}
\begin{proof}
Let $\mu \in \bC$, $y\in \bR^\times$, $u>0$ and 
$f\in C_0^\infty (\bR)$ with 
$\supp (f)\subset \{x\in \bR \mid |x|\leq u\}$. 
For $\nu \in \bC$ such that $\mathrm{Re}(\nu)>0$, 
we have 
\begin{align*}
&|\cF(f_{\mu,\nu,\infty})(y)|
=\left|\sum_{\varepsilon \in \{\pm 1\}}
e^{\pi \sI \varepsilon \mu /2}
\int_{1/u}^{\infty} t^{-2\nu -1}
f(\varepsilon /t){e}^{-2\pi \sI \varepsilon ty}\, dt\right| \\
&\leq 
2e^{\pi |\mathrm{Im}(\mu)|/2}
\left(\int_{1/u}^{\infty} t^{-2\mathrm{Re}(\nu )-1}
\, dt\right)\sup_{x\in \bR}|f(x)|
=\frac{e^{\pi |\mathrm{Im}(\mu)|/2}u^{2\mathrm{Re}(\nu)}}
{\mathrm{Re}(\nu)}\sup_{x\in \bR}|f(x)|.
\end{align*}
The assertion follows immediately from this inequality and 
(\ref{eqn:twist_fourier}). 
\end{proof}

\begin{proof}[{Proof of Lemma \ref{lem:twist_fourier_0}}]
Let $\mu,\nu \in \bC$. 
By (\ref{eqn:twist_fourier0}) and 
the results in \S \ref{subsec:dist_local_zeta}, 
we know that $C_0^\infty(\bR)\ni f\mapsto \cF_{\mu,\nu,\infty}(f)(0)\in \bC$ 
is a distribution on $\bR$. 
By Lemma \ref{lem:twist_fourier}, 
we know that $C_0^\infty(\bR)\ni f\mapsto \cF_{\mu,\nu,\infty}(f)(y)\in \bC$ 
is a distribution on $\bR$ for $y\in \bR^\times $. 
The equality (\ref{eqn:fourier_infty}) follows from 
the definition and 
$\delta^{(n)}(f)=0$ ($f\in C^\infty_0(\bR^\times )$). 
\end{proof}

For $t>0$ and a function $f$ on $\bR$ or $\bR^\times$, 
we define a function $f_{[t]}$ as in \S \ref{subsec:DS_FE}, 
that is, $f_{[t]}(x)=f(tx)$.

\begin{lem}
\label{lem:pre_pole001}
Let $\mu,\nu \in\bC$, $y\in \bR$, $t>0$ and $f\in C^\infty_0(\bR)$. 

\noindent (i) 
$\delta^{(n)}(f_{[t]})=t^n\delta^{(n)}(f)$ for $n\in \bZ_{\geq 0}$. 

\noindent (ii) 
$\cF_{\mu,\nu,\infty}(f_{[t^{-1}]})(y)
=t^{2\nu }\cF_{\mu,\nu,\infty}(f)(t^{-1}y)$ 
if either $2\nu \not\in \bZ_{\leq 0}$ or $y\neq 0$ holds. 

\noindent (iii) 
For $n\in \bZ_{\geq 0}$, we have 
\begin{align*}
\cF_{\mu,-n/2,\infty}(f_{[t^{-1}]})(0)
=&\,t^{-n}\cF_{\mu,-n/2,\infty}(f)(0)\\
&+(\sI)^n\frac{2\cos \bigl(\tfrac{\pi (n+\mu )}{2}\bigr)}{n!}
\delta^{(n)}(f)t^{-n}\log t.
\end{align*} 
\end{lem}
\begin{proof}
We obtain the statement (i) by direct computation. 
By Lemma \ref{lem:rel_f_t}, if $\mathrm{Re}(\nu)>0$, we have 
\begin{align*}
\cF_{\mu,\nu,\infty}(f_{[t^{-1}]})(y)
&=\cF((f_{[t^{-1}]})_{\mu,\nu,\infty})(y)
=t^{2\nu +1}\cF((f_{\mu,\nu,\infty})_{[t]})(y)\\
&=t^{2\nu }\cF(f_{\mu,\nu,\infty})(t^{-1}y)
=t^{2\nu }\cF_{\mu,\nu,\infty}(f)(t^{-1}y).
\end{align*}
Hence, by the uniqueness of the analytic continuation 
as a function of $\nu$, we obtain the statement (ii). 
For $n\in \bZ_{\geq 0}$, we have 
\begin{align*}
&\cF_{\mu,-n/2,\infty}(f_{[t^{-1}]})(0)
-t^{-n}\cF_{\mu,-n/2,\infty}(f)(0)\\
&=\lim_{s \to -n/2}\Biggl\{
\left(\cF_{\mu,s ,\infty}(f_{[t^{-1}]})(0)
-(\sI)^n
\frac{2\cos \bigl(\tfrac{\pi (n+\mu )}{2}\bigr)}
{n!(2s +n)}
\delta^{(n)}(f_{[t^{-1}]})\right)\\
&\hphantom{=======}
-t^{2s }\left(\cF_{\mu,s ,\infty}(f)(0)
-(\sI)^n
\frac{2\cos \bigl(\tfrac{\pi (n+\mu )}{2}\bigr)}
{n!(2s +n)}
\delta^{(n)}(f)\right)\Biggr\}\\
&=(\sI)^n\frac{2\cos \bigl(\tfrac{\pi (n+\mu )}{2}\bigr)}
{n!}
\delta^{(n)}(f)
\lim_{s \to -n/2}\frac{t^{2s}-t^{-n}}{2s +n}.
\end{align*}
Here the second equality follows from the statements (i) and (ii). 
By l'Hospital's rule, we have 
$\displaystyle \lim_{s\to -n/2}\frac{t^{2s}-t^{-n}}{2s+n}
=t^{-n}\log t$, and obtain the statement (iii).
\end{proof}

\subsection{The Jacquet integrals and the delta distribution}
\label{subsec:Jacquet}

In this subsection, 
we introduce some properties of the distributions $\cJ_l$ and 
$\delta^{(m)}_\infty$ on $I_{\mu,\nu}^\infty $, 
which are defined by (\ref{eqn:jacquet_ext}) and 
(\ref{eqn:def_delta_infty}), respectively.

\begin{lem}
\label{lem:delta_functional}
Let $\mu,\nu \in \bC$ and $m\in \bZ_{\geq 0}$. 
Then $\delta^{(m)}_{\infty}=\Lambda (0,\delta^{(m)})$, that is, 
\begin{align}
\label{eqn:deltaG_partition}
&\delta^{(m)}_{\infty}(F)
=\delta^{(m)}(f_2)&
&(\,\text{$F\in I_{\mu,\nu}^\infty$ with a partition $(f_1,f_2)$}\,). 
\end{align}
Here $(0,\delta^{(m)})$ is regarded as 
an element of $\cA (J_{\mu,\nu})$. 
\end{lem}
\begin{proof}
By (\ref{eqn:embedding1}), for $f\in C^\infty_0(\bR)$, 
there is a neighborhood of $(1_2,0)$ in $\widetilde{G}$ 
on which $\iota_{\mu,\nu}(f)$ is 
equal to $0$. 
Hence, by Lemmas \ref{lem:univ_exp1} and \ref{lem:gact_fct_on_R}, we have 
\begin{align*}
\delta^{(m)}_{\infty}(F)
&=(\rho (E_+)^{m}(\iota_{\mu,\nu}(f_1))_\infty)(\tilde{w})
+(\rho (E_+)^{m}\iota_{\mu,\nu}(f_2))(\tilde{w})\\
&=e^{-\pi \sI \mu /2}(\rho (-E_-)^{m}\iota_{\mu,\nu}(f_1))((1_2,0))
+(-1)^m\iota_{\mu,\nu}((f_2)^{(m)})(\tilde{w})\\
&=0+(-1)^m(f_2)^{(m)}(0)=\delta^{(m)}(f_2)
\end{align*}
for $F\in I_{\mu,\nu}^\infty$ with a partition $(f_1,f_2)$. 
\end{proof}

\begin{prop}
\label{prop:Jacquet_property}
Let $\mu,\nu \in \bC$, $l,t\in \bR$ and $F\in I^\infty_{\mu,\nu}$.  

\noindent (i) 
For $n\in \bZ_{\geq 0}$, we have 
\begin{align*}
&\delta^{(n)}_{\infty}(\rho (\tilde{\ru}(t))F)=\sum_{i=0}^n
\binom{n}{i}(-2\nu -n)_{n-i}(-t)^{n-i}\delta^{(i)}_{\infty}(F),
\end{align*}
where $(z)_i={\Gamma (z+i)}/{\Gamma (z)}=z(z+1)\cdots (z+i-1)$ is the 
Pochhammer symbol. 

\noindent (ii) 
If either $l\neq 0$ or $2\nu \not\in \bZ_{\leq 0}$ holds, 
we have 
$\cJ_l(\rho (\tilde{\ru}(t))F)=e^{2\pi \sI lt}\cJ_l(F)$. 

\noindent (iii) 
If $-2\nu =n$ with some $n\in \bZ_{\geq 0}$, 
we have 
\begin{align*}
&\cJ_0(\rho (\tilde{\ru}(t))F)
=\cJ_0(F)+(\sI)^n\sum_{i=0}^{n-1}
\frac{2\cos \bigl(\frac{\pi (n+\mu )}{2}\bigr)}{i!(n-i)}
(-t)^{n-i}\delta^{(i)}_{\infty}(F). 
\end{align*}
In particular, if $\nu =0$, we have 
$\cJ_0(\rho (\tilde{\ru}(t))F)=\cJ_0(F)$. 
\end{prop}
\begin{proof}
Since $
\displaystyle 
\tilde{w}\tilde{\ru}(-x)\tilde{\overline{\ru}}(-t)
=\tilde{\ru}\!\left(\frac{t}{1+tx}\right)
\tilde{\ra}\!\left(\frac{1}{(1+tx)^2}\right)
\tilde{w}\tilde{\ru}\!\left(\frac{-x}{1+tx}\right)$ for 
$x\in \bR$ such that $|x|$ is sufficiently small, 
by Lemma \ref{lem:gact_fct_on_R} (i), we have 
\begin{align*}
\delta^{(n)}_{\infty}(\rho (\tilde{\ru}(t))F)
=&(\rho (E_+)^{n}(\rho (\tilde{\ru}(t))F)_\infty)(\tilde{w})
=(\rho (E_+)^{n}\rho (\tilde{\overline{\ru}}(-t))F_\infty)(\tilde{w})\\
=&\!\left.(-1)^n\frac{d^n}{dx^n}\right|_{x=0}\!F_\infty (
\tilde{w}\tilde{\ru}(-x)\tilde{\overline{\ru}}(-t))\\
=&\!\left.(-1)^n
\frac{d^n}{dx^n}\right|_{x=0}\!
(1+tx)^{-2\nu -1}F_\infty \!\left(\tilde{w}
\tilde{\ru}\!\left(\frac{-x}{1+tx}\right)\!\right)
\end{align*}
for $n\in \bZ_{\geq 0}$. 
Hence, by Lemma \ref{lem:F_lambda_abcd} 
with $a=d=1$, $b=0$, $c=t$, $s=-2\nu -1$ and 
$f(x)=F_\infty (\tilde{w}\tilde{\ru}(-x))$, 
we obtain the statement (i). 

Let $(f_1,f_2)$ be a partition of $F$. 
We set 
$F^{[s]}=\iota_{\mu,s}(f_1)+(\iota_{\mu,s}(f_2))_\infty \in I^\infty_{\mu,s}$ 
$(s\in \bC)$. 
We take a partition $(f_1^{[s]},f_2^{[s]})$ of 
$\rho (\tilde{\ru}(t))F^{[s]}$ as in Lemma \ref{lem:partition_exist}, 
that is, 
\begin{align*}
&f_1^{[s]}(x)=(\rho (\tilde{\ru}(t))F^{[s]})
(\tilde{w}\tilde{\ru}(-x))\Delta (x),&
&f_2^{[s]}(x)=(\rho (\tilde{\ru}(t))F^{[s]})_\infty 
(\tilde{w}\tilde{\ru}(-x)){\Delta}(x)
\end{align*}
for $x\in \bR$. We consider the Jacquet integrals
\begin{align}
\label{eqn:pf_Jproperty001}
&\cJ_{l}(F^{[s]})=\cF (f_1)(l)+\cF_{\mu,s,\infty}(f_2)(l),\\
\label{eqn:pf_Jproperty002}
&\cJ_{l}(\rho (\tilde{\ru}(t))F^{[s]})
=\cF (f_1^{[s]})(l)+\cF_{\mu,s,\infty}(f_2^{[s]})(l). 
\end{align}
It is easy to see that the right hand sides of 
(\ref{eqn:pf_Jproperty001}) and (\ref{eqn:pf_Jproperty002}) 
are both holomorphic on $\bC$ if $l\neq 0$, and on 
$\bC \smallsetminus \frac{1}{2}\bZ_{\leq 0}$ if $l=0$, 
as functions of $s$. 
If $\mathrm{Re}(s)>0$, we have 
\begin{align*}
\cJ_l(\rho (\tilde{\ru}(t))F^{[s]})=
\int_{-\infty}^\infty F^{[s]}(\tilde{w}\tilde{\ru}(-x+t))
e^{2\pi \sI lx}dx
=e^{2\pi \sI lt}\cJ_l(F^{[s]}),
\end{align*}
where the second equality follows from the substitution 
$x\to x+t$. 
By the uniqueness of the analytic continuation as a function of $s$, 
we have 
\[
\cJ_l(\rho (\tilde{\ru}(t))F^{[s]})
=e^{2\pi \sI lt}\cJ_l(F^{[s]})
\]
if either $l\neq 0$ or 
$2s \not\in \bZ_{\leq 0}$ holds. Since $F^{[\nu]}=F$, 
we obtain the statement (ii). 

Assume $-2\nu =n$ with some $n\in \bZ_{\geq 0}$. 
By (\ref{eqn:jacquet_ext}) and (\ref{eqn:deltaG_partition}), we have 
\begin{align*}
&\cJ_0(\rho (\tilde{\ru}(t))F)\\
&=\!\lim_{s\to -n/2}\!\left(
\cJ_{0}(\rho (\tilde{\ru}(t))F^{[s]})
-(\sI)^n\frac{2\cos \bigl(\frac{\pi (n+\mu )}{2}\bigr)}
{n!(2s+n)}\delta^{(n)}_{\infty}(\rho (\tilde{\ru}(t))F^{[s]})\right)\!,\\
&\cJ_0(F)=\!\lim_{s\to -n/2}\!\left(
\cJ_{0}(F^{[s]})-(\sI)^n
\frac{2\cos \bigl(\frac{\pi (n+\mu )}{2}\bigr)}
{n!(2s+n)}\delta^{(n)}_{\infty}(F^{[s]})\right)\!.
\end{align*}
The statement (iii) follows from the statement (i), (ii) 
and these expressions.
\end{proof}

\begin{prop}
\label{prop:act_A_Jac_int}
Let $\mu ,\nu\in \bC$, $y\in \bR_{>0}$, $l\in \bR$ and 
$F\in I_{\mu,\nu}^\infty$. \\[1mm]
(i) $\delta^{(m)}_\infty (\rho (\tilde{\ra}(y))F)
=y^{\nu +m+1/2}\delta^{(m)}_\infty (F)$ for $m\in \bZ_{\geq 0}$. \\[1mm]
(ii) $\cJ_{l}(\rho (\tilde{\ra}(y))F)
=y^{-\nu +1/2}\cJ_{ly}(F)$ 
if either $-2\nu \not\in \bZ_{\geq 0}$ or $l\neq 0$ holds. \\[1mm]
(iii) If $-2\nu =n$ with some $n\in \bZ_{\geq 0}$, 
then we have 
\begin{align*}
&\cJ_{0}(\rho (\tilde{\ra}(y))F)
=y^{\frac{n+1}{2}}\cJ_{0}(F)
-(\sI)^n\frac{2\cos \bigl(\tfrac{\pi (n+\mu )}{2}\bigr)}{n!}
\delta^{(n)}_\infty (F)y^{\frac{n+1}{2}}\log y.
\end{align*}
\end{prop}
\begin{proof}
Let $(f_1,f_2)$ be a partition of $F\in I_{\mu,\nu}^\infty$. 
By Lemma \ref{lem:NA_act_iota_infty},  We have 
\begin{align*}
\rho (\tilde{\ra}(y))F&=
\rho (\tilde{\ra}(y))(\iota_{\mu,\nu}(f_1)+(\iota_{\mu,\nu}(f_2))_\infty )\\
&=
y^{-\nu -1/2}\iota_{\mu,\nu}((f_1)_{[y^{-1}]})
+y^{\nu +1/2}(\iota_{\mu,\nu}((f_2)_{[y]}))_\infty .
\end{align*}
This implies that $(y^{-\nu -1/2}(f_1)_{[y^{-1}]}, 
y^{\nu +1/2}(f_2)_{[y]})$ is a partition of $\rho (\tilde{\ra}(y))F$. 
Hence, by (\ref{eqn:deltaG_partition}) and (\ref{eqn:jacquet_ext}), 
we have 
\begin{align*}
&\delta^{(m)}_\infty (\rho (\tilde{\ra}(y))F)
=y^{\nu +1/2}\delta^{(m)}((f_2)_{[y]}),\\
&\cJ_{l}(\rho (\tilde{\ra}(y))F)
=y^{-\nu -1/2}\cF ((f_1)_{[y^{-1}]})(l)
+y^{\nu +1/2}\cF_{\mu,\nu,\infty}((f_2)_{[y]})(l).
\end{align*}
Applying Lemmas \ref{lem:rel_f_t}, \ref{lem:pre_pole001}, 
and comparing with (\ref{eqn:deltaG_partition}) and (\ref{eqn:jacquet_ext}), 
we obtain the assertion. 
\end{proof}

\subsection{The Fourier expansions}
\label{subsec:Jacquet_integral}

In this subsection, 
we give proofs of Propositions \ref{prop:QAD_Fourier} 
and \ref{prop:Fexp_functional2}.

\begin{proof}[Proof of Proposition \ref{prop:QAD_Fourier}]
Let $L$ be a shifted lattice in $\bR$. Let $\mu,\nu \in \bC$. 

First, let $\alpha \in {\mathfrak{M}} (L)$, 
$\beta \in {\mathfrak{N}}(\bZ_{\geq 0})$, 
and we will prove that 
$\lambda_{\alpha,\beta}\in (I^{-\infty}_{\mu ,\nu})_{L}^{\rm quasi}$ 
and $T^{\lambda_{\alpha,\beta}}=T_\alpha$. 
We define a $\bC$-linear map 
$(T_\alpha)_{\mu,\nu,\infty}\colon C_0^\infty (\bR)\to \bC$ by 
\begin{align}
\label{eqn:def_twist_Talpha}
&(T_\alpha)_{\mu,\nu,\infty}(f)=\sum_{l\in L}\alpha (l)
\cF_{\mu,\nu,\infty}(f)(l)&&(f\in C_0^\infty (\bR)). 
\end{align}
By Lemmas \ref{lem:twist_fourier_0} and \ref{lem:twist_fourier}, 
the right-hand side of (\ref{eqn:def_twist_Talpha}) 
converges absolutely, 
and defines a distribution on $\bR$. 
By Lemma \ref{lem:twist_fourier_0}, we have 
\begin{align}
\label{eqn:twist_dist}
&(T_\alpha)_{\mu,\nu,\infty}(f)
=\sum_{l\in L}\alpha (l)\cF (f_{\mu,\nu,\infty})(l)
=T_\alpha (f_{\mu,\nu,\infty})&
&(f\in C^\infty_0(\bR^\times )).
\end{align} 
This equality implies 
$(T_\alpha ,(T_\alpha)_{\mu,\nu,\infty})\in \cA (J_{\mu,\nu})$. 
Moreover, by (\ref{eqn:jacquet_ext}), we have 
\begin{align*}
&\Lambda (T_\alpha ,(T_\alpha)_{\mu,\nu,\infty})(F)
=\sum_{l\in L}\alpha (l)\cJ_l(F)&
&(F\in I_{\mu,\nu}^\infty).
\end{align*}
By this equality and Lemma \ref{lem:delta_functional}, 
we have 
\begin{align}
\label{eqn:pf_QAD_Fourier002}
\lambda_{\alpha,\beta}
=\Lambda \left(T_\alpha ,\, 
{\textstyle (T_\alpha)_{\mu,\nu,\infty}
+\sum_{m=0}^\infty \beta (m)\delta^{(m)}}\right).
\end{align} 
This implies 
$\lambda_{\alpha,\beta}\in (I^{-\infty}_{\mu ,\nu})_{L}^{\rm quasi}$ 
and $T^{\lambda_{\alpha,\beta}}=T_\alpha$ 
by Proposition \ref{prop:pair_to_functionl}. 

Next, we will prove that 
(\ref{eqn:QAD_Fexp_map}) is bijective. 
Let $\lambda \in (I^{-\infty}_{\mu ,\nu})_{L}^{\rm quasi}$. 
By Proposition \ref{prop:periodic_dist}, there is a unique 
$\alpha \in {\mathfrak{M}} (L)$ such that $T^\lambda =T_\alpha$. 
Hence, because of (\ref{eqn:pf_QAD_Fourier002}) and 
Proposition \ref{prop:pair_to_functionl}, it suffices to show that 
there is a unique $\beta \in \mathfrak{N}(\bZ_{\geq 0})$ satisfying 
$T^{\lambda_\infty} 
=(T_\alpha )_{\mu,\nu,\infty}
+\sum_{m=0}^\infty \beta (m)\delta^{(m)}$. 
By Proposition \ref{prop:pair_to_functionl} and 
(\ref{eqn:twist_dist}), we have 
\begin{align*}
&T^{\lambda_\infty} (f)=T^\lambda (f_{\mu,\nu,\infty })
=T_\alpha (f_{\mu,\nu,\infty })
=(T_\alpha )_{\mu,\nu,\infty}(f)&
&(f\in C^\infty_0(\bR^\times )).
\end{align*}
Hence, $T^{\lambda_\infty} -(T_\alpha )_{\mu,\nu,\infty}$ is 
a distribution on $\bR$ whose support is contained in $\{0\}$, 
and there is a unique $\beta \in \mathfrak{N}(\bZ_{\geq 0})$ 
satisfying $T^{\lambda_\infty} -(T_\alpha )_{\mu,\nu,\infty}
=\sum_{m=0}^\infty \beta (m)\delta^{(m)}$ 
by \cite[Theorem 3.2.1]{Friedlander_001}.  
\end{proof}

Let $\mu,\nu \in \bC$. 
For $m\in \bZ_{\geq 0}$, 
we set $F_{m}=(\iota_{\mu ,\nu }(\Delta_{m}))_\infty 
\in I_{\mu,\nu}^\infty$ with 
\begin{align}
\label{eqn:def_Delta_m}
&\Delta_{m}(x)=x^m\Delta (x)&&(x\in \bR), 
\end{align}
where  
$\Delta$ is the function defined in \S \ref{subsec:test_ftn}. 
Proposition \ref{prop:Fexp_functional2} follows immediately from 
Proposition \ref{prop:QAD_Fourier} and 
the following lemma. 
\begin{lem}
Let $L$ be a shifted lattice in $\bR$. 
Let $\mu ,\nu \in \bC$, 
$t\in L^\vee$ and $F\in I^{\infty}_{\mu,\nu}$. 
Let $\alpha \in \mathfrak{M}(L)$ and 
$\beta \in \mathfrak{N}(\bZ_{\geq 0})$. 

\noindent (i) 
Assume $-2\nu \not\in \bZ_{>0}$. Then 
\begin{align*}
\lambda_{\alpha,\beta} (\rho (\tilde{\ru}(t))F)
=&\,\omega_L(t)\lambda_{\alpha,\beta} (F)+(1-\omega_L(t))
\sum_{m=0}^\infty \beta (m)\delta^{(m)}_{\infty}(F)\\
&+\sum_{m=1}^\infty \beta (m)
\left(\sum_{i=0}^{m-1}
\binom{m}{i}(-2\nu -m)_{m-i}(-t)^{m-i}\delta^{(i)}_{\infty}(F)\right).
\end{align*}
In particular, 
\begin{align*}
\lambda_{\alpha,\beta} (\rho (\tilde{\ru}(t))F_{0})
=&\omega_L(t)\lambda_{\alpha,\beta} (F_0)
+(1-\omega_L(t))\beta (0)\\
&+\sum_{m=1}^\infty \beta (m)(-2\nu -m)_{m}(-t)^{m}.
\end{align*}
\noindent (ii) 
Assume $-2\nu =n$ with some $n\in \bZ_{>0}$. 
Then 
\begin{align*}
\lambda_{\alpha,\beta} (\rho (\tilde{\ru}(t))F)
=&\,\omega_L(t)\lambda_{\alpha,\beta} (F)+(1-\omega_L(t))
\sum_{m=0}^\infty \beta (m)\delta^{(m)}_{\infty}(F)\\
&+(\sI)^n\alpha (0)\sum_{i=0}^{n-1}
\frac{2\cos \bigl(\frac{\pi (n+\mu )}{2}\bigr)}
{i!(n-i)}(-t)^{n-i}\delta^{(i)}_{\infty}(F)\\
&+\sum_{m=1}^{n-1}\beta (m)
\left(\sum_{i=0}^{m-1}\binom{m}{i}\frac{(n-i-1)!}{(n-m-1)!}
(-t)^{m-i}\delta^{(i)}_{\infty}(F)\right)\\
&+\sum_{m=n+1}^\infty \beta (m)
\left(\sum_{i=n}^{m-1}\binom{m}{i}
\frac{(m-n)!}{(i-n)!}t^{m-i}\delta^{(i)}_{\infty}(F)\right).
\end{align*}
Here we understand $\alpha (0)=0$ if $0\not\in L$. 
In particular, 
\begin{align*}
\lambda_{\alpha,\beta} (\rho (\tilde{\ru}(t))F_0)=
&\,\omega_L(t)\lambda_{\alpha,\beta} (F_0)
+(1-\omega_L(t))\beta (0)\\
&+\alpha (0)
\frac{2\cos \bigl(\frac{\pi (n+\mu )}{2}\bigr)}{n(\sI)^n}t^{n}
+\sum_{m=1}^{n-1}
\frac{\beta (m)(n-1)!}{(n-m-1)!}
(-t)^{m},\\
\lambda_{\alpha,\beta}(\rho (\tilde{\ru}(t))F_n)
=&\,\omega_L(t)\lambda_{\alpha,\beta}(F_n)
+(-1)^n(1-\omega_L(t))\beta (n)n!\\
&+\sum_{m=n+1}^\infty 
(-1)^n\beta (m)m!\,t^{m-n}.
\end{align*}
\end{lem}
\begin{proof}
This lemma follows from Proposition \ref{prop:Jacquet_property} 
and the equalities 
\begin{align*}
&(n-m)_{m-i}=\left\{\begin{array}{ll}
\dfrac{(n-i-1)!}{(n-m-1)!}&\text{if}\ i\leq m<n,\\[8pt]
(-1)^{m-i}\dfrac{(m-n)!}{(i-n)!}&\text{if}\ n\leq i\leq m,\\
0&\text{if}\ i<n\leq m
\end{array}\right.
\end{align*}
and $\displaystyle 
\delta^{(n)}_{\infty}(F_{m})=\left\{\begin{array}{ll}
(-1)^mm!&\text{if}\ n=m,\\
0&\text{if}\ n\neq m
\end{array}\right.$ for $m,n,i\in \bZ_{\geq 0}$. 
\end{proof}

\subsection{The explicit calculation of the Jacquet integrals}
\label{subsec:Jac_int_special}

In this subsection, we give proofs of Theorem \ref{thm:QAD_Fexp} (i) 
and Proposition \ref{prop:poisson} (ii). 
Proposition \ref{prop:poisson} (ii) is an immediate consequence of 
Propositions \ref{prop:Jacquet_property}, \ref{prop:act_A_Jac_int} 
and the following proposition.

\begin{prop}
\label{prop:Jac_int_Ktype}
Let $\mu,\nu \in \bC$ and $\kappa \in \mu +2\bZ$. 
Let $F_{\nu,\kappa}$ be an element of 
$I_{\mu,\nu}^\infty$ defined by (\ref{eqn:def_F_nu_kappa}). 

\noindent (i) Let $l\in \bR$. Then 
\begin{align*}
&e^{\pi \sI \kappa /2}\delta^{(0)}_{\infty}(F_{\nu,\kappa})
=(-1)^{\frac{\mu -\kappa }{2}},\\
&e^{\pi \sI \kappa /2}\cJ_{l}(F_{\nu,\kappa})
=\left\{\begin{array}{l}
\displaystyle 
\frac{2^{1-2\nu }\pi \Gamma (2\nu )}
{\Gamma \bigl(\tfrac{2\nu +1-\kappa}{2}\bigr)
\Gamma \bigl(\tfrac{2\nu +1+\kappa }{2}\bigr)}
\quad \text{if } l=0 \text{ and } -2\nu \not\in \bZ_{\geq 0}, \\[4mm]
\displaystyle 
\frac{\pi^{\nu +\frac{1}{2}}|l|^{\nu -\frac{1}{2}}
}{\Gamma \bigl(\frac{2\nu +1+\sgn (l) \kappa }{2}\bigr)}
W_{\sgn (l)\frac{\kappa }{2},\nu }(4\pi |l|)
\quad \text{if } l\neq 0.
\end{array}\right.
\end{align*}
\noindent (ii) Assume $-2\nu =n$ with some $n\in \bZ_{\geq 0}$. 
Then 
\begin{align*}
&e^{\pi \sI \kappa /2}\delta^{(n)}_{\infty}(F_{-n/2,\kappa})
=(-1)^{\frac{\mu -\kappa }{2}}\mathbf{d}(n,\kappa ),\\
&e^{\pi \sI \kappa /2}\cJ_{0}(F_{-n/2,\kappa})\\
&=\left\{\begin{array}{ll}
2\mathbf{j} (\kappa )\cos \bigl(\tfrac{\pi \kappa }{2}\bigr)
&\text{if } n=0 \text{ and } \mu -1 \not\in 2\bZ,\\[2mm]
\displaystyle 
(-1)^{\frac{|\kappa|-n-1}{2}}
\frac{2^{n}\pi \,\bigl(\tfrac{|\kappa|+n-1}{2}\bigr)!}
{n!\, \bigl(\tfrac{|\kappa|-n-1}{2}\bigr)!}
&\text{if } |\kappa | -n+1 \in 2\bZ_{>0},\\[2mm]
0&\text{if } |\kappa | -n+1 \in 2\bZ_{\leq 0},
\end{array}\right.
\end{align*}
where $\mathbf{d}(n,\kappa )$ and $\mathbf{j} (\kappa )$ 
are defined by (\ref{eqn:def_d_n_kappa}) and (\ref{eqn:def_j_kappa}), 
respectively. 
\end{prop}
\begin{proof}
The equality 
$e^{\pi \sI \kappa /2}\delta^{(0)}_{\infty}(F_{\nu,\kappa})
=(-1)^{\frac{\mu -\kappa }{2}}$ 
follows from the definition. 
We take a partition $(f_{(\nu,\kappa),1},f_{(\nu,\kappa),2})$ of 
$F_{\nu,\kappa}$ as in Lemma \ref{lem:partition_exist}, 
that is, 
\begin{align*}
&f_{(\nu,\kappa),1}(x)=F_{\nu,\kappa}(\tilde{w}\tilde{\ru}(-x))\Delta (x),&
&f_{(\nu,\kappa),2}(x)
=(F_{\nu,\kappa})_\infty (\tilde{w}\tilde{\ru}(-x))\Delta (x),
\end{align*}
where $\Delta$ is the function defined in \S \ref{subsec:test_ftn}. 
Then we have 
\begin{align*}
\cJ_{l}(F_{\nu,\kappa})=\cF (f_{(\nu,\kappa),1})(l)+
\cF_{\mu,\nu,\infty}(f_{(\nu,\kappa),2})(l).
\end{align*}
By this expression, 
we know that $\cJ_{l}(F_{\nu,\kappa})$ is a holomorphic function 
of $\nu$ on $\bC\smallsetminus \frac{1}{2}\bZ_{\leq 0}$ if $l=0$, 
and on $\bC$ if $l\neq 0$. 
Because of the uniqueness of the analytic continuation, 
in order to prove the equality for 
$e^{\pi \sI \kappa /2}\cJ_{l}(F_{\nu,\kappa})$ 
in the statement (i), 
it suffices to show the case of $\mathrm{Re}(\nu)>0$. 
When $\mathrm{Re}(\nu )>0$, we have 
\begin{align*}
&e^{\pi \sI \kappa /2}\cJ_{l}(F_{\nu,\kappa})
=e^{\pi \sI \kappa /2}\int_{-\infty}^\infty 
F_{\nu,\kappa}(\tilde{w}\tilde{\ru}(-x))
e^{2\pi \sI lx}dx\\
&\hspace{2cm}
=e^{\pi \sI \kappa /2}\int_{-\infty}^\infty 
(1+x^2)^{-\nu -1/2}e^{-\sI \kappa \arg (-x+\sI )}e^{2\pi \sI lx}dx\\
&\hspace{2cm}
=\int_{-\infty}^\infty 
(1-\sI x)^{-\nu -\frac{1-\kappa}{2}}(1+\sI x)^{-\nu -\frac{1+\kappa }{2} }
e^{2\pi \sI lx}dx.
\end{align*}
The last integral is calculated by 
Maass \cite[Chapter IV, \S 3]{Maass_003}, 
and we obtain the equality for 
$e^{\pi \sI \kappa /2}\cJ_{l}(F_{\nu,\kappa})$ in 
the statement (i).

Assume $-2\nu =n$ with some $n\in \bZ_{\geq 0}$. 
By direct computation, we have 
\begin{align*}
&e^{\pi \sI \kappa /2}\delta^{(n)}_{\infty}(F_{-n/2,\kappa})
=e^{\pi \sI \kappa /2}(-1)^{\frac{\mu -\kappa}{2}}
(\rho (E_+)^{n}F_{-n/2,\kappa })(\tilde{w})\\
&\hspace{10mm}=e^{\pi \sI \kappa /2}(-1)^{n+\frac{\mu -\kappa }{2}}
\left.\frac{d^n}{dx^n}\right|_{x=0}
F_{-n/2,\kappa }(\tilde{w}\tilde{\ru}(-x))\\
&\hspace{10mm}=(-1)^{n+\frac{\mu -\kappa }{2}}
\left.\frac{d^n}{dx^n}\right|_{x=0}
(1-\sI x)^{\frac{\kappa +n-1}{2}}
(1+\sI x)^{-\frac{\kappa -n+1}{2}}\\
&\hspace{10mm}=(-1)^{\frac{\mu -\kappa }{2}}\mathbf{d}(n,\kappa ).
\end{align*}
The equality for $e^{\pi \sI \kappa /2}\cJ_{0}(F_{-n/2,\kappa})$ 
in the case of $\mu +n-1 \in 2\bZ$ follows from 
the statement (i). We have 
\begin{align*}
&e^{\pi \sI \kappa /2}\cJ_0(F_{0,\kappa })
=e^{\pi \sI \kappa /2}\lim_{s\to 0}\!\left(
\cJ_{0}(F_{s,\kappa})-
\frac{\cos \bigl(\frac{\pi \mu }{2}\bigr)}{s}
\delta^{(0)}_{\infty}(F_{s,\kappa})\right)\\
&\hspace{10mm}=
\lim_{s\to 0}\frac{1}{s}\left(
\frac{2^{-2s}\pi \Gamma (2s+1)}
{\Gamma \bigl(\tfrac{2s +1-\kappa}{2}\bigr)
\Gamma \bigl(\tfrac{2s +1+\kappa }{2}\bigr)}
-\frac{\pi }{\Gamma \bigl(\tfrac{1-\kappa}{2}\bigr)
\Gamma \bigl(\tfrac{1+\kappa }{2}\bigr)}
\right)\\
&\hspace{10mm}=
\frac{\pi }{\Gamma \bigl(\tfrac{1-\kappa}{2}\bigr)
\Gamma \bigl(\tfrac{1+\kappa }{2}\bigr)}
\left(-2\log 2+2\Gamma'(1)
-\frac{\Gamma'\bigl(\tfrac{1-\kappa}{2}\bigr)}
{\Gamma \bigl(\tfrac{1-\kappa}{2}\bigr)}
-\frac{\Gamma'\bigl(\tfrac{1+\kappa}{2}\bigr)}
{\Gamma \bigl(\tfrac{1+\kappa}{2}\bigr)}\right)\\
&\hspace{10mm}=
2\mathbf{j} (\kappa )\cos \bigl(\tfrac{\pi \kappa }{2}\bigr).
\end{align*}
Here the second equality follows from the statement (i), 
the functional equation $\Gamma (s+1)=s\Gamma (s)$ and 
Euler's reflection formula
\begin{align}
\label{eqn:Gamma_reflection}
\Gamma (s)\Gamma (1-s)=\frac{\pi}{\sin (\pi s)}
=\frac{\pi}{\cos \bigl(\pi \bigl(s-\tfrac{1}{2}\bigr)\bigr)}.
\end{align} 
The third equality follows from l'Hospital's rule, 
and the fourth equality follows from 
the expansion 
\[
\frac{\Gamma'(s)}{\Gamma (s)}
=\Gamma'(1)-\sum_{i=0}^\infty 
\left(\frac{1}{s+i}-\frac{1}{i+1}\right) 
\]
and (\ref{eqn:Gamma_reflection}). 
Hence, we obtain the statement (ii). 
\end{proof}

Let us prove Theorem \ref{thm:QAD_Fexp} (i). 
Let $L_1$ and $L_2$ be two shifted lattices in $\bR$. 
Let $\mu,\nu \in \bC$. 
Let $\lambda \in (I^{-\infty}_{\mu ,\nu})_{L_1,L_2}^{\rm quasi}$. 
By Proposition \ref{prop:QAD_Fourier}, 
there are unique $(\alpha_i,\beta_i)\in \mathfrak{M}(L_i)\times 
{\mathfrak{N}}(\bZ_{\geq 0})$ $(i=1,2)$ 
such that $\lambda =\lambda_{\alpha_1,\beta_1}$ and 
$(\lambda_{\alpha_1,\beta_1})_\infty =\lambda_{\alpha_2,\beta_2}$. 
If $\lambda \in (I^{-\infty}_{\mu ,\nu})_{L_1,L_2}$, 
then we have $(\alpha_i,\beta_i)\in \mathfrak{M}(L_i)_{\mu,\nu}^0\times 
{\mathfrak{N}}(S_{\nu}(L_i))$ $(i=1,2)$ by 
Proposition \ref{prop:Fexp_functional2}. 
Hence, our proof is completed by the following lemma. 

\begin{lem}
\label{lem:b0_vanish_Fexp}
Let $L_1$ be a lattice in $\bR$. 
Let $L_2$ be a shifted lattice in $\bR$. 
Let $n\in \bZ_{>0}$ and $\mu \in 1-n+2\bZ$. 
Let $(\alpha_1,\beta_1)\in \mathfrak{M}(L_1)\times 
{\mathfrak{N}}(S_{-n/2}(L_1))$. 
Assume $\lambda_{\alpha_1,\beta_1}
\in (I^{-\infty}_{\mu ,-n/2})_{L_1,L_2}$ and 
$(n>1$ or $0\not\in L_2)$. 
Then we have $\beta_1(0)=0$. 
\end{lem}
\begin{proof}
Let $\kappa \in \mu +2\bZ$ such that $|\kappa |\leq n-1$. 
Let $F_{-n/2,\kappa}$ be the function in 
$I_{\mu,-n/2}^\infty$ defined by (\ref{eqn:def_F_nu_kappa}). 
By Propositions \ref{prop:Jacquet_property}, \ref{prop:act_A_Jac_int}, 
\ref{prop:Jac_int_Ktype} and 
\[
\tilde{\overline{\ru}}(t)=
\tilde{\ru}\!\left(\frac{t}{t^2+1}\right)
\tilde{\ra}\!\left(\frac{1}{t^2+1}\right)
\tilde{\rk}(-\arg (1+\sI t)), 
\]
we have 
$e^{\pi \sI \mu /2}
\lambda_{\alpha_1,\beta_1}(\rho (\tilde{\overline{\ru}}(t))F_{-n/2, \kappa })
=(1+t^2)^{\frac{n-1}{2}}e^{-\sI \kappa \arg (1+\sI t)}\beta_1 (0)$ 
for $t\in \bR$. 
By this equality and the definition of 
$(I^{-\infty}_{\mu ,-n/2})_{L_1,L_2}$, 
we have 
\begin{align*}
\omega_{L_2}(-t)\beta_1 (0)
&=e^{\pi \sI \mu /2}\omega_{L_2}(-t)
\lambda_{\alpha_1,\beta_1}(F_{-n/2, \kappa })\\
&=e^{\pi \sI \mu /2}\lambda_{\alpha_1,\beta_1}
(\rho (\tilde{\overline{\ru}}(t))F_{-n/2, \kappa })\\
&=(1+t^2)^{\frac{n-1}{2}}e^{-\sI \kappa \arg (1+\sI t)}\beta_1 (0)
\end{align*}
for $t\in L_2^\vee $. Hence, we have $\beta_1(0)=0$ by the assumption 
$(n>1$ or $0\not\in L_2)$. 
\end{proof}

\subsection{Poles of the Dirichlet series}
\label{subsec:pole_DS}

In this subsection, 
we give a proof of Proposition \ref{prop:poles_DS}. 
Proposition \ref{prop:poles_DS} follows from 
Euler's reflection formula (\ref{eqn:Gamma_reflection}) 
of the Gamma function, and Proposition \ref{prop:pre_pole005} below. 

\begin{prop}
\label{prop:pre_pole005}
Let $L_1$ and $L_2$ be two shifted lattices in $\bR$. 
Let $\mu,\nu \in \bC$. Let 
$(\alpha_i,\beta_i)\in \mathfrak{M}(L_i)\times 
{\mathfrak{N}}(\bZ_{\geq 0})$ $(i=1,2)$ 
such that $(\lambda_{\alpha_1,\beta_1})_\infty =\lambda_{\alpha_2,\beta_2}$. 
We understand that $\alpha_2(0)=0$ if $0\not\in L_2$, 
and double signs are in the same order. 

\noindent (i) 
The functions 
\begin{align}
\label{eqn:DS_entire_2nd}
\sin (\pi s)
 \Xi_\pm (\alpha_{1};s)
-\frac{\cos \bigl(\tfrac{\pi (2\nu \mp \mu )}{2}\bigr)\alpha_{2}(0)}
{s+2\nu -1}\! 
\end{align}
are entire. 

\noindent (ii) Let $m\in \bZ_{\geq 0}$, and assume $-2\nu \neq m$. 
Then the functions 
\begin{align*}
&\Xi_\pm (\alpha_1;s)-\frac{(\pm \sI)^{m}m!\beta_2(m)}{2\pi (s-m-1)}
\end{align*}
are holomorphic at $s=m+1$.

\noindent (iii) Let $m\in \bZ_{\geq 0}$, and assume $-2\nu =m$. 
Then the functions 
\begin{align*}
&\Xi_\pm (\alpha_{1};s)
+\frac{\pi \sin \bigl(\tfrac{\pi (m\mp \mu )}{2}\bigr)\alpha_{2}(0)
-(\pm \sI)^{m}m!\beta_2(m)}{2\pi (s-m-1)}
+\frac{\cos \bigl(\tfrac{\pi (m\mp \mu )}{2}\bigr)\alpha_2(0)}{\pi (s-m-1)^2}
\end{align*}
are holomorphic at $s=m+1$. 
\end{prop}

In order to prove this proposition, 
we prepare some lemmas. 

\begin{lem}
\label{lem:pre_pole003}
We use the notation in Proposition \ref{prop:pre_pole005}. 
Take $r\in \bR$ so that $\alpha_{1}\in 
\mathfrak{M}_r({L}_1)$. Let $f\in C^\infty_0(\bR)$. 
Then, for $s\in \bC$ such that $\mathrm{Re}(s)>
\max \{r+1,0\}$, we have 
\begin{align*}
&Z(\alpha_{1},\cF (f);s)
=
Z_+(\alpha_{1},\cF (f);s)
+Z_+(\alpha_{2},\cF_{\mu,\nu,\infty}(f);-s-2\nu +1)\\
&\hspace{5mm}
-\frac{\alpha_{1}(0)\cF (f)(0)}{s}
+\frac{\alpha_{2}(0)\cF_{\mu,\nu,\infty}(f)(0)}{s+2\nu -1}
+\sum_{m=0}^\infty 
\frac{\beta_2(m)\delta^{(m)}(f)}{s-m-1}\\
&\hspace{5mm}
-\left\{\!\begin{array}{ll}\displaystyle 
(\sI)^n\alpha_2(0)
\frac{2\cos \bigl(\tfrac{\pi (n+\mu )}{2}\bigr)}
{n!(s-n-1)^2}\delta^{(n)}(f)
&\displaystyle \text{if}\ -2\nu =n\ 
\text{with some}\ n\in \bZ_{\geq 0},\\[2mm]
0&\text{otherwise}.
\end{array}
\!\right.\!
\end{align*}
\end{lem}
\begin{proof}
By the equality $\lambda_{\alpha_1,\beta_1}(\iota_{\mu,\nu}(f_{[t^{-1}]}))
=\lambda_{\alpha_2,\beta_2}((\iota_{\mu,\nu}(f_{[t^{-1}]}))_\infty )$, 
we have 
\begin{align*}
&\sum_{l\in L_1}\alpha_1(l)\cF (f_{[t^{-1}]})(l)
=\sum_{l\in L_2}\alpha_2(l)\cF_{\mu,\nu,\infty} (f_{[t^{-1}]})(l)
+\sum_{m=0}^\infty \beta_2(m)\delta^{(m)}(f_{[t^{-1}]})
\end{align*}
for $t>0$. 
Applying Lemmas \ref{lem:rel_f_t} (i) and \ref{lem:pre_pole001}, 
we have 
\begin{align*}
\sum_{l\in {L}_1}\alpha_{1}(l)\cF (f)(tl)
&=t^{2\nu -1}\sum_{l\in {L}_2}\alpha_{2}(l)
\cF_{\mu,\nu,\infty} (f )(t^{-1}l)
+\sum_{m=0}^\infty \beta_2(m)\delta^{(m)}(f)t^{-m-1}\\
&
+\left\{\!\begin{array}{l}\displaystyle 
(\sI)^n\alpha_2(0)\frac{2\cos \bigl(\tfrac{\pi (n+\mu )}{2}\bigr)}{n!}
\delta^{(n)}(f)t^{-n-1}\log t\\[3mm]
\hspace{33mm}
\displaystyle \text{if}\ -2\nu =n\ 
\text{with some}\ n\in \bZ_{\geq 0},\\[3mm]
0\hspace{31.2mm}\text{otherwise}
\end{array}
\!\right.\!
\end{align*}
for $t>0$. 
Using this equality instead of Lemma \ref{lem:Po_sum_modoki}, 
we obtain the assertion similar to 
the proof of Lemma \ref{lem:AC_global_zeta}. 
\end{proof}

For $m\in \bZ_{\geq 0}$, 
let $\Delta_m$ be the function on $\bR$ defined by (\ref{eqn:def_Delta_m}). 

\begin{lem}
\label{lem:pre_pole004}
Let $m\in \bZ_{\geq 0}$. Then 
\begin{align}
\label{eqn:prepole4_001}
\Phi_{1}(\Delta_m;-s+1)+\frac{1}{s-m-1}\ \text{is entire, }
\end{align}
and $\Delta_m$ satisfies the following equalities 
\begin{align}
\label{eqn:prepole4_002}
&\delta^{(n)}(\Delta_m)=\left\{\begin{array}{ll}
(-1)^mm!&\text{if}\ m=n,\\
0&\text{otherwise}
\end{array}\right.\hspace{28mm}(n \in \bZ_{\geq 0} ),\\
\label{eqn:prepole4_003}
&\Phi_{-1}(\Delta_m;s)=(-1)^m\Phi_{1}(\Delta_m;s),\\ 
\label{eqn:prepole4_005}
\begin{split}
&\cF_{\mu,-m/2,\infty}(\Delta_m)(0)
=2(\sI )^m\cos \bigl(\tfrac{\pi (\mu -m)}{2}\bigr)\\
&\hspace{42pt}
\times \lim_{s\to m+1}\left(\Phi_{1}(\Delta_m;-s+1)+\frac{1}{s -m-1}\right)
\hspace{10mm}(\mu \in \bC ).
\end{split}
\end{align}
\end{lem}
\begin{proof}
The assertion follows from the definition and 
the results in \S \ref{subsec:AC_local_zeta}. 
\end{proof}

\begin{proof}[Proof of Proposition \ref{prop:pre_pole005}]
By Propositions \ref{prop:pair_to_functionl} and 
\ref{prop:QAD_Fourier}, we have 
$(T_{\alpha_1},T_{\alpha_2})\in \cA (L_1,L_2;J_{\mu,\nu})$. 
Multiplying by 
\[
\frac{1}{2\sI }\left(\begin{array}{cc}
e^{\pi \sI s/2}&-e^{-\pi \sI s/2}\\
-e^{-\pi \sI s/2}&e^{\pi \sI s/2}
\end{array}\right)
\]
from the left and subtracting some entire function, 
the entire $\bC^2$-valued function (\ref{eqn:DS_entire}) 
becomes a $\bC^2$-valued function whose entries are 
(\ref{eqn:DS_entire_2nd}). 
Hence, we obtain the statement (i). 

Let $m\in \bZ_{\geq 0}$. Multiplying by 
$(e^{-\pi \sI (m+1)/2},e^{\pi \sI (m+1)/2})$ from the left, 
the entire function (\ref{eqn:DS_entire}) becomes 
\begin{align*}
&2\cos \bigl(\tfrac{\pi (s-m-1)}{2}\bigr) 
(\Xi_+(\alpha_{1};s)-(-1)^{m}\Xi_-(\alpha_{1};s))\\
&+2\cos \bigl(\tfrac{\pi (m+1)}{2}\bigr)\frac{\alpha_{1}(0)}{s}
-2\sin \bigl(\tfrac{\pi (\mu -m)}{2}\bigr)\frac{\alpha_{2}(0)}{s+2\nu -1}.
\end{align*}
This implies that  
\begin{align}
&\Xi_+(\alpha_{1};s)-(-1)^{m}\Xi_-(\alpha_{1};s)
-\sin \bigl(\tfrac{\pi (\mu -m)}{2}\bigr)\frac{\alpha_{2}(0)}{s+2\nu -1}
\label{eqn:pf_prepole5_001}
\end{align}
is holomorphic at $s=m+1$. Hence, 
in order to prove the statements (ii) and (iii), 
it suffices to give the principal part of 
$\Xi_+(\alpha_{1};s)+(-1)^{m}\Xi_-(\alpha_{1};s)$ at $s=m+1$. 
We have 
\begin{align*}
&\nonumber 
Z(\alpha_{1},\cF (\Delta_m);s)
=\left(\,\Phi_{1}(\Delta_m;-s+1),\,\Phi_{-1}(\Delta_m;-s+1)\,\right)
{\mathrm{E}}(s)
\left(\begin{array}{c}
\Xi_+(\alpha_{1} ;s)\\
\Xi_-(\alpha_{1} ;s)
\end{array}\right)\\
&=-2(\sI)^{m}\sin \bigl(\tfrac{\pi (s-m-1)}{2}\bigr)\Phi_{1}(\Delta_m;-s+1)
(\Xi_+(\alpha_1;s)+(-1)^{m}\Xi_-(\alpha_1;s)).
\end{align*}
Here the first equality follows from Lemma \ref{lem:zeta_converge} and 
Proposition \ref{prop:LFE}, 
and the second equality follows from (\ref{eqn:prepole4_003}). 
On the other hand, we have 
\begin{align*}
&\nonumber 
Z(\alpha_{1},\cF (\Delta_m);s)=
Z_+(\alpha_{1},\cF (\Delta_m);s)
+Z_+(\alpha_{2},\cF_{\mu,\nu,\infty}(\Delta_m);-s-2\nu +1)\\
&\nonumber \hphantom{=====}
-\frac{\alpha_{1}(0)\cF (\Delta_m)(0)}{s}
+\frac{\alpha_{2}(0)\cF_{\mu,\nu,\infty}(\Delta_m)(0)}{s+2\nu -1}
+\frac{(-1)^mm!\beta_2(m)}{s-m-1}\\
&\hphantom{=====}
-\left\{\!\begin{array}{ll}\displaystyle 
2(\sI)^{m}\cos \bigl(\tfrac{\pi (\mu -m)}{2}\bigr)\frac{\alpha_2(0)}
{(s-m-1)^2}
&\displaystyle \text{if}\ -2\nu =m,\\[5pt]
0&\text{otherwise}
\end{array}
\!\right.\!
\end{align*}
by Lemma \ref{lem:pre_pole003} and (\ref{eqn:prepole4_002}). 
By these equalities, we find that 
\begin{align}
\label{eqn:pf_prepole5_002}
\begin{split}
&2\sin \bigl(\tfrac{\pi (s-m-1)}{2}\bigr)\Phi_{1}(\Delta_m;-s+1)
(\Xi_+(\alpha_1;s)+(-1)^{m}\Xi_-(\alpha_1;s))\\
&+\frac{\alpha_{2}(0)\cF_{\mu,\nu,\infty}(\Delta_m)(0)}{(\sI)^{m}(s+2\nu -1)}
+\frac{(\sI)^{m}m!\beta_2(m)}{s-m-1}\\
&-\left\{\!\begin{array}{ll}\displaystyle 
2\cos \bigl(\tfrac{\pi (\mu -m)}{2}\bigr)\frac{\alpha_2(0)}
{(s-m-1)^2}
&\displaystyle \text{if}\ -2\nu =m,\\[5pt]
0&\text{otherwise}
\end{array}
\!\right.\!
\end{split}
\end{align}
is holomorphic at $s=m+1$. 
By (\ref{eqn:pf_prepole5_002}) and 
(\ref{eqn:prepole4_001}), if $2\nu \neq -m$, 
we know that 
\begin{align*}
&\Xi_+(\alpha_1;s)+(-1)^{m}\Xi_-(\alpha_1;s)
-\frac{(\sI)^{m}m!\beta_2(m)}{\pi (s-m-1)}
\end{align*}
is holomorphic at $s=m+1$. 
This completes the proof of the statement (ii). 

Assume $-2\nu =m$. By (\ref{eqn:prepole4_001}) and (\ref{eqn:prepole4_005}), 
we find that 
\begin{align}
\label{eqn:pf_prepole5_004}
\begin{split}
&\frac{\alpha_2(0)\cF_{\mu,-m/2,\infty}(\Delta_m)(0)}{(\sI)^{m}(s-m-1)}
-\frac{2\cos \bigl(\tfrac{\pi (\mu -m)}{2}\bigr)\alpha_2(0)}{(s-m-1)^2}\\
&-\left(\frac{2\sin \bigl(\tfrac{\pi (s-m-1)}{2}\bigr)}{\pi (s-m-1)}\right) 
\frac{2\cos \bigl(\tfrac{\pi (\mu -m)}{2}\bigr)\alpha_2(0)}{s-m-1}
\Phi_{1}(\Delta_m;-s+1)
\end{split}
\end{align}
is holomorphic at $s=m+1$. 
Subtracting (\ref{eqn:pf_prepole5_004}) from (\ref{eqn:pf_prepole5_002}), 
we get a function 
\begin{align*}
&2\sin \bigl(\tfrac{\pi (s-m-1)}{2}\bigr)\Phi_{1}(\Delta_m;-s+1)\\
&\times \biggl(\Xi_+(\alpha_1;s)+(-1)^{m}\Xi_-(\alpha_1;s)
+\frac{2\cos \bigl(\tfrac{\pi (\mu -m)}{2}\bigr)\alpha_2(0)}{\pi (s-m-1)^2}
\biggr)+\frac{(\sI)^{m}m!\beta_2(m)}{s-m-1},
\end{align*}
which is holomorphic at $s=m+1$. 
Hence, by (\ref{eqn:prepole4_001}), we know that 
\begin{align*}
&\Xi_+(\alpha_1;s)+(-1)^{m}\Xi_-(\alpha_1;s)
+\frac{2\cos \bigl(\tfrac{\pi (\mu -m)}{2}\bigr)\alpha_2(0)}{\pi (s-m-1)^2}
-\frac{(\sI)^{m}m!\beta_2(m)}{\pi (s-m-1)}
\end{align*}
is holomorphic at $s=m+1$. This completes the proof of the statement (iii). 
\end{proof}

\section{Automorphic distributions and Maass forms}
\label{sec:Maass_forms}

\subsection{Moderate growth functions on $\widetilde{G}$}
\label{subsec:mg_ftn_G}

In this subsection, 
we give a proof of Proposition \ref{prop:Maass_subspace}. 
Let $\kappa \in \bC$. 
Let $C^\infty (\widetilde{G}/\widetilde{K};\kappa )$ be the subspace 
of $C^\infty (\widetilde{G})$ consisting of all functions $F$ such that 
\begin{align}
\label{eqn:def_CG_over_K}
&F(\tilde{g}\tilde{\rk}(\theta))=F(\tilde{g})e^{\sI \kappa \theta }&
&(\tilde{g}\in \widetilde{G},\ \theta \in \bR).
\end{align}
We define a $\bC$-linear map 
$C^\infty (\widetilde{G}/\widetilde{K};\kappa )\ni F\mapsto 
\phi_F\in C^\infty (\gH)$ by 
\begin{align}
\label{eqn:def_phiF}
&\phi_F (z)=F(\tilde{\ru}(x)\tilde{\ra}(y))&
&(z=x+\sI y\in \gH ). 
\end{align}
By the Iwasawa decomposition (\ref{eqn:Iwasawa_g}), 
we know that this map is bijective, and 
the inverse map $C^\infty (\gH)\ni \phi \mapsto F_\phi \in 
C^\infty (\widetilde{G}/\widetilde{K};\kappa )$ of this map is given by 
\begin{align}
\label{eqn:def_Fphi}
&F_\phi (\tilde{g})=\phi (g\sI)e^{\sI \kappa \theta}&
&(\tilde{g}=(g,\theta )\in \widetilde{G}). 
\end{align}

Let $\| \cdot \|$ be the euclidean norm on $M_2(\bR )\simeq \bR^4$, that is,  
\begin{align*}
&\|g\|=\sqrt{a^2+b^2+c^2+d^2}&
&\left(\ g=\left(\begin{array}{cc}
a&b\\c&d\end{array}\right)\in M_2(\bR )\ \right).
\end{align*}
Then it is well-known and easy to show that 
\begin{align}
\label{eqn:norm_multiply}
&\|g_1g_2\|\leq \|g_1\|\hspace{0.2mm}\|g_2\|&&(g_1,g_2\in M_2(\bR)),&\\
\label{eqn:norm_cpt_inv}
&\|kg\|=\|gk\|=\|g\|&&(k \in K,\ g\in M_2(\bR )).
\end{align}
Let {\normalfont [M1]} and {\normalfont [M2]} be the conditions 
for $\phi \in C^\infty (\gH)$ in \S \ref{subsec:poisson} and 
\S \ref{subsec:Maass_forms}, respectively. 
We consider the following condition 
{\normalfont [M3]} for $F\in C^\infty (\widetilde{G})$: 
\begin{description}
\item[{\normalfont [M3]}]
There are $c,r>0$ such that 
$|F({}^s\!g)|\leq c\|g\|^r$ ($g\in G$). 
\end{description}
Let $\Omega_\kappa $ be the hyperbolic Laplacian 
of weight $\kappa$ on $\gH$ defined by (\ref{eqn:def_Omega}). 
Let $\cC$ be the Casimir element 
\begin{align*}
\cC =\frac{1}{4}(H^2+2E_+E_-+2E_-E_+),
\end{align*}
which generates the center of $\cU(\g_\bC )$ 
as a $\bC$-algebra.

\begin{lem}
\label{lem:autom_GtoH_001}
Retain the notation. 
Let $F\in C^\infty (\widetilde{G}/\widetilde{K};\kappa )$. 
\vspace{1mm}

\noindent (i) We have 
$F({}^s\!g \tilde{\ru}(x)\tilde{\ra}(y))
=(\phi_F \big|_\kappa g)(z)$ for $g \in G$ and $z=x+\sI y\in \gH$.\\[1mm]
\noindent (ii) For $\nu \in \bC$, we have 
$\Omega_\kappa \phi_F 
=\left(\tfrac{1}{4}-\nu^2\right)\phi_F$ if and only if 
$\rho (\cC )F=\left(\nu^2-\tfrac{1}{4}\right)F$. \\[1mm]
\noindent (iii) 
The function 
$\phi =\phi_F$ satisfies {\normalfont [M1]} if and only if 
$F$ satisfies {\normalfont [M3]}.\\[1mm]
\noindent (iv) 
The function 
$\phi =\phi_F$ satisfies {\normalfont [M2]} if $F$ 
satisfies {\normalfont [M3]}.
\end{lem}
\begin{proof}
Our proof is based on the Iwasawa decomposition 
$\widetilde{G}=\widetilde{U}\widetilde{A}\widetilde{K}$. 
The statement (i) follows from the equality 
\[
{}^s\!g \tilde{\ru}(x)\tilde{\ra}(y)
=\tilde{\ru}(\mathrm{Re}(gz))
\tilde{\ra}(\mathrm{Im}(gz))\tilde{\rk}(-\arg J(g,z)) 
\]
for $g \in G$ and $z=x+\sI y\in \gH$. 
The statement (ii) follows from 
\begin{align*}
&(\rho (\cC )F)(\tilde{\ru}(x)\tilde{\ra}(y))
=-(\Omega_\kappa \phi_F)(z)&&(z=x+\sI y\in \gH) 
\end{align*}
and $\rho (\tilde{k})\circ \rho (\cC )
=\rho (\cC )\circ \rho (\tilde{k})$ $(k\in K)$. 
The statement (iii) follows from (\ref{eqn:norm_cpt_inv}) 
and the equality $\|{\ru}(x){\ra}(y)\|^2
=(|z|^2+1)/y$ for $z=x+\sI y\in \gH$. 
By (\ref{eqn:norm_multiply}), we have 
\begin{align*}
\|\gamma {\ru}(x){\ra}(y)\|
&\leq \|\gamma \| \,
\|{\ru}(x)\| \,
\|{\ra}(y)\|=\|\gamma \|\sqrt{2+x^2}\sqrt{y+y^{-1}}\\
&\leq \|\gamma \|\sqrt{2+\frac{1}{4}}\,\sqrt{y+\frac{4}{3}y}
=\frac{\|\gamma \|\sqrt{21}}{2}\,y^{1/2}
=\frac{\|\gamma \|\sqrt{21}}{2}\,\mathrm{Im}(z)^{1/2}
\end{align*}
for $\gamma \in SL(2,\bZ)$ and $z=x+\sI y \in \gD$. 
By this inequality and the statement (i), 
we obtain the statement (iv). 
\end{proof}

Let $\cM_\nu (\gH ;\kappa )$ be the subspace 
of $C^\infty (\gH )$ defined in \S \ref{subsec:poisson}. 
Let $\cM_\nu (\widetilde{G}/\widetilde{K};\kappa )$ be a subspace 
of $C^\infty (\widetilde{G}/\widetilde{K};\kappa )$ 
consisting of all functions $F$ satisfying {\normalfont [M3]} and 
$\rho (\cC ) F
=\left(\nu^2-\tfrac{1}{4}\right)F$. 
Then we obtain the following corollary of Lemma \ref{lem:autom_GtoH_001}. 
\begin{cor}
\label{cor:autom_G_to_H}
Let $\kappa \in \bC$. 
A bijective $\bC$-linear map from 
$\cM_{\nu }(\widetilde{G}/\widetilde{K};\kappa)$ to 
$\cM_{\nu }(\gH ;\kappa )$ is given by 
$F\mapsto \phi_F$, 
and its inverse map is given by $\phi \mapsto F_\phi $. 
\end{cor}

Let us prove Proposition \ref{prop:Maass_subspace}. 
We assume $\kappa \in \bR$ until the end of this subsection. 
Let $\Gamma $ be a cofinite subgroup of $SL(2,\bZ)$ 
such that $-1_2\in \Gamma$. 
Let $v$ be a multiplier system on $\Gamma$ 
of weight $\kappa $. 
Proposition \ref{prop:Maass_subspace} follows immediately from 
Lemma \ref{lem:autom_GtoH_001} 
and the following lemma.

\begin{lem}
\label{lem:autom_GtoH_002}
Retain the notation. 
Let $\phi \in C^\infty (\gH)$ satisfying {\normalfont [M2]} and the equality 
$\phi \big|_\kappa \gamma =v(\gamma )\phi$ \ $(\gamma \in \Gamma )$. 
Then $F=F_\phi$ satisfies {\normalfont [M3]}. 
\end{lem}
\begin{proof}
Let $\{\gamma_{i}\}_{i=1}^d$ be 
a complete system of representatives of $\Gamma \backslash SL(2,\bZ)$. 
By the condition {\normalfont [M2]}, there are $c_0,r_0>0$ such that 
\begin{align}
\label{eqn:pf_GtoH_lem2_001}
&|(\phi \big|_\kappa \gamma_i )(z)|\leq c_0y^{r_0}&
&(z=x+\sI y\in \gD,\ 1\leq i\leq d). 
\end{align}
Let $\tilde{g}\in \widetilde{G}$. 
Since $\gD$ is the closure of the fundamental domain of 
$SL(2,\bZ) \backslash \gH 
\simeq SL(2,\bZ) \backslash G/K$, we can take 
$\displaystyle \gamma =\left(\begin{array}{cc}a&b\\ c&d\end{array}\right)
\in SL(2,\bZ)$, $z_0=x_0+\sI y_0\in \gD$ and $\theta \in \bR$ 
so that $\tilde{g}={}^s\!\gamma \tilde{\ru}(x_0)
\tilde{\ra}(y_0)\tilde{\rk}(\theta )$. 
By Lemma \ref{lem:autom_GtoH_001} (i), we have 
\begin{align}
\label{eqn:pf_GtoH_lem2_002}
|F_\phi (\tilde{g})|=
|F_\phi ({}^s\!\gamma \tilde{\ru}(x_0)
\tilde{\ra}(y_0))e^{\sI \kappa \theta }|
=|(\phi \big|_\kappa \gamma )(z_0)|.
\end{align}
By (\ref{eqn:norm_cpt_inv}) and 
$\gamma \ru (x_0)\ra (y_0)
=\left(\begin{array}{cc}a\sqrt{y_0}&(ax_0+b)/\sqrt{y_0}\\ 
c\sqrt{y_0}&(cx_0+d)/\sqrt{y_0}\end{array}\right)$, we have 
\begin{align}
\label{eqn:pf_GtoH_lem2_003}
\|g\|=\|\gamma \ru (x_0)\ra (y_0)\|
\geq \sqrt{(a\sqrt{y_0})^2+(c\sqrt{y_0})^2}=\sqrt{(a^2+c^2)y_0}
\geq \sqrt{y_0}. 
\end{align}
We take $\gamma_0 \in \Gamma$ and $1\leq j\leq d$ 
so that $\gamma =\gamma_0\gamma_j$. 
Since $\phi \big|_\kappa \gamma_0 =v(\gamma_0 )\phi$, 
we have 
\begin{align}
\label{eqn:pf_GtoH_lem2_004}
&(\phi \big|_\kappa \gamma )(z)
=e^{\sI \kappa t}((\phi \big|_\kappa \gamma_0 )\big|_\kappa \gamma_j)(z)
=e^{\sI \kappa t}v(\gamma_0)(\phi \big|_\kappa \gamma_j)(z)&
&(z\in \gH )
\end{align}
with $t=\arg J(\gamma_0, \gamma_jz)+\arg J(\gamma_j,z)
-\arg J(\gamma ,z)$. 
By (\ref{eqn:pf_GtoH_lem2_001}), (\ref{eqn:pf_GtoH_lem2_002}), 
(\ref{eqn:pf_GtoH_lem2_003}) and (\ref{eqn:pf_GtoH_lem2_004}), 
we have 
\begin{align*}
&|F_\phi (\tilde{g})|
=|(\phi \big|_\kappa \gamma )(z_0)|
=|(\phi \big|_\kappa \gamma_j)(z_0)|
\leq c_0y_0^{r_0}
\leq c_0\|g\|^{2r_0}, 
\end{align*}
and this completes the proof.  
\end{proof}

\subsection{Representations of moderate growth} 
\label{subsec:ps_mod_growth}

In this subsection, we give a proof of Proposition \ref{prop:poisson} (i).

\begin{lem}
\label{lem:pf_mg_001}
Let $\mu,\nu \in \bC$. 
For $\tilde{g}=(g,\theta )\in \widetilde{G}$ and $F\in I_{\mu,\nu}$, we have 
\begin{align*}
&|F(\tilde{g})|\leq 
e^{-\mathrm{Im}(\mu )(\theta +\arg J(g,\sI))}
\|g\|^{|2\mathrm{Re}(\nu )+1|}|F|_K.
\end{align*}
\end{lem}
\begin{proof}
Let $\tilde{g}=(g,\theta )\in \widetilde{G}$ and 
set $x=\mathrm{Re}(g\sI)$ and $y=\mathrm{Im}(g\sI)$. 
By (\ref{eqn:Iwasawa_g}), we have 
$\tilde{g}=\tilde{\ru}(x)\tilde{\ra}(y)\tilde{\rk}(\theta)$. 
Hence, we have 
\begin{align*}
\|g\|=\|\ru (x)\ra (y)\|
=\sqrt{y+y^{-1}+x^2y^{-1}}
\geq \sqrt{y+y^{-1}}
\end{align*}
by (\ref{eqn:norm_cpt_inv}). 
The assertion follows immediately form this inequality 
and (\ref{eqn:ps_red_section}). 
\end{proof}

For $j\in \bZ_{\geq 0}$, let $\cU_j(\g_\bC)$ be 
the space defined in \S \ref{subsec:pf_PFprop}, 
that is, the subspace of $\cU(\g_\bC)$ spanned by the products of 
$j$ or less elements of $\g_{\bC}$.

\begin{lem}
\label{lem:pf_mg_002}
For $j\in \bZ_{\geq 0}$ and $X\in \cU_j(\g_{\bC})$, 
there are positive integer $m$, 
$Y_1,Y_2,\cdots ,Y_m\in \cU_j(\g_{\bC})$ and polynomial functions 
$p_i$ of degree at most $2j$ on $M_2(\bR )\simeq \bR^4$ 
$(1\leq i\leq m)$ such that 
\begin{align*}
&\rho (X)\circ \rho (\tilde{g})
=\sum_{i=1}^m
p_{i}(g)\rho (\tilde{g})\circ \rho (Y_i)
&(\tilde{g}=(g,\theta )\in \widetilde{G}).
\end{align*}
\end{lem}
\begin{proof}
Let $\tilde{g}=(g,\theta )\in G$ with 
$g=\left(\begin{array}{cc}
a&b\\ c&d
\end{array}\right)$. 
By Lemma \ref{lem:univ_exp1} (ii), we have 
\begin{align*}
&\rho (X)\circ \rho (\tilde{g})=
\rho (\tilde{g})\circ \rho (\Ad (g^{-1})X)&
&(X\in \cU (\g_{\bC})).
\end{align*}
By this equality and the formulas 
\begin{align*}
&\Ad (g^{-1})H=(ad+bc)H+2bdE_+-2acE_-,\\
&\Ad (g^{-1})E_+=cdH+d^2E_+-c^2E_-,\\
&\Ad (g^{-1})E_-=-abH-b^2E_++a^2E_-
\end{align*}
on the basis $\{H,E_+,E_-\}$ of $\g_{\bC}$, 
we obtain the assertion. 
\end{proof}

\begin{prop}
\label{prop:seminorm_mg}
Let $\mu,\nu \in \bC$. 
For $j\in \bZ_{\geq 0}$ and $X\in \cU_j(\g_{\bC})$, 
there are $c>0$, $m\in \bZ_{>0}$ and 
$Y_1, Y_2,\cdots ,Y_m\in \cU_j(\g_{\bC})$ such that 
\begin{align}
&\cQ_{X}(\rho ({}^s\!g)F)\leq 
c\|g\|^{|2\mathrm{Re}(\nu)+1|+2j}
\sum_{i=1}^m\cQ_{Y_i}(F)&
&(g\in G,\ F\in I^\infty_{\mu,\nu}). 
\end{align}
\end{prop}
\begin{proof}
By Lemma \ref{lem:pf_mg_001} and (\ref{eqn:norm_cpt_inv}), 
we have 
\begin{align*}
&|\rho ({}^s\!g)F({}^s\!k)|
=|F({}^s\!k\,{}^s\!g)|
\leq 
e^{3\pi |\mathrm{Im}(\mu )|}
\|g\|^{|2\mathrm{Re}(\nu )+1|}|F|_K&
&(k\in K,\ g\in G).
\end{align*}
This implies that $|\rho ({}^s\!g)F|_K\leq 
e^{3\pi |\mathrm{Im}(\mu )|}
\|g\|^{|2\mathrm{Re}(\nu )+1|}|F|_K$ $(g\in G)$. 
The assertion follows from this inequality and Lemma \ref{lem:pf_mg_002}. 
\end{proof}

Because of (\ref{eqn:cdn_conti_ps}), we obtain the following as 
a corollary of Proposition \ref{prop:seminorm_mg}. 
\begin{cor}
\label{cor:mg_conti_functional}
Let $\mu,\nu \in \bC$. 
For $\lambda \in I^{-\infty}_{\mu,\nu }$, 
there are $j\in \bZ_{\geq 0}$, $m\in \bZ_{>0}$, 
$X_1,X_2,\cdots ,X_m\in \cU_j(\g_{\bC})$ 
and $c>0$ such that 
\begin{align*}
&\lambda (\rho ({}^s\!g)F)\leq c\|g\|^{|2\mathrm{Re}(\nu)+1|+2j}
\sum_{i=1}^m\cQ_{X_i}(F)&
&(g \in G,\ F\in I^\infty_{\mu,\nu}).
\end{align*}
\end{cor}

Let us prove Proposition \ref{prop:poisson} (i). 
Let $\mu ,\nu \in \bC$ and $\kappa \in \mu +2\bZ$. 
For a distribution $\lambda $ on $I^{\infty}_{\mu,\nu}$, 
we define a function 
$\cP_{\nu,\kappa}^{\widetilde{G}}(\lambda)$ on $\widetilde{G}$ by 
\begin{align*}
&\cP_{\nu,\kappa}^{\widetilde{G}}(\lambda)(\tilde{g})
=e^{\pi \sI \kappa /2}
\lambda \bigl(\rho (\tilde{g})F_{\nu,\kappa}\bigr)&
&(\tilde{g}\in \widetilde{G}),
\end{align*}
where $F_{\nu,\kappa}$ is an element of $I_{\mu,\nu}^\infty$ defined by 
(\ref{eqn:def_F_nu_kappa}). Since 
\begin{align*}
&\cP_{\nu,\kappa}(\lambda)(z)
=\phi_{\cP_{\nu,\kappa}^{\widetilde{G}}(\lambda)}(z)
=\cP_{\nu,\kappa}^{\widetilde{G}}(\lambda)(\tilde{\ru}(x)\tilde{\ra}(y))&
&(z=x+\sI y\in \gH ),
\end{align*}
Proposition \ref{prop:poisson} (i) follows immediately from 
Lemma \ref{lem:autom_GtoH_001} (i), 
Corollary \ref{cor:autom_G_to_H} and 
Proposition \ref{prop:dist_to_form_G} below.

\begin{prop}
\label{prop:dist_to_form_G}
Let $\mu,\nu \in \bC$, and $\kappa \in \mu +2\bZ$. 
Let $\lambda $ be a distribution on $I^{\infty}_{\mu,\nu}$. 
Then we have $\cP_{\nu,\kappa}^{\widetilde{G}}(\lambda)
\in \cM_\nu (\widetilde{G}/\widetilde{K};\kappa )$. 
\end{prop}
\begin{proof}
By the definition (\ref{eqn:def_F_nu_kappa}) of $F_{\nu,\kappa}$, 
the function $F_{\nu,\kappa}$ is smooth and satisfies 
\begin{align}
\label{eqn:poisson2}
&\rho (\tilde{\rk}(\theta ))F_{\nu,\kappa }
=e^{\sI \kappa \theta }F_{\nu,\kappa }&
&(\theta \in \bR ).
\end{align}
This implies that 
$\cP_{\nu,\kappa}^{\widetilde{G}}(\lambda)
\in C^\infty (\widetilde{G}/\widetilde{K};\kappa )$. 
By direct computation, we have 
\begin{align*}
(\Omega_\kappa \phi_{F_{\nu,\kappa}})(z)
&=\left(-y^2\left(\frac{\partial^2}{\partial x^2}+
\frac{\partial^2}{\partial y^2}\right)
+\sI \kappa y\frac{\partial}{\partial x}\right)y^{\nu +\frac{1}{2}}
=\bigl(\tfrac{1}{4}-\nu^2\bigr) \phi_{F_{\nu,\kappa}}(z)
\end{align*}
with $\phi_{F_{\nu,\kappa}}(z)=F_{\nu,\kappa}(\tilde{\ru}(x)\tilde{\ra}(y))
=y^{\nu +\frac{1}{2}}$ and $z=x+\sI y\in \gH $. 
By Lemma \ref{lem:autom_GtoH_001} (ii), we have 
$\rho (\cC )F_{\nu,\kappa }=(\nu^2-\tfrac{1}{4})F_{\nu,\kappa }$. 
This implies 
\[
\rho (\cC )\cP_{\nu,\kappa}^{\widetilde{G}}(\lambda)
=(\nu^2-\tfrac{1}{4})\cP_{\nu,\kappa}^{\widetilde{G}}(\lambda). 
\]
By Corollary \ref{cor:mg_conti_functional}, 
the function $F=\cP_{\nu,\kappa}^{\widetilde{G}}(\lambda)$ 
satisfies {\normalfont [M3]}. 
\end{proof}

\subsection{Twists of automorphic distributions}
\label{subsec:proof_weil}

In this subsection, we give a proof of Theorem \ref{thm:Weil_converse}. 
Contrast to the previous paper \cite{preMSSU}, 
our method here is simplified and does not depend on 
the explicit expressions of multiplier systems.

Let $\mu\in \bR$, $\nu \in \bC$ and $N\in \bZ_{>1}$. 
Let $\kappa \in \mu +2\bZ$. 
Let $v$ be a multiplier system on $\Gamma_0(N)$ of weight $\kappa$. 
Let $L=u+\bZ$ and $\hat{L}=N^{-1}(\hat{u}+\bZ )$ with 
$0\leq u,\hat{u}<1$ determined by 
$v(\ru(1))=e^{2\pi \sI u}$ and 
$v(\overline{\ru}(-N))=e^{2\pi \sI \hat{u}}$. 
Let $\bP_N$ be a subset of the set of 
positive odd prime integers not dividing $N$, 
such that $\bP_N\cap \{am+b\mid m\in \bZ\}\neq \emptyset $ 
for any integers $a$, $b$ coprime to each other. 
Theorem \ref{thm:Weil_converse} follows immediately from 
Theorems \ref{thm:QAD_Fexp}, \ref{thm:QAD_DS}, 
Propositions \ref{prop:poisson} (ii), \ref{prop:dist_to_form}, 
and Lemmas \ref{lem:weil1}, \ref{lem:weil2} below.

\begin{lem}
\label{lem:weil1}
Retain the notation. \\[1mm]
\noindent (i) Let $\lambda \in 
(I_{\mu,\nu}^{-\infty})^{\widetilde{\Gamma}_0(N),\tilde{\chi}_v}$. 
Then we have 
\begin{align}
\label{eqn:cdn_IGammaN_01}
&\lambda (\rho ({}^s\!\gamma)F)=v(\gamma )\lambda (F)&
&\left(\gamma=\left(\begin{array}{cc}
a&b\\
Nc&d
\end{array}\right),\ F\in I_{\mu,\nu}^\infty \right) 
\end{align}
for any $a,b,c,d\in\bZ$ such that $d>0$ and $ad-Nbc=1$. \\[1mm]
\noindent (ii) Let $\lambda \in (I_{\mu,\nu}^{-\infty})_{L,\hat{L}}$, and 
assume that (\ref{eqn:cdn_IGammaN_01}) holds 
for any $a,b,c,d\in\bZ$ such that $d\in \bP_N$ and $ad-Nbc=1$. 
Then we have $\lambda \in 
(I_{\mu,\nu}^{-\infty})^{\widetilde{\Gamma}_0(N),\tilde{\chi}_v}$. 
\end{lem}
\begin{proof}
The statement (i) follows from the definition 
of $(I_{\mu,\nu}^{-\infty})^{\widetilde{\Gamma}_0(N),\tilde{\chi}_v}$. 
Since 
\begin{align*}
&\left(\begin{array}{cc}
a&b\\
Nc&d
\end{array}\right) \ru (m)
=\left(\begin{array}{cc}
a&am+b\\
Nc&Ncm+d
\end{array}\right)&
&(a,b,c,d,m\in \bZ ), 
\end{align*}
we note that, for any element $\tilde{\gamma}$ of $\widetilde{\Gamma}_0(N)$, 
there are some $a,b,c,d,m_1,m_2\in \bZ$ such that 
$d\in \bP_N$, $ad-Nbc=1$ and 
\begin{align*}
&\tilde{\gamma}=
\sideset{^s\!}{}{\mathop{\!\left(\begin{array}{cc}
a&b\\
Nc&d
\end{array}\right)}}
\tilde{\ru}(m_1)\,\tilde{\rk}(m_2\pi ). 
\end{align*}
Hence, the statement (ii) follows from (\ref{eqn:wt_compatibility}) and 
the definition of $(I_{\mu,\nu}^{-\infty})_{L,\hat{L}}$. 
\end{proof}

\begin{lem}
\label{lem:weil2}
Retain the notation, and let $d\in \bZ_{>0}$ coprime to $N$. 
Let $(\alpha ,\beta )\in \mathfrak{M}(L)_{\mu,\nu}^0\times 
{\mathfrak{N}}(S_{\nu}(L))$ 
and $(\hat{\alpha},\hat{\beta})\in \mathfrak{M}(\hat{L})_{\mu,\nu}^0\times 
{\mathfrak{N}}(S_{\nu}(\hat{L}))$ 
such that $(\lambda_{\alpha ,\beta })_\infty 
=\lambda_{\hat{\alpha},\hat{\beta}}$. 
For a Dirichlet character $\psi $ modulo $d$, 
we take $\alpha_{\psi}$, $\beta_\psi$, $\hat{\alpha}_{\psi,v}$ and 
$\hat{\beta}_{\psi,v}$ as in \S \ref{subsec:Auto_Dist_group}. 
Then $\lambda =\lambda_{\alpha ,\beta }$ satisfies 
(\ref{eqn:cdn_IGammaN_01}) for any $a,b,c\in\bZ$ such that $ad-Nbc=1$, 
if and only if, the equality 
\[
(\lambda_{\alpha_\psi,\beta_\psi})_\infty 
=\lambda_{\hat{\alpha}_{\psi,v},\hat{\beta}_{\psi,v}}
\]
holds for any Dirichlet character $\psi $ modulo $d$. 
\end{lem}
\begin{proof}
Let $a,b,c\in\bZ$ such that $ad-Nbc=1$, and 
we set $\displaystyle \gamma=\left(\begin{array}{cc}
a&b\\
Nc&d
\end{array}\right)$. Since  
${}^s\!\gamma 
=\tilde{\ru}(b/d)\tilde{\ra}(d^{-2})\tilde{\overline{\ru}}(Nc/d)$, 
by the substitution 
$\rho \bigl(\tilde{\ra}(d^{-2})\tilde{\overline{\ru}}(Nc/d)\bigr)F\to F$, 
the equality (\ref{eqn:cdn_IGammaN_01}) becomes 
\begin{align}
\label{eqn:cdn_IGammaN_02}
&\lambda (\rho (\tilde{\ru}(b/d))F)=
v(\gamma )\lambda \bigl(\rho \bigl(\tilde{\overline{\ru}}(-Nc/d)
\tilde{\ra}(d^{2})\bigr)F\bigr)&
&(F\in I_{\mu,\nu}^\infty ).
\end{align}
Multiplying the both sides by $e^{-2\pi \sI ub/d}$, 
and applying the equality 
$\lambda (F)=\lambda_\infty (F_\infty )$ ($F\in I_{\mu,\nu}^\infty$) with 
Lemma \ref{lem:NA_act_iota_infty} (ii) to the right-hand side, 
the equality (\ref{eqn:cdn_IGammaN_02}) becomes 
\begin{align}
\label{eqn:cdn_IGammaN_03}
\begin{split}
&e^{-2\pi \sI ub/d}\lambda (\rho (\tilde{\ru}(b/d))F)\\
&=v(\gamma )e^{-2\pi \sI ub/d}
\lambda_\infty \bigl(\rho \bigl(\tilde{\ru}(Nc/d)
\tilde{\ra}(d^{-2})\bigr)F_\infty \bigr)
\hspace{10mm}
(F\in I_{\mu,\nu}^\infty ).
\end{split}
\end{align}
If $\lambda =\lambda_{\alpha,\beta}$, then 
$\lambda_\infty=\lambda_{\hat{\alpha},\hat{\beta}}$ and 
the equality (\ref{eqn:cdn_IGammaN_03}) becomes 
\begin{align}
\label{eqn:cdn_IGammaN_04}
\begin{split}
&\sum_{l\in L }e^{2\pi \sI (l-u)b/d}\alpha (l)\cJ_l(F)
+\sum_{m\in S_{\nu}(L)}
e^{-2\pi \sI ub/d}\beta (m)\delta^{(m)}_{\infty}(F)\\
&=
\sum_{l\in \hat{L}}v(\gamma )e^{2\pi \sI (Nlc-ub)/d}d^{2\nu -1}
\hat{\alpha}(l)\cJ_{l/d^2}(F_\infty )\\[1mm]
&\hspace{3mm}
+\left\{\!\begin{array}{ll}
4\cos \bigl(\tfrac{\pi \mu }{2}\bigr)
v(\gamma )e^{-2\pi \sI ub/d}
\hat{\alpha}(0)
\delta^{(0)}_\infty (F_\infty)d^{-1}\log d
&\text{if $0\in \hat{L}$ and $\nu =0$},\\[1mm]
0&\text{otherwise}.
\end{array}\!\right.\\[1mm]
&\hspace{3mm}
+\sum_{m\in S_{\nu}(\hat{L})}
v(\gamma )e^{-2\pi \sI ub/d}\hat{\beta}(m)
d^{-2\nu -2m-1}\delta^{(m)}_{\infty}(F_\infty)
\hspace{17mm}
(F\in I_{\mu,\nu}^\infty )
\end{split}
\end{align}
by Propositions \ref{prop:Jacquet_property} and \ref{prop:act_A_Jac_int}. 
Multiplying by the respective sides of 
\[
\psi (b)=\overline{\psi (-Nc)}, 
\]
and taking the sum over all integers $b$ in 
a reduced residue system modulo $d$ (here $c$ ranges over 
a reduced residue system modulo $d$, since $ad-Nbc=1$), 
we have 
\begin{align}
\label{eqn:cdn_IGammaN_05}
\begin{split}
&\sum_{l\in L }\alpha_\psi (l)\cJ_l(F)
+\sum_{m\in S_{\nu}(L)}\beta_\psi (m)\delta^{(m)}_{\infty}(F)\\
&=
\sum_{l\in \hat{L}}
\hat{\alpha}_{\psi,v}(l)
\cJ_{l/d^2}(F_\infty )
+\sum_{m\in S_{\nu}(\hat{L})}
\hat{\beta}_{\psi ,v}(m)\delta^{(m)}_{\infty}(F_\infty )
\hspace{7mm}
(F\in I_{\mu,\nu}^\infty )
\end{split}
\end{align}
for a Dirichlet character $\psi $ modulo $d$. 
Conversely, multiplying by the respective sides of 
\[
\frac{1}{\# (\bZ /d\bZ )^\times }\overline{\psi (b)}
=\frac{1}{\# (\bZ /d\bZ )^\times }\psi (-Nc), 
\]
and taking the sum over all Dirichlet character $\psi $ modulo $d$, 
we obtain (\ref{eqn:cdn_IGammaN_04}) from (\ref{eqn:cdn_IGammaN_05}). 
Here $\# (\bZ /d\bZ )^\times $ is the cardinality of $(\bZ /d\bZ )^\times$, 
and we use the orthogonality relation 
\begin{align*}
&\frac{1}{\# (\bZ /d\bZ )^\times }
\sum_{\psi}\overline{\psi (m_1)}\psi (m_2)
=\left\{\begin{array}{ll}
1&\text{if $m_1\equiv m_2\mod d$},\\
0&\text{otherwise}
\end{array}\right.&
&(m_1,m_2\in \bZ),
\end{align*}
where the sum $\sum_{\psi}$ is 
over all Dirichlet character $\psi $ modulo $d$. 
Therefore, we know that 
$\lambda =\lambda_{\alpha,\beta}$ satisfies 
(\ref{eqn:cdn_IGammaN_01}) 
for any $a,b,c\in\bZ$ such that $ad-Nbc=1$, 
if and only if, (\ref{eqn:cdn_IGammaN_05}) holds 
for any Dirichlet character $\psi $ modulo $d$. 
\end{proof}


\def\cprime{$'$}


\end{document}